\documentclass[11pt] {article}
\usepackage{amsmath,amssymb,amsfonts}
\usepackage[utf8]{inputenc} 
\usepackage{amsthm}

\usepackage{amsmath}

\usepackage{tikz}
\usepackage{hyperref}
\usepackage{graphicx}
\usepackage{amsfonts}

\usepackage{hyperref}
\hypersetup{hidelinks}

\usepackage{color}
\usepackage[english]{babel}
\usepackage{ifthen}
\usepackage{graphicx,nicefrac}
\usepackage{tikz}
\usetikzlibrary{decorations.markings}
\usetikzlibrary{plotmarks}
\usetikzlibrary{patterns}
\usepackage{pgfplots}
\usepackage{makeidx}
\usetikzlibrary{calc}

\usepackage{geometry} 
\usepackage{graphicx} 

\usepackage{array} 
\usepackage{paralist} 
\usepackage{verbatim} 
\usepackage{subfig} 

\usepackage{textpos}

\usepackage[backend=bibtex,style=numeric]{biblatex}           
\tikzset { domaine/.style 2 args={domain=#1:#2} }

\newcommand{\Odd}{\mathrm{odd}}
\newcommand{\Ev}{\mathrm{even}}

  \def\<{\langle}
  \def\>{\rangle}
  
  \DeclareMathOperator{\Ker}{Ker}
  \DeclareMathOperator{\Img}{Im}

\newtheorem{theorem}{Theorem}

\newtheorem{theo}{Theorem}[section]
\newtheorem*{theo*}{Theorem}
\newtheorem{conj}{Conjecture}[section]
\newtheorem{conj'}{Conjecture'}[section]
\newtheorem{prop}[theo]{Proposition}
\newtheorem*{prop*}{Proposition}
\newtheorem{lem}{Lemma}[section]
\newtheorem{cor}{Corollary}[section]
\newtheorem*{cor*}{Corollary}
\newtheorem{defn}{Definition}[section]

\newtheorem{claim}{Claim}

\newtheorem{remark}{Remark}[section]

\newcommand{\nocontentsline}[3]{}
\newcommand{\tocless}[2]{\bgroup\let\addcontentsline=\nocontentsline#1{#2}\egroup}

\numberwithin{equation}{section}

\title{Sobolev regularity of compactified 3-manifolds and the ADM Center of Mass}
\author{Rodrigo Avalos \\ Email: \href{rdravalos@gmail.com}{rdravalos@gmail.com} \footnote{Institut für Mathematik, Universität Rostock, Ulmenstraße 69, Haus 3, 18057 Rostock, Germany.}}
\date{}

\begin{document}
\maketitle

\begin{abstract}
In this paper we address the existence of preferred asymptotic coordinates on asymptotically Euclidean (AE) manifolds $(M^3,g)$ such that $g$ admits an asymptotically Schwarzschildian first order expansion, based purely on a priori geometric conditions, which will then be used to establish geometric criteria guaranteeing the convergence of the ADM center of mass (COM). This problem is analysed relating it to the study of the regularity of conformal compactifications of such manifolds, which is itself explored via elliptic theory for operators with coefficients of very limited regularity. With these related problems in mind, we first establish Sobolev regularity properties of conformally compactified AE 3-manifolds via the decay of the Cotton tensor, improving on previous results. This allows us to construct preferred asymptotic coordinates on such AE manifolds where the metric has a first order Schwarzschildian expansion, which in turn allows us to address a version of a conjecture posed by C. Cederbaum and A. Sakovich concerning the convergence of the COM of such manifolds.
\end{abstract}



\medskip
\noindent\textbf{Acknowledgements:} The author would like to thank Jan Metzger, Carla Cederbaum, Paul Laurain, Albachiara Cogo and Andoni Royo Abrego for helpful discussions and references related to this paper, as well as reading preliminary versions of it. Also, the author would like to thank the Alexander von Humboldt Foundation for financial support during the writing of this paper. Finally, the author would like to thank an anonymous referee for a careful reading of the first version of this paper, including suggestions which permitted to improve it and broaden the scope of some results.

\tableofcontents


\section{Introduction}

In this paper we address a few related problems concerning regularity of geometric structures. In order to motivate this analysis, let us start with two problems well-known within geometric analysis. The first one of these is related to the ADM center of mass (COM) associated with a $3$-dimensional asymptotically Euclidean (AE) Riemannian manifold $(M^3,g)$. Let us briefly recall that an AE manifold is a non-compact manifold which (in the case of one end) outside a compact set $K\subset M^n$ is diffeomorphic to the exterior of a closed ball in $\mathbb{R}^n$, that is $M^n\backslash K\cong \mathbb{R}^n\backslash \overline{B_{1}(0)}$, and the corresponding diffeomorphism induces a set of asymptotic coordinates, in which the Riemannian metric $g$ approaches the Euclidean metric as we move towards infinity. There are different ways to measure this decay, and while full details are provided within Section \ref{AESection}, let us now say that an AE manifold is of order $\tau>0$ if
\begin{align}
g_{ij}(x)=\delta_{ij}+O_k(|x|^{-\tau}),
\end{align} 
on $\mathbb{R}^n\backslash \overline{B_{1}(0)}$ for some $k\geq 0$.\footnote{In this context, a function $f$ is said to be of class $O_k(|x|^{-\tau})$ if $\partial^{\alpha}f=O(|x|^{-\tau-|\alpha|})$ for any multi-index $\alpha$ such that $|\alpha|\leq k$.} In this context, in the case of an AE 3-manifold with one end and asymptotic coordinates $\{x^i\}_{i=1}^3$, the ADM energy is defined as\footnote{The Einstein summation convention along repeated indices will be assumed through the paper.}
\begin{align}\label{ADMenergyInto}
\begin{split}
E\doteq \frac{1}{16\pi}\lim_{r\rightarrow\infty}\int_{S_r}\left(\partial_ig_{ij}-\partial_jg_{ii} \right)\nu^jd\omega_r \; \; ; \; \;
\end{split}
\end{align}
whenever the limit exists, and where $S_r\hookrightarrow \mathbb{R}^3\backslash\overline{B_{R_0}(0)}$ denotes a topological sphere of radius $r>R_0$ contained within the end of $M^3$, while $\nu$ denotes the outward-pointing Euclidean unit normal to it and $d\omega_r$ the volume form on $S_r$ induced by the Euclidean  metric. Let us highlight that precise geometric conditions can be formulated to guarantee the convergence of (\ref{ADMenergyInto}) as well as the independence of this limit on the asymptotic coordinates (see, for instance, \cite{BartnikMass}). Also, the Riemannian positive mass theorem implies that $E\geq 0$ for any 3-dimensional AE manifold of order $\tau>\frac{1}{2}$ with non-negative and integrable scalar curvature, and $E=0$ iff $(M^3,g)\cong \mathbb{E}^3$ (See \cite{SchoenYauPM1,WittenPMT,BartnikMass}).

In the above context, the COM of a 3-dimensional AE manifold $(M^3,g)$ is given by a vector $C_{B\acute{O}M}=(C^1,C^2,C^3)\in \mathbb{R}^3$, whose components are defined by the limits
\begin{align}
C^k_{B\acute{O}M}&\doteq\frac{1}{16\pi E} \lim_{r\rightarrow\infty}\left(\int_{S_r}x^k\left(\partial_ig_{ij}-\partial_jg_{ii} \right)\nu^jd\omega_r - \int_{S_r}\left(g_{ik}\nu^i - g_{ii}\nu^k \right)d\omega_r \right)\label{COMIntro},
\end{align}
whenever the limits exist, and where we have appealed to the same notations introduced above.\footnote{More detailed definitions concerning AE manifolds and ADM asymptotic charges are provided within Section \ref{AESection}.} Just as with other asymptotic charges associated with isolated gravitational systems, the COM has proven to be related to highly interesting problems in geometry, most notably the existence of geometric foliations of infinity in AE manifolds and the convergence of the center of such foliations, which has been proposed as a geometric way of defining the COM. A non-exhaustive list of related important references is given by \cite{HY,Metzger1,EichmaierMetzger,Nerz1,Nerz2,CederbaumSakovich,EichmairWillmore}, which treat constant mean curvature foliations, constant expansion foliations, constant space-time mean curvature foliations and foliations by Willmore surfaces. Nevertheless, in particular for the convergence of the COM and the center of these associated foliations, one typically needs to impose very precise asymptotic behaviour for the metric $g$. Namely, certain asymptotic parity conditions need to be imposed, which in the literature are known as the Regge-Teitelboim (RT) conditions. In their simplest form, such conditions are given by the assumption that there are asymptotic coordinates $\{x^i\}_{i=1}^3$ for $M$ where the AE metric $g$ can be written as
\begin{align}\label{RTcondsIntro}
g_{ij}(x)=\delta_{ij} + O_2(|x|^{-\tau}), \:\: g^{\Odd}_{ij}(x)=O_2(|x|^{-\tau-1}), \:\: R_g^{\Odd}(x)=O(|x|^{-\tau-3}),
\end{align}  
for some $\tau>\frac{1}{2}$, where the superscript ``$\Odd$" refers to the odd-part of the given function, and $R_g$ denotes the scalar curvature of the metric $g$. Actually, in many cases a \emph{Schwarzschildian} expansion near infinity is demanded, which implies that the above expansion is actually replaced by
\begin{align}\label{SchwarzExpansionIntro}
g_{ij}=\left(1 + \frac{A}{|x|}\right)\delta_{ij} + O_2(|x|^{-2}),
\end{align}
where $A$ is a constant, related to the ADM energy of $g$. Clearly, the Schwarzschildian case implies a strong version of the RT parity conditions. In this context, a generally recognised problem is that the existence of such an a priori expansion with these parity properties is not know to be a consequence of some set of geometric hypotheses, but is typically merely assumed to exist. Although it is true that in sufficiently weak topologies metrics obeying the RT conditions are known to form a dense subset of the space of solutions to the Einstein constraint equations (ECE) of general relativity (GR) due to the work in \cite{CorvinoSchoen}, these topologies are not strong enough to guarantee convergence of the associated COM. It is therefore an interesting open question to determine geometric conditions which guarantee the existence of an asymptotic coordinate system where an AE manifold $(M^3,g)$ is guaranteed to have an expansion of the form of (\ref{RTcondsIntro}). This question relates to a conjecture proposed by C. Cederbaum and A. Sakovich in \cite[Conjecture 1]{CederbaumSakovich} which concerns the convergence of a suitably redefined version of the COM under a curvature decay type condition,\footnote{See Conjecture \ref{CSConjecture} below for further details.} which we take as one of the motivations for this paper. Let us highlight that this can also be motivated as a question within GR, since the expansion of $g$ near infinity is associated with the decay of the gravitational field at infinity for an isolated system. The known explicit solutions modelling such idealised systems obey expansions like (\ref{SchwarzExpansionIntro}), and from the weak Newtonian limit this kind of decay could be expected more generically. Nevertheless, proving clear-cut geometric conditions which a priori guarantee such an expansion seems to be quite subtle.

Let us notice that one may construct AE manifolds satisfying (\ref{SchwarzExpansionIntro}) in a very geometric way via decompactification of certain closed Riemannian manifolds. That is, let $(\hat{M}^3,\hat{g})$ be a smooth closed Riemannian manifold, and assume that $g$ is Yamabe positive. Then, given a point $p\in \hat{M}$, it is well-known that there exists a conformally related metric $g$ such that $(M^3\doteq \hat{M}^3\backslash\{p\},g)$ is AE (with $p$ being the point at infinity) and $g$ has an expansion of the form of (\ref{SchwarzExpansionIntro}) near infinity due to results in \cite{LeeParker}. This classic result is obtained since the Yamabe condition guarantees the existence of a positive Green function $G_{L_{\tilde{g}}}$ for the conformal Laplacian $L_{\tilde{g}}=-8\Delta_{\tilde{g}}+R_{\tilde{g}}$ of each element in $[\hat{g}]$, and then one can look for a suitable element within $[\hat{g}]$ where $G_{L_{\tilde{g}}}$ has a very explicit expansion in normal coordinates near its singularity at the selected point $p$. Then, the inversion of coordinates $z=\frac{x}{|x|^2}$ around $p$ induces an asymptotic chart for $M^3$ where (\ref{SchwarzExpansionIntro}) is valid in the $z$-coordinate system. Furthermore, by construction $R_g\equiv 0$ and thus $(M^3,g)$ can be regarded as a time-symmetric vacuum initial data set for the evolution problem in GR. This procedure provides an interesting connection between special classes of closed Riemannian manifolds and AE manifolds obeying the RT conditions. Notice, in particular, that in the case of AE manifolds obtained in this manner, the conditions which guaranteed the existence of an expansion like (\ref{SchwarzExpansionIntro}) were strictly geometric, and no reference to special coordinates with parity conditions needed to be invoked. Generalising this kind of construction to more general AE manifolds will be one of our main motivations.

Concerning the above description on how to obtain AE manifolds which naturally admit asymptotic coordinates where the RT conditions are satisfied, we would like to focus on the relation between an AE manifold and (in the case of manifolds with one end) its one-point compactification. Ideally, one would like to be able to start with an AE manifold and guarantee that under a certain set of geometric criteria there is an asymptotic chart where the RT conditions are satisfied, instead of starting already from a closed manifold and decompactifying it. One strategy one may adopt could be to start with a given AE manifold $(M^3,g)$, then attempt to conformally compactify it into a closed Riemannian manifold $(\hat{M}^3,\hat{g})$, then move within the conformal class $[\hat{g}]$ so as to find a suitable element in the class whose conformal decompactification $(M^3,\gamma)$ has an explicit asymptotic expansion, say like (\ref{SchwarzExpansionIntro}), and finally estimate the conformal factor between $g$ and $\gamma$ so as to extract a similar expansion for $g$. 

The above strategy brings us to the second geometric problem motivating this paper, since it is immediately faced with very interesting regularity issues, which have been highlighted in  \cite{Herzlich1997} and \cite{MaxwellDiltsYamabeAE}, where it has been pointed out that the compactification of a smooth AE manifold will, in general, be only of very limited regularity around the point of compactification. In \cite{Herzlich1997} the author has studied how additional control on conformal objects can be used to improve the regularity of the compactified manifold. In particular, for a $3$-dimensional AE manifold $(M^3,g)$, the decay of the Cotton tensor $C_g$ is tightly related to the regularity of its one-point conformal compactification $(\hat{M}^3,\hat{g})$. Conditions to obtain $C^{2}$-regularity for $\hat{g}$ were studied in detail in \cite{Herzlich1997}, while in \cite{MaxwellDiltsYamabeAE} the authors guarantee that, a priori, $\hat{g}$ can be shown to have $W^{2,p}(\hat{M})$ Sobolev regularity, with $p>\frac{3}{2}$, under minimal conditions on $g$ and with no extra requirements on $C_g$. Therefore, in order to pursue the existence of RT expansions following the reasoning described in the above paragraph, one first needs to address regularity issues concerning conformal compactifications. In particular, we aim to obtain $\hat{g}\in W^{2,p}(\hat{M})$ with $p>3$, but without imposing a priori decays which might be too strong to be interesting. For instance, such regularity is guaranteed from \cite{MaxwellDiltsYamabeAE} if one starts with a smooth AE manifold of order $\tau>1$, but this decay implies that the ADM energy of $g$ is zero, and therefore no AE metric with non-negative scalar curvature can obey these conditions due to the positive mass theorem. Therefore, we would like to pursue the validity of the strategy described in the above paragraph, but for decay rates $\tau\in (\frac{1}{2},1]$. 

Having the above paragraph in mind, to impose geometric conditions as weak as possible on $(M^3,g)$, we would like to pursue Sobolev-type controls on its compactification, interpolating more subtly in between $C^{k}$-spaces, and furthermore demanding controls on decays in an integral sense, rather than pointwise. In some sense, we attempt to interpolate in between the results of \cite{MaxwellDiltsYamabeAE} and \cite{Herzlich1997}. The strategy to do so is to appeal to a sequence of bootstraps. First, one may appeal to the general results of \cite{MaxwellDiltsYamabeAE} to obtain, a priori, a conformal compactification $(\hat{M}^3,\bar{g})$, with $\bar{g}\in W^{2,p}(\hat{M})$ and $2<p<3$. Then, several results on the Yamabe problem for low regularity metrics can be used to select an element $\hat{g}\in [\bar{g}]$ such that its scalar curvature has improved regularity, given by $R_{\hat{g}}\in W^{2,p}(\hat{M})$. The next goal is to improve the regularity of the Ricci tensor. For that, we appeal to the conformal invariance of the Cotton tensor $C_{\hat{g}}$ in three dimensions. In particular, an a priori weighted $L^p$-control of $C_{g}$ translates into an a priori $L^p$ control of $C_{\hat{g}}$, in turn translating into $W^{-1,p}(\hat{M})$-control for $\mathrm{div}_{\hat{g}}C_{\hat{g}}$. At this stage, the following geometric identity provides an elliptic PDE for $\mathrm{Ric}_{\hat{g}}$, which is central to proceed further in this bootstrap analysis:\footnote{See (\ref{LapRicciEstimate}) for further details.}
\begin{align}\label{CottonEqIntro}
\mathrm{div}_{\hat{g}}C_{\hat{g}}=\Delta_{\hat{g}}\mathrm{Ric}_{\hat{g}} + \mathrm{Riem}_{\hat{g}}\cdot \mathrm{Ric}_{\hat{g}} - \mathrm{Ric}_{\hat{g}}\cdot \mathrm{Ric}_{\hat{g}} -\frac{1}{4}(\hat{\nabla}^2R_{\hat{g}} + \Delta_{\hat{g}}R_g \hat{g}),
\end{align}
where above we have used the notation $\mathrm{Riem}_{\hat{g}}\cdot \mathrm{Ric}_{\hat{g}}$ to denote a specific contraction between the two highlighted tensors. Following the ideas described above, although one would a priori have $\mathrm{Ric}_{\hat{g}}\in L^p(\hat{M})$, $2<p<3$, (\ref{CottonEqIntro}) implies that $\Delta_{\hat{g}}\mathrm{Ric}_{\hat{g}}\in W^{-1,q}(\hat{M})$ for some $q>\frac{3}{2}$ (see the Lemma \ref{ApriopriCompactifiedRegularity}). This makes $\mathrm{Ric}_{\hat{g}}$ an $L^p$-solution of an elliptic equation with rough coefficients, where the remaining improvements of regularity can be obtained through rather refined elliptic regularity theory. In fact, regularity questions associated with a bootstrap from an a priori $L^r$ solution $u$ (with $r$ sufficiently large) for an equation of the form $\Delta_gu\in W^{-1,q}$ (or even in $L^p$) for metrics $g\in W^{2,p}$ and $\frac{n}{2}<p<n$, seem to be very close to critical cases in regularity theory. We would like to highlight that this type of regularity statements are outside of the scope of classical references, such as \cite{GilbargTrudinger,TaylorToolsForPDEs,Hormander2,Hormander3}. Notice that when treating equations with coefficients of limited regularity, it is typical to impose a priori $W^{1,p}$ controls on the a priori solution,\footnote{See, for instance, \cite[Chapter 8]{GilbargTrudinger}, \cite[Chapter 3]{TaylorToolsForPDEs} and \cite[Chapter 5]{Morrey}.} and we need to start with weaker solutions than that. In a broader context, a related problem is the conjecture posed by J. Serrin \cite{Serrin} in the case of weak solutions of scalar elliptic equations of the form
\begin{align}\label{SerrinEq}
\mathrm{div}(a^{ij}\nabla_i u)=0
\end{align}
on a bounded regular domain $\Omega\subset \mathbb{R}^n$ where $a^{ij}$ are bounded and measurable coefficients. These equations have well-established regularity properties for weak $W^{1,2}_{loc}$-solutions, and in \cite{Serrin} the author provided an example of a weak solution $W_{loc}^{1,p}$ to (\ref{SerrinEq}) with $1<p<2$, such that $u\not\in L_{loc}^{\infty}$, showing that the $W^{1,2}_{loc}$ requirement was necessary to obtain $L^{\infty}_{loc}$ control for the solution. The author then conjectured that if the coefficients $a^{ij}\in C^{0,\alpha}(\Omega)$, then any $W^{1,1}_{loc}$-solution to (\ref{SerrinEq}) should be $W^{1,2}_{loc}$. This conjecture was solved by H. Brezis in \cite{Brezis1}, actually showing a stronger statement. That is, it was shown that under a Dini condition on $a^{ij}$, a weak solution of (\ref{SerrinEq}) which is a priori of bounded variation must actually be in $W^{1,2}_{loc}$. Since this result, refinements contemplating different kinds of Sobolev regularity on the coefficients and the a priori solution have been established, for instance in \cite{Regularity1,Regularity2}, and also counterexamples to borderline cases have been presented \cite{JIN}. Concerning these results, let us comment that in \cite{Regularity2} the authors analyse (\ref{SerrinEq}) and establish that if $a^{ij}\in C^{0,1}_{loc}$ and $u\in L^1_{loc}$, then $u\in W^{2,p}_{loc}$ for any $p<\infty$, while in \cite{Regularity1} it is established that if $a^{ij}\in W^{1,n}(\Omega)$ also satisfies a Dini-type condition and $u\in L^{n'}_{loc}$, then $u\in W^{1,2}_{loc}$.

Coming back to the regularity analysis needed for (\ref{CottonEqIntro}), this relates to bootstrapping an a priori $L^{q'}$-solution to the tensor equation
\begin{align}\label{SerrinEq2}
\Delta_gu=f
\end{align} 
for a $W^{2,q}(M)$ metric on a closed manifold $M$, with $q>\frac{n}{2}$, $u$ a symmetric $(0,2)$-tensor field, and $f\in W^{-1,p}(M)$ for appropriate values of $p\in (1,\infty)$. This is a related problem to the one described above for (\ref{SerrinEq}), although it is out of the scope of the referenced papers. The geometric problem put forward above for the compactification of an AE manifold naturally poses this regularity question as the requirement to bootstrap $\mathrm{Ric}_{\hat{g}}$ to $L^p(\hat{M}^3)$ for some $p>3$. Once this is established, we can appeal to harmonic coordinates $\{y^i\}_{i=1}^3$ around the point of compactification to bootstrap this extra regularity gained for the Ricci tensor into the metric and obtain $\hat{g}\in W^{2,p}(\hat{M})$, with $p>3$, which was our target. This last process is also filled with subtle issues, since the harmonic coordinates associated with $\hat{g}$ will be a priori only $C^{1,\alpha}$-compatible with the differentiable structure induced by the inversion of coordinates $x^i=\frac{z^i}{|z|^2}$. Nevertheless, these extra subtle issues can be overcome.

Let us highlight how the above problems became naturally interrelated. We started with an open question in geometry and mathematical GR, concerning the existence of geometric conditions such that an AE 3-manifold admits asymptotic coordinates so that the metric has an expansion satisfying (\ref{RTcondsIntro}). We then noticed the relation between this problem and the problem of the regularity of the compactification of such AE manifolds, and finally we pointed out that this last problem poses a relevant question within elliptic regularity theory, related to recent developments in the area. Let us then notice that, after solving the last two of these problems, there is still work to be done in order to solve the questions associated with the COM and expansions like (\ref{RTcondsIntro}). In particular, one then needs to conformally decompactify $(\hat{M}^3,\hat{g})$ into $(M^3,\gamma)$, with $\gamma\in [g]$, and analyse these questions. The main idea is that if $\hat{g}$ is shown to be $W^{2,p}(\hat{M})$ with $p>3$, one may then use normal-type coordinates $\{\bar{y}^i\}_{i=1}^3$ around the point of compactification $p_{\infty}$ to obtain an expansion of the form:
\begin{align}\label{IntroMetricExpansion.1}
\hat{g}_{ij}(\bar{y})=\delta_{ij}+O_1(|\bar{y}|^{1+\alpha}), \text{ for some } \alpha>0.
\end{align} 
Then, we may decompactify following the ideas of \cite{MaxwellDiltsYamabeAE} and inducing a structure of infinity around $p_{\infty}$ via the inversion $\bar{z}=\frac{\bar{y}}{|\bar{y}|^2}$. Along this process, we will need to keep track of the relation between the coordinates $\{\bar{z}^i\}_{i=1}^3$ and the original asymptotic coordinates $\{z^i\}_{i=1}^3$ on $M$, at least up to sufficiently high order, which will depend on similar questions concerning the intermediary coordinate systems on $\hat{M}$ around $p_{\infty}$, given by the inversion $x=\frac{z}{|z|^2}$, then the change of coordinate to harmonic coordinates $\{y^i\}_{i=1}^3$, and finally the change to the normal coordinates $\{\bar{y}^i\}_{i=1}^3$ described above. Through (\ref{IntroMetricExpansion.1}) this process will allow us to estimate
\begin{align*}
\gamma_{ij}(\bar{z})=\delta_{ij}+O_1(|\bar{z}|^{-1-\alpha})
\end{align*} 
and, writing $g=u^4\gamma$, also grant that $u-1$ belongs to some weighted $L^p_{-\epsilon}(M,\Phi_{\bar{z}})$-space with $\epsilon>0$, where $\Phi_{\bar{z}}$ denotes the structure of infinity induced by the $\bar{z}$-coordinates. Then, the conformal covariance of the scalar curvature, together with a priori controls on $u-1$, $R_g$ and $R_{\gamma}$ give us an expansion
\begin{align}\label{IntroMetricExpansion.2}
u=1+\frac{C}{|\bar{z}|}+O_1(|\bar{z}|^{-1-\alpha}), \text{ for some } \alpha>0.
\end{align}
All this together, guarantees the existence of a structure of infinity with coordinates $\{\bar{z}^i\}^3_{i=1}$, such that
\begin{align}\label{IntroMetricExpansion.2}
\begin{split}
g(\partial_{\bar{z}^i},\partial_{\bar{z}^j})&=\left(1+\frac{4C}{|\bar{z}|}\right)\delta_{ij} + O_1(|\bar{z}|^{-1-\alpha}),\\
\bar{z}(z)-\mathrm{Id}(z)&\in C^1_{1-\alpha}(\mathbb{R}^n\backslash\overline{B_{R_0}(0)}).
\end{split}
\end{align}

Notice the above expansion will settle the question initially posed concerning (\ref{RTcondsIntro}), at least for expansions up to first order. We should highlight that higher order expansions could be extracted requiring further decay for the derivatives of the Cotton tensor $C_g$. In the above case, we need only a weighted $L^p$ condition on $C_g$. Moreover, from the above expansion one may read quite explicitly the ADM energy, and, under an $L^p$-integrability condition on $z^iR_g$, the above expansion is enough to guarantee the convergence of the COM in the coordinates $\{\bar{z}^i\}^3_{i=1}$, which settles a Riemannian and $L^p$-version of the conjecture posed by Cederbaum-Sakovich in \cite[Conjecture 1]{CederbaumSakovich}, under an additional control on $C_g$. We would like to highlight that the proof of this version of the conjecture in \cite{CederbaumSakovich} presents a relevant advance in the geometric understanding of the COM of AE 3-manifolds, since, besides the case of asymptotically conformally flat manifolds treated in \cite{CorvinoCOM}, to the best of our knowledge, it is the first time such convergence condition is formulated  purely in geometric terms, without invoking RT parity conditions a priori. 

\subsection{Main results}

In this section we present the main results obtained in this paper. Based on the topics discussed above and their relation, let us start with the regularity results for compactified AE 3-manifolds, which are understood to have only 1-end. The main result within this analysis is the following:

\begin{theorem}\label{ThmCIntro}
Let $(M^3,g)$ be a smooth $W^{k,p}_{\tau}$-AE manifold, relative to a structure of infinity $\Phi_z$ with coordinates $\{z^i\}_{i=1}^3$ and $p> 2$, $\tau\in (-1,-\frac{1}{2})$, $k\geq 2$. If $C_g\in L^{p_1}_{\sigma}(M,\Phi_z)$ with $-6<\sigma<-4$ and $p_1=\frac{3}{6+\sigma}$, then $(M^3,g)$ can be conformally compactified into $(\hat{M},\hat{g})$, where $\hat{M}$ stands for the 1-point compactification of $M$, and $\hat{M}$ can be equipped with a preferred differentiable structure $\mathcal{D}_{\mathrm{Har}}(\hat{M})$ which is $W^{2,q}$-compatible with the differentiable structure provided by the inverted coordinates $x=\frac{z}{|z|^2}$,\footnote{Actually, by Corollary \ref{HarmoniCoordCorollary}, we actually get $W^{3,q_0}$-compatibility between the two differentiable structures.} such that $\hat{g}\in W^{2,q}(\mathcal{D}_{\mathrm{Har}}(\hat{M}))$ for some $q>3$. In particular $\hat{g}\in C^{1,\alpha}(\mathcal{D}_{\mathrm{Har}}(\hat{M}))$, for some $\alpha\in (0,1)$.
\end{theorem}

As stated in the introduction, the above theorem should be compared with the results of \cite{MaxwellDiltsYamabeAE} and \cite{Herzlich1997}. In comparison with \cite[Lemma 5.2]{MaxwellDiltsYamabeAE}, the above theorem obtains improved regularity of the metric $\hat{g}$ under an additional weighted $L^p$-control on the Cotton tensor of the AE metric $g$. In the case of \cite{Herzlich1997}, since we are concerned with $C^{1,\alpha}$-regularity we should compare with \cite[Theorem A]{Herzlich1997}, in which case the above theorem demands weaker assumptions on the decay of the Cotton tensor (see Remark \ref{RemarkComparisonHerzlich}) and obtains intermediate Sobolev control for the compactified metric. More explicitly, in \cite[Theorem A]{Herzlich1997} a pointwise control of the form $C_g=O(|z|^{-5-\epsilon})$ ($\epsilon>0$) is demanded to obtain a $C^2$-compactification, while $C_g=O(|z|^{-5})$ is demanded to obtain a $C^{1,\alpha}$-compactification. These conditions are strictly stronger than the ones imposed in Theorem \ref{ThmCIntro}, both regarding the strength of the decay and the integral versus pointwise character. Along these lines, let us also comment that in \cite[Theorem C]{Herzlich1997} the author also analyses the existence of a $C^{0,\alpha}$ compactification under the decay $C_g=O(|z|^{-3-\epsilon})$. During our work we have not pursued establishing refinements to the $C^{0,\alpha}$-compactification for two reasons: Firstly, the applications to the existence of refined expansions like (\ref{SchwarzExpansionIntro}) require $C^{1,\alpha}$-controls for the compactification. Secondly, $C^{0,\alpha}$-regularity for the compatification can be deduced from \cite[Lemma 5.2]{MaxwellDiltsYamabeAE} under even weaker conditions, for instance, no pointwise a priori control on $C_g$. Finally, as stated in the introduction, further control can be obtained from the same methods by demanding further control on the Cotton tensor $C_g$. 

The above theorem plays a key role in the proof of the following result, which concerns the existence of a preferred structure of infinity where an AE 3-manifold is granted to have a \emph{first order Schwarzschildian expansion}:

\begin{theorem}\label{MainAEthemIntro}
Let $(M^3,g)$ be a smooth $W^{2,p}_{\tau}$-AE manifold with respect to a structure of infinity with coordinates $\{z^i\}_{i=1}^3$, with $\tau\in (-1,-\frac{1}{2})$ and $p>2$. Assume furthermore that: 
\begin{enumerate}
\item $R_g\in L^r_{-3-\epsilon}(M,\Phi_z)$ for some $r>3$ and $\epsilon>0$;
\item $C_{g}\in L^{p_1}_{\sigma}(M,\Phi_z)$ for some $-6<\sigma<-4$ and $p_1=\frac{3}{6+\sigma}$.
\end{enumerate}
Then, there is a structure of infinity with coordinates $\{\bar{z}^i\}_{i=1}^3$, which is $C^{1,\alpha}$-compatible with the original one, such that
\begin{align}\label{ThmDIntroExpansion}
\begin{split}
g(\partial_{\bar{z}^i},\partial_{\bar{z}^j})&=\left(1+\frac{4C}{|\bar{z}|}\right)\delta_{ij} + O_1(|\bar{z}|^{-1-\alpha}),\\
\bar{z}(z)-\mathrm{Id}(z)&\in C^1_{1-\alpha}(\mathbb{R}^n\backslash\overline{B_{R_0}(0)}),
\end{split}
\end{align}
for some $\alpha>0$.
\end{theorem}

The above theorem provides purely geometric conditions which guarantee the existence of geometrically constructed asymptotic coordinates where a type of RT parity conditions are satisfied, up to first order. Remarkably, this will be enough to guarantee convergence of the COM in the given coordinates. Let us also notice that the existence of this preferred structure of infinity is related to asymptotic conformal properties of the AE manifold $(M^3,g)$. In that sense, the above result could be compared with results in \cite{CorvinoCOM}, where asymptotically conformally flat manifolds were analysed. Along those lines, Theorem \ref{MainAEthemIntro} above provides a measure of how far from asymptotic conformal flatness one can be in order to still obtain (at least to first order) a Schwarzschildian expansion, such measure being provided by the asymptotic $L^p$-control on the Cotton tensor. An interesting open question, which will be addressed in future work, is the extension of the above theorem to higher dimensions.

Let us also briefly contrast Theorem \ref{MainAEthemIntro} above with \cite[Proposition 3.3]{BartnikMass}, where asymptotic control for an AE metric $g$ is obtained in asymptotic harmonic coordinates from the decay of $\mathrm{Ric}_g$. We highlight that an expansion like (\ref{ThmDIntroExpansion}) cannot be obtained in this manner for AE $3$-manifolds with non-vanishing ADM energy. This is further described through a family of examples presented after the proof of Theorem \ref{MainThmAE} in Section \ref{SectionDecompactification}.


\medskip
Finally, the above theorem will be used to address (at least partly) a conjecture posed by C. Cerderbaum and A. Sakovich concerning the convergence of the COM of a relativistic initial data set. Before stating the result, let us notice that the conjecture posed in \cite{CederbaumSakovich} concerns a novel proposed definition for the COM via space-time constant mean curvature foliations of infinity, where the extrinsic geometry of the initial data set within the evolving space-time plays an important role. Nevertheless, whenever the initial data set is \emph{time-symmetric} (totally geodesic), the problem becomes properly Riemannian. That is, the initial data set reduces to the prescription of a Riemannian manifold $(M^3,g)$ (with some scalar curvature constraint), and the definition for the COM proposed in \cite{CederbaumSakovich} reduces to (\ref{COMIntro}). In this context, the Riemannian version of \cite[Conjecture 1]{CederbaumSakovich} can be stated as follows:
\begin{conj}[Cederbaum-Sakovich \cite{CederbaumSakovich}]\label{CSConjecture}
Given an AE manifold $(M^3,g)$ of order $\tau\in (-1,-\frac{1}{2})$ with respect to a structure of infinity $\Phi_z$ with coordinates $\{z^i\}_{i=1}^3$, there is a geometric condition on the coordinates $\{z^i\}_{i=1}^3$ ensuring that (\ref{COMIntro}) converges in the $z$-coordinates, provided the assumption that $z^iR_g\in L^1(M,\Phi_z)$ holds. 
\end{conj}
We address the above conjecture from a slightly different angle, in particular weakening the claim. Conceptually, the main difference in our approach is that we look for a condition on the coordinates $\{z^i\}_{i=1}^3$, such that if $z^iR_g\in L^1(M,\Phi_z)$ one can guarantee the existence of \emph{compatible} asymptotic coordinates $\{\bar{z}^i\}_{i=1}^3$ where the COM (\ref{COMIntro}) converges.\footnote{By compatible coordinates, we mean that the coordinate change is given by a transformation which is asymptotic to the identity, such as those described by (\ref{ThmDIntroExpansion}).} We believe this approach is a possible reformulation of Conjecture \ref{CSConjecture} based on the examples produced in \cite{Huang2}. In this last reference, the author produces a family of scalar flat (and thus solutions to the time-symmetric vacuum ECE) AE manifolds $(M^3,g)$ of order $\tau=-1$ with respect to a given asymptotic coordinate system $\{z^i\}_{i=1}^3$, but for which the COM (\ref{COMIntro}) fails to converge in the given coordinates.\footnote{See \cite[Proposition 3.6 and Corollary 3.7]{Huang2}.} Thus, such a family presents a counterexample to the strongest possible version of Conjecture \ref{CSConjecture}, but one may still consider the following alternative:  
\begin{conj}[Weak version of Conjecture \ref{CSConjecture}]\label{CSConjectureWeak}
Given an AE manifold $(M^3,g)$ of order $\tau\in (-1,-\frac{1}{2})$ with respect to a structure of infinity $\Phi_z$ with coordinates $\{z^i\}_{i=1}^3$, if $z^iR_g\in L^1(M,\Phi_z)$, then there is a geometric condition on the coordinates $\{z^i\}_{i=1}^3$ ensuring the existence of a compatible asymptotic chart $\{\bar{z}^i\}_{i=1}^3$ such that (\ref{COMIntro}) converges in the $\bar{z}$-coordinates. 
\end{conj}

We will pursue the validity of Conjecture \ref{CSConjectureWeak}. Although this does not guarantee the convergence with respect to the originally given $\{z^i\}_{i=1}^3$ coordinates, one can think that this weaker version of Conjecture \ref{CSConjecture} actually identifies preferred asymptotic coordinates $\{\bar{z}^i\}_{i=1}^3$ where the Conjecture \ref{CSConjecture} is true. In that precise sense, one may consider that the weaker claim we address provides an answer to Conjecture \ref{CSConjecture}.

On the more technical side, we address an $L^p$-version of Conjecture \ref{CSConjectureWeak}. That is, we replace $z^iR_g\in L^1(M,\Phi_z)$ with $z^iR_g\in L^p(M,\Phi_z)$ for some $p>1$, and the answer we provide below to this question is that the desired geometric condition is given by a weighted $L^p$-control of the Cotton tensor. That is, the coordinates $\{z^i\}_{i=1}^3$ should provide a structure of infinity $\Phi_z$ such that $C_g\in L^{p_1}_{\sigma}(M,\Phi_z)$ for some $-6<\sigma<-4$ and $p_1=\frac{3}{6+\sigma}$. This will grant the existence of coordinate $\{\bar{z}^i\}_{i=1}^3$ satisfying (\ref{ThmDIntroExpansion}) and $\bar{z}^iR_g\in L^p(M,\Phi_{\bar{z}})$ as well, where the COM converges.
\begin{theorem}\label{ThmEIntro}
Let $(M^3,g)$ be an AE Riemannian manifold satisfying the hypotheses of Theorem \ref{MainAEthemIntro}. If moreover $R_g\in L^r_{-4-\epsilon}(M,\Phi_z)$, with $r>3$ and $\epsilon>0$, then the center of mass (\ref{COMIntro}) converges in the coordinates given by (\ref{ThmDIntroExpansion}).
\end{theorem}
 
We would like to once more stress that the above theorem relies on purely geometric hypotheses to grant convergence of the COM. Although the coordinates $\{\bar{z}^i\}_{i=1}^3$ will not be given explicitly, one can have explicit control over them up to first order under the hypotheses presented in Theorem \ref{ThmEIntro}. This kind of control is enough to provide an explicit computation of the ADM energy and COM, and as commented before, higher order controls should be expected if one demands higher order controls for the Cotton tensor.

\medskip
With all the above in mind, the paper is structured as follows. In Section \ref{Preliminaries} we will present detailed definitions concerning analytic tools to be used in the core of the paper. We will also set up our geometric conventions, and review some important results related to mapping properties of linear partial differential operators with low regularity coefficients. Experts in the associated fields can certainly skip most of this section and use it just for reference purposes about our notations. In Section \ref{SectionRegularity} we use the tools of Section \ref{Preliminaries} to establish some mapping properties for $\Delta_g$ and $\nabla^2$, for rough metrics, on low regularity sections of the bundle $S_2M$, consisting of symmetric $(0,2)$-tensor fields. The results and formulas established in Section \ref{SectionRegularity} are then used in Section \ref{SectionCompactification}, which is dedicated to the proof of Theorem \ref{ThmCIntro}. Then, Section \ref{SectionDecompactification} is dedicated to the proof of Theorem \ref{MainAEthemIntro}, and finally Section \ref{SectionCOM} is devoted to the proof of Theorem \ref{ThmEIntro}. 

We also include two appendices: Appendix \ref{AppendixBartnik} is motivated by \cite[Proposition 1.6]{BartnikMass}, where a key part of the proof is related to regularity theory of scalar elliptic operators with low regularity coefficients. Due to subtleties associated with the corresponding regularity theory which are explained in this appendix, we have decided to provide a self-contained proof of the associated statements within the context that will be used in this paper, which concerns a special case of the operators treated in \cite{BartnikMass}. Finally, within Appendix \ref{AppendixMaxwell} we rephrase \cite[Lemma 5.2]{MaxwellDiltsYamabeAE} with some explicit details which we believe may help the reader, since we appeal to them within Section \ref{SectionDecompactification}. 
 
\section{Preliminaries}\label{Preliminaries}

\subsection{Geometric conventions}

To avoid any ambiguity, let us make explicit the curvature conventions we follow in this text, where, given an $(n+1)$-dimensional pseudo-Riemannian manifold $(M^{n+1},g)$ and denoting by $\nabla$ its associated Riemmanian connection, the curvature tensor is defined as:
\begin{align*}
R(X,Y)Z=\nabla_X\nabla_YZ - \nabla_Y\nabla_XZ - \nabla_{[X,Y]}Z, \text{ for all } X,Y,Z\in \Gamma(TM).
\end{align*}
Also, given an arbitrary coordinate system $\{x^i\}^{n+1}_{i=1}$ on $M$, we label its components as follows:
\begin{align*}
\begin{split}
R^{i}_{jkl}&=dx^{i}(R(\partial_k,\partial_l)\partial_j)=\partial_k\Gamma^{i}_{lj}-\partial_l\Gamma^{i}_{kj} +  \Gamma^{i}_{ku}\Gamma^u_{jl}- \Gamma^{i}_{lu}\Gamma^u_{jk}.\\
\end{split}
\end{align*}
where $\Gamma$ stands for the corresponding Christoffel symbols of the Riemannian connection associated with $g$. From this we get the Ricci tensor from the following contraction:
\begin{align*}
\mathrm{Ric}_{ij}\doteq R^l_{ilj}.
\end{align*}
We also consider the $(0,4)$-curvature tensor, given by
\begin{align*}
R(V,X,Y,Z)&=g(R(Y,Z)X,V).
\end{align*}

Let us also recall the following local formulae the Ricci tensor, presented in arbitrary local coordinates $\{x^i\}_{i=1}^n$, which is known to be useful to understand optimal regularity for a Riemannian metric based on the regularity of its Ricci tensor:
\begin{align}\label{gauge0}
\begin{split}
R_{ij}(g)&=-\frac{1}{2}g^{ab}\partial_{ab}g_{ij}+\frac{1}{2}(g_{ia}\partial_{j}F^{a}+g_{ja}\partial_{i}F^{a}) +f_{ij}(g,\partial g),\\
F^{a}&\doteq g^{kl}\Gamma^{a}_{kl}(g)
\end{split}
\end{align}
where $f_{ij}$ is a quadratic form on $\partial g$, explicitly given by
\begin{align}\label{f-tensor}
\begin{split}
\!\!f_{ij}(g,\partial g)&\doteq   - \frac{1}{2}\big\{  \partial_{i}g^{ab}\partial_{a}g_{b j}  + \partial_{j}g^{ab}\partial_{a}g_{b i} \}  + \frac{1}{2}F^{a}\partial_{a}g_{ij} - \Gamma^{a}_{b i}\Gamma^{b}_{a j}.
\end{split}
\end{align}

Let us now introduce the Cotton tensor associated with a Riemannian manifold $(M^n,g)$, which in an arbitrary coordinate system is given by:
\begin{align}\label{CottonTensorInto}
C_{ijk}(g)\doteq \nabla_k\mathrm{Ric}_{ij} - \nabla_j\mathrm{Ric}_{ik} + \frac{1}{4}(\nabla_jR_gg_{ik} - \nabla_kR_gg_{ij}).
\end{align}
It is important to recall the well-known fact that in three dimensions the Cotton tensor is a conformal invariant, which presents the main obstruction for a 3-manifold to be conformally flat.

\subsection{Analytic tools}

In this section we elaborate on several analytic tools needed for the PDE analysis in subsequent sections. We will exploit this space to set up several notations and conventions, and review both well-known as well as some more subtle properties associated with Sobolev spaces and elliptic operators with coefficients of limited regularity. Experienced readers in PDEs and geometric analysis can easily skip the details of this section and use it just as a source for our conventions. 

\subsubsection{Sobolev spaces}

\begin{defn}
Let $U\subset \mathbb{R}^n$ be an arbitrary domain, $k$ be a non-negative integer and $1\leq p\leq \infty$ an extended real number. We define the Sobolev space $W^{k,p}(U)$ as the space of functions $f\in L^p(U)$ which possess weak derivatives $\{\partial^{\alpha}f\}_{0\leq |\alpha|\leq k}$ of order up to $k$ in $L^p(U)$. That is,
\begin{align}
W^{k,p}(U)\doteq \{ f\in L^p(U)\: :\: \partial^{\alpha}f\in L^p(U) \: \forall\: 0\leq |\alpha|\leq k\}.
\end{align}
We equip these vector subspaces of $L^p$ with the norm
\begin{align}\label{SobolevNorm1}
\begin{split}
\Vert f\Vert_{W^{k,p}(U)}&\doteq \left(\sum_{|\alpha|=0}^{k}\Vert \partial^{\alpha}f\Vert^p_{L^p(U)} \right)^{\frac{1}{p}} \text{ if } 1\leq p<\infty,\\
\Vert f\Vert_{W^{k,\infty}(U)}&\doteq \max_{0\leq |\alpha|\leq k}\Vert\partial^{\alpha}u\Vert_{L^{\infty}(U)}.
\end{split}
\end{align}
\end{defn}

\begin{defn}
Let $M^n$ be a closed smooth manifold, $n\geq 3$. Let $E\xrightarrow[]{\pi} M$ be a vector bundle over $M$ with fibre dimension $r$, let $\{U_i,\phi_i,\rho_i,\eta_i\}_{i=1}^N$ be a cover of $M$ by coordinate charts $\{U_i,\phi_i\}_{i=1}^N$ trivialising $E$ over each coordinate chart, with bundle charts $\{U_i,\phi_i,\rho_i\}_{i=1}^N$, where $\rho_i:\pi^{-1}(U_i)\to U_i\times \mathbb{R}^r$ denotes the bundle trivialisation, and $\{\eta_i\}$ is a partition of unity subordinate to such cover. Given a real number $1\leq p< \infty$ and a non-negative integer $k\in\mathbb{N}_0$, we then define the Sobolev spaces $W^{k,p}(E)$ of section of $E$ as
\begin{align}\label{SobolevClassical}
\!\!\!\!W^{k,p}(E)=\{u\in L^p(E) \: : \: \tilde{\rho}^l_i\circ (\eta_iu)\circ\phi_i^{-1} \in W^{k,p}(\phi(U_i)) \text{ for all } i=1,\cdots,N \text{ and all } 1\leq l\leq r \},
\end{align}
where $\tilde{\rho}_i$ denotes the projection of $\rho_i$ onto its second factor, equipped with the norm
\begin{align}\label{SobolevNormVB}
\Vert u\Vert_{W^{k,p}(E)}\doteq \sum_{i=1}^N\sum_{l=0}^r\Vert \tilde{\rho}^l_i\circ(\eta_iu)\circ\phi_i^{-1} \Vert_{W^{k,p}(\phi(U_i))}.
\end{align}
\end{defn}

The above spaces are seen to be equivalently defined as the space of $L^p(E,dV_g)$ sections having weak covariant derivatives in a given smooth background metric $g$ up to order $k$ in $L^p(E,dV_g)$. Let us now notice that for negative integer value $-k$ and $1<p<\infty$, on any domain $\Omega\subset \mathbb{R}^n$, we have the definition of Sobolev spaces $W^{-k,p'}(\Omega)\doteq (W_0^{k,p}(\Omega))'$ for scalar functions, where $W_0^{k,p}(\Omega)$ denotes the closure of $C^{\infty}_0(\Omega)$ in the Sobolev norm (\ref{SobolevNorm1}) and $\frac{1}{p}+\frac{1}{p'}=1$. If one looks at distributional sections of a vector bundle $E\to M$ of rank $r$, as maps which are locally described as mappings from a local coordinate domain $U\to [\mathcal{D}'(\varphi(U))]^r$, we can still use the scalar version of Sobolev spaces to give meaning to negative order Sobolev spaces of sections of vector bundles as distributional sections whose local components take values in $W^{-k,p'}(\varphi(U))$. In order to make this more precise, let us introduce the following conventions, which we are extracting from \cite{HolstBehzadanSob1}. In this reference, the reader will find a very detailed description of Sobolev spaces of vector bundle sections on compact manifolds.

Let us start considering a smooth Riemannian manifold $(M^n,g_0)$ and a vector bundle $E\xrightarrow[]{\pi} M$ and denote by $(U_i,\rho_i)_{i=1}^N$ a set bundle charts for $E$. That is, $\rho_i:\pi^{-1}(U_i)\to U_i\times \mathbb{R}^r$ are the local trivialisation maps. Then, consider the density bundle over $M$ given by
\begin{align*}
|\Lambda|(M)=\coprod_{p\in M}|\Lambda|(T_pM)
\end{align*}
where $|\Lambda|(T_pM)$ denotes the space of 1-densities on $T_pM$. On a coordinate domain $(U_i,\varphi_i)$, with coordinates $x^i$, we fix $\mu_g$ to be the unique $1$-density satisfying $\mu_g=|\sqrt{\mathrm{det}(g_0)}dx^1\wedge\cdots\wedge dx^n|$, and we fix the associated bundle charts to $|\Lambda|(M)$ over $U_i$ by
\begin{align*}
(\rho_{\Lambda})_i:\pi^{-1}_{\Lambda}(U_i)&\to U_i\times \mathbb{R},\\ 
(p,\nu_p=a\mu_g|_{p})&\mapsto (p,a).
\end{align*}

Let us now introduce the bundle $E^{\vee}\doteq \mathrm{Hom}(E,|\Lambda|(M))$, given by
\begin{align*}
E^{\vee}\doteq \coprod_{p\in M}\mathrm{Hom}(E_p,|\Lambda|_p(M)),
\end{align*}
where $\mathrm{Hom}(E_p,|\Lambda|_p(M))$ denotes the set of linear maps from $E_p$ to $|\Lambda|_p(M)$, and, given a cover $\{U_i,\varphi_i\}_{i=1}^N$ by coordinate charts, the associated bundle charts to $E^{\vee}$ are given by
\begin{align*}
\rho^{\vee}_i:\pi^{-1}_{E^{\vee}}(U_i)&\to U_i\times\mathbb{R}^r\\
(p,A)&\mapsto (p,\Phi\circ (\tilde{\rho}_{\Lambda})_i|_{\pi^{-1}_{\Lambda}(p)}\circ A\circ (\rho_i\vert_{E_p})^{-1}),
\end{align*}
where $\rho_i$ and $\rho_{\Lambda}$ were fixed above, $\tilde{\rho}_{\Lambda}$ denotes the projection of $\rho_{\Lambda}$ onto its second factor, $A\in \mathrm{Hom}(E_p,|\Lambda|_p(M))$,  and, fixing a canonical basis $\{e_i\}^r_{i=1}$ for $\mathbb{R}^r$, $\Phi:(\mathbb{R}^r)'\to \mathbb{R}^r$ is the usual isomorphism identifying $(\mathbb{R}^r)'\cong \mathbb{R}^r$, via\footnote{Notice that $(\tilde{\rho}_{\Lambda})_i|_{\pi^{-1}_{\Lambda}(p)}\circ A\circ (\rho_i\vert_{E_p})^{-1}):\mathbb{R}^r\to \mathbb{R}$ defines an element of $(\mathbb{R}^r)'$.}
\begin{align*}
\Phi(u)=\sum_{i=1}^ru(e_i)e_i.
\end{align*}

We then take the space of test sections associated with $E$ to be the set $C_0^{\infty}(M;E^{\vee})$ equipped, as usual, with the inductive limit topology induced by the $C^{k}$ family of semi-norms on compact subsets. We denote this topological vector space by $\mathcal{D}(M;E^{\vee})$ and then define the set of distributions as $\mathcal{D}'(E)\doteq \left(\mathcal{D}(M;E^{\vee})\right)'$. There are a few things to highlight, whose details can be found in \cite[Section 6]{HolstBehzadanSob1}. First of all, over given a chart $(U,\varphi)$, let us denote
\begin{align*}
[\mathcal{D}(\varphi(U))]^r\doteq \underbrace{\mathcal{D}(\varphi(U))\times \cdots\times \mathcal{D}(\varphi(U))}_{\text{r-times}},
\end{align*}
and notice that the linear map
\begin{align*}
\tilde{T}_{E^{\vee},U,\varphi}:\mathcal{D}(U,E_U^{\vee})&\to [\mathcal{D}(\varphi(U))]^r\\
 \xi&\mapsto \tilde{\rho}^{\vee}\circ\xi\circ\varphi^{-1},
\end{align*}
is a topological isomorphism due to \cite[Theorem 46]{HolstBehzadanSob1} and we denote its (continuous) inverse by
\begin{align*}
T_{E^{\vee},U,\varphi}\doteq \tilde{T}^{-1}_{E^{\vee},U,\varphi}:[\mathcal{D}(\varphi(U))]^r&\to \mathcal{D}(U,E_U^{\vee})
\end{align*}
It then follows that the adjoint map
\begin{align*}
T^{*}_{E^{\vee},U,\varphi}:\left(\mathcal{D}(U,E_U^{\vee})\right)'&\to \left[\mathcal{D}(\varphi(U))]^r\right]^{'},\\
u&\mapsto (T^{*}_{E^{\vee},U,\varphi}u)(\xi_1,\cdots,\xi_r)=u\left(T_{E^{\vee},U,\varphi}(\xi_1,\cdots,\xi_r) \right)
\end{align*}
is also a linear topological isomorphism, where
\begin{align*}
T_{E^{\vee},U,\varphi}(\xi_1,\cdots,\xi_r)=(\tilde{\rho}^{\vee})^{-1}\circ(\xi_1,\cdots,\xi_r)\circ \varphi.
\end{align*}

Let us also introduce the linear map:
\begin{align*}
L:\left(\left[ \mathcal{D}(\varphi(U))\right]^r\right)'&\to \left[ \mathcal{D}'(\varphi(U))\right]^r,\\
 v&\mapsto Lv=(v\circ i_1,\cdots,v\circ i_r)
\end{align*}
where 
\begin{align*}
i_j: \mathcal{D}(\varphi(U))&\to \left[ \mathcal{D}(\varphi(U))\right]^r,\\
f&\mapsto (0,\cdots,\underbrace{f}_{j-th \text{ slot}},\cdots,0)
\end{align*}
which is also a topological isomorphism due to \cite[Theorem 24]{HolstBehzadanSob1}. Then, define
\begin{align*}
H_{E^{\vee},U,\varphi}\doteq L\circ T^{*}_{E^{\vee},U,\varphi}:\left(\mathcal{D}(U,E_U^{\vee})\right)'\to \left[ \mathcal{D}'(\varphi(U))\right]^r,
\end{align*}
which assigns to a distribution on $\mathcal{D}(U,E_U^{\vee})$ a vector of $r$-distributions on $\mathcal{D}'(\varphi(U))$, which we understand as the components of the vector bundle distributions in $\left(\mathcal{D}(U,E_U^{\vee})\right)'$.

Since multiplication by smooth functions is a continuous map on $\mathcal{D}(M,E^{\vee})$, then we can localise a distribution 
$\mathcal{D}'(E)$ in a coordinate chart $(U,\varphi)$ multiplying by a cut-off function $\eta\in C^{\infty}_0(U)$, so that
\begin{align*}
m_{\eta}:\mathcal{D}'(E)&\to \left(\mathcal{D}(U,E_U^{\vee})\right)',\\
u&\mapsto \eta u,
\end{align*}
where $\eta u$ is defined as usual by duality:
\begin{align*}
(\eta u)(v)=u(\eta v), \: \forall\: v\in \mathcal{D}(U,E_U^{\vee}),
\end{align*}
where we understand $\eta v\in \mathcal{D}(M,E^{\vee})$ extended by zero outside of $U$. All this notation gives us a natural way to introduce Sobolev sections of negative degree of regularity: Since we have a well-understood meaning for spaces $W^{-k,p}(\varphi(U))$, $k\in\mathbb{N}_0$ and $1<p<\infty$, in the case of scalar functions, in analogy to (\ref{SobolevNormVB}), we can define
\begin{align}\label{SobolevVBnegativeReg}
\!\!\!\!W^{-k,p}(E)=\{u\in \mathcal{D}'(E) \: : \: H_{E^{\vee},U_i,\phi}^l\circ (\eta_iu)\circ\phi_i^{-1} \in W^{-k,p'}(\phi(U_i)) \text{ for all } i=1,\cdots,N \text{ and all } 1\leq l\leq r \},
\end{align}
equipped with the norm (\ref{SobolevNormVB}), with $(k,p)$ replaced by $(-k,p)$ understood in the usual scalar case for each component. That is, $W^{-k,p}(E)$ denotes the subspace of $\mathcal{D}'(E)$ consisting of vector bundle distributions which components (understood via the maps $H_{E^{\vee},U_i,\varphi}$ described above) belong to $W^{-k,p}(\varphi(U_i))$.

Let us end these definitions highlighting a few important details. First, given a smooth closed Riemannian manifold $(M,g_0)$ and a vector bundle $E\to M$ with fibre metric $\langle \cdot,\cdot\rangle_E$, regular distributions are given by maps
\begin{align*}
l_u:\mathcal{D}(M;E^{\vee})&\to \mathbb{R},\\
v&\mapsto l_u(v)=\int_Mv(u)
\end{align*}
associated with $u\in C^{\infty}(M;E)$. In this context one has the maps
\begin{align}\label{RegDistComp1}
\begin{split}
\varphi(U)&\to \mathbb{R}^r\\
x&\mapsto (\tilde{\rho}^1\circ u\circ\varphi^{-1}(x),\cdots,\tilde{\rho}^r\circ u\circ\varphi^{-1}(x)),
\end{split}
\end{align} 
attaching a vector valued map to $u|_{U}$, and
\begin{align}\label{RegDistComp2}
\begin{split}
H^1_{E^{\vee},U,\varphi}l_u = \left(\left( H_{E^{\vee},U,\varphi}\circ l_u\right)^1,\cdots,\left( H_{E^{\vee},U,\varphi}\circ l_u\right)^r\right)
\end{split}
\end{align}
attaching $r$-distributions on $\varphi(U)\subset \mathbb{R}^n$ to $l_u|_{U}$. One might wonder whether the $k$-th component of the regular distribution $l_u$ given in (\ref{RegDistComp2}) agrees with the regular distribution $l_{\tilde{u}^k}$ defined by the $k$-th component $\tilde{u}^k\doteq \tilde{\rho}^k\circ u\circ\varphi^{-1}$ of $u$ as given in (\ref{RegDistComp1}). Using all the above definitions, one can see that this is actually the case (see \cite[Remark 32]{HolstBehzadanSob1}):
\begin{align*}
\left( H_{E^{\vee},U,\varphi}\circ l_u\right)^k=l_{\tilde{\rho}^k\circ u\circ\varphi^{-1}}, \text{ for all } u\in C^{\infty}(M;E).
\end{align*} 
This makes it possible to use (\ref{SobolevVBnegativeReg}) for any $k\in\mathbb{Z}$ and $1<p<\infty$, which in the case of regular distributions agrees with (\ref{SobolevClassical}).

Finally, notice that if in addition to a (smooth) Riemannian metric $g_0$ on $M$ we have a fibre metric $\langle\cdot,\cdot \rangle_E$ on $E$, with the conventions established above, the map
\begin{align*}
\mathcal{T}:\mathcal{D}(M;E)&\to \mathcal{D}(M;E^{\vee}),\\
v&\mapsto \langle v,\cdot\rangle_EdV_{g_0}
\end{align*}
is a linear topological isomorphism \cite[Lemma 13]{HolstBehzadanSob1}, where $\mathcal{T}(v)_p\in \mathrm{Hom}(E_p,|\Lambda|_p(M))$ is given by
\begin{align*}
\mathcal{T}(v)_p(u_p)=\langle v,u\rangle_E|_{p}dV_{g_0}.
\end{align*}
As a consequence, the adjoint map 
\begin{align*}
\mathcal{T}^{*}:\mathcal{D}'(E)\doteq \left(\mathcal{D}(M;E^{\vee})\right)'&\to \left(\mathcal{D}(M;E)\right)'
\end{align*}
is a linear continuous bijective map, which can be used to identify $\mathcal{D}'(E)\cong \left(\mathcal{D}(M;E)\right)'$ with the aid of these additional Riemannian structures.

\medskip
Let us highlight that in the case of vector bundles, from the definitions above, the relation between $W^{k,p}(E)$ and $W^{-k,p'}(E)$ does not come by definition as in the scalar case, although the following result holds (see \cite[Theorem 100]{HolstBehzadanSob1}):

\begin{theo}\label{DualIsomorphismHolst}
Let $(M^n,\bar{g})$ be a closed Riemannian manifold, $\bar{g}$ smooth, $E\xrightarrow[]{\pi} M$ a vector bundle over $M$ equipped with a smooth fibre metric $\langle \cdot,\cdot\rangle_E$. Let $k\in\mathbb{Z}$ and $1<p<\infty$ and denote by
\begin{align*}
\langle u,v \rangle_{L^2(M,\bar{g})}=\int_M\langle u,v\rangle_EdV_{\bar{g}}
\end{align*}
the associated $L^2$ inner product. Then, 
\begin{enumerate}
\item $\langle \cdot, \cdot \rangle_{L^2(M,\bar{g})}:\mathcal{D}(M;E)\times \mathcal{D}(M;E)\to \mathbb{R}$ extends by continuity to a bilinear pairing $\langle \cdot, \cdot \rangle_{(M,\bar{g})}:W^{-k,p'}(E)\times W^{k,p}(E)\to \mathbb{R}$;
\item The map
\begin{alignat*}{3}
S_{k,p}:W^{-k,p'}(M;E)&\to (W^{k,p}(M;E))',\\
		u&\mapsto S_{k,p}(u)=\mathcal{L}_u: & W^{s,p}(M;E) &\to  \mathbb{R},\\
		&		   & v &\mapsto \langle u,v\rangle_{(M,\bar{g})}
\end{alignat*}
is a topological isomorphism.
\end{enumerate}
\end{theo}

Let us now recall the following lemma about composition with Sobolev functions, which proof follows along the lines of \cite[Lemma 2.2]{MaxwellRoughClosed}:

\begin{lem}[Composition Lemma]\label{CompositionLemma}
Let $U$ be a smooth bounded domain in $\mathbb{R}^n$, and let $F:I\to \mathbb{R}$ be a function of class $C^{m}$ on some open interval $I\subset \mathbb{R}$ and $f\in W^{m,p}(U)$ with $m>\frac{n}{p}$ and $1\leq p<\infty$, satisfying $\overline{f(U)}\subset I$. Then, $F\circ f\in W^{m,p}(U)$.
\end{lem}

Before presenting further results, let us highlight that, although during most of this paper we will only use Sobolev spaces with \emph{integer degree of regularity}, we will need in a few situations to appeal to interpolating spaces. We have then decided to present all the above definitions and results in the context of $W^{k,p}$-spaces with $k\in \mathbb{Z}$ to avoid unnecessary complications. Nevertheless, we will now introduce our conventions for the interpolating spaces, and appeal to them whenever necessary. With this in mind, let us introduce the space of Bessel potentials $H^{s,p}(\mathbb{R}^n)$ with $s\in \mathbb{R}$ and $1<p<\infty$:
\begin{align}\label{BesselPotentials.1}
H^{s,p}(\mathbb{R}^n)=\{u\in \mathcal{S}'(\mathbb{R}^n)\: :\: \mathcal{F}^{-1}\big((1+|\xi|^2)^{\frac{s}{2}}\hat{u}\big)\in L^p(\mathbb{R}^n)\},
\end{align}
where above $\mathcal{S}(\mathbb{R}^n)$ denotes the space of Schwartz functions, $\mathcal{S}'(\mathbb{R}^n)$ the space of tempered distributions, $\mathcal{F}:\mathcal{S}'(\mathbb{R}^n)\to \mathcal{S}'(\mathbb{R}^n)$ the Fourier transform, and we have denoted by $\hat{u}\doteq \mathcal{F}(u)$. We furthermore equip $H^{s,p}(\mathbb{R}^n)$ with the norm:
\begin{align}\label{BesselPotentials.2}
\Vert u\Vert_{H^{s,p}(\mathbb{R}^n)}\doteq \Big\Vert \mathcal{F}^{-1}\big((1+|\xi|^2)^{\frac{s}{2}}\hat{u}\big)\Big\Vert_{L^{p}(\mathbb{R}^n)}.
\end{align}
These are known to be Banach spaces which coincide with $W^{k,p}(\mathbb{R}^n)$ whenever $k\in \mathbb{N}_0$ (see, for instance, \cite[Chapter 13, Proposition 6.1]{Taylor3}), and also $H^{-k,p'}(\mathbb{R}^n)\cong \left(H^{k,p}(\mathbb{R}^n)\right)'$. Moreover, for $s\in \mathbb{R}$, the usual Sobolev embeddings extend to this setting under the same conditions, as can be seen from \cite[Proposition 6.3 and Proposition 6.4]{Taylor3}, and a summary of all these properties can be found in \cite[Chapter VII, Theorem 7.63]{Adams}. An additional important property of Bessel potential spaces, is that these are known to coincide with the interpolation spaces via complex interpolation, as introduced by A. P. Calderón, between $W^{k,p}$-spaces. Below, we appeal to a few of the properties which follow from such characterisation. For detailed descriptions of the corresponding interpolation theory, we refer the reader to \cite{Calderon1,Triebel1}, although most of the properties we use are contained in the concise description of \cite[Chapter 4, Section 2]{Taylor1}. 

Let us recall that, given a pair of Banach spaces $E$ and $F$ such that $F\hookrightarrow E$, complex interpolation produces a family of Banach spaces, denoted by $[E,F]_{\theta}\hookrightarrow E$, parametrised by $\theta\in (0,1)$, which interpolate between $E$ and $F$. In particular, given $1<p<\infty$, $k\in\mathbb{N}$, and setting $E=L^p(\mathbb{R}^n)$ and $F=W^{k,p}(\mathbb{R}^n)$ one has (see, for instance, \cite[Proposition 6.2]{Taylor3}):
\begin{align}
[L^p(\mathbb{R}^n),W^{k,p}(\mathbb{R}^n)]_{\theta}=H^{s,p}(\mathbb{R}^n), \text{ for } s=\theta k \text{ and } \theta\in (0,1).
\end{align}

These last set of properties can then be used to \emph{define} these spaces on bounded domains, for instance following \cite[Chapter VII, Section 7.66]{Adams1st}. That is, letting $\Omega$ be a smooth bounded domain in $\mathbb{R}^n$, $1<p<\infty$, $\theta\in (0,1)$ and $k\in \mathbb{N}_0$, we define:
\begin{align}\label{BesselPotentials.3}
H^{s,p}(\Omega)\doteq [L^p(\Omega),W^{k,p}(\Omega)]_{\theta}, \text{ for } s=\theta k.
\end{align}
We then denote by $H^{s,p}_0(\Omega)$ the closure of $C^{\infty}_0(\Omega)$ in $H^{s,p}(\Omega)$, and for $1<p<\infty$ and $s<0$ a real number, we define
\begin{align}\label{BesselPotentials.4}
H^{s,p'}(\Omega)\doteq [H_0^{-s,p}(\Omega)]'.
\end{align}
In the above setting, we extend these definitions for functions and vector bundle sections defined on a closed manifold $M$ appealing to a coordinate cover by coordinate balls, a partition of unity subordinate to it, trivialisations over it whenever necessary, and then using the same definitions as above for the $W^{k,p}(M)$-spaces, but this time demanding the coordinate expressions to belong to $H^{s,p}(\Omega)$-spaces instead of $W^{k,p}(\Omega)$.

\medskip
It will be crucial for us to understand multiplication properties of Sobolev sections of tensor bundles, even for negative Sobolev regularity. Such properties are key when analysing continuity properties of partial differential operators acting of Sobolev-type spaces and have been analysed in several references, such as \cite[Section 9]{PalaisBook}, \cite[Chapter VI, Section 3]{CB-deWitt2}, \cite[Section 2]{MaxwellHolstRegularity} and \cite{HolstBehazdanMult}, establishing related results under a variety of hypotheses. Below, we present a useful version of such properties, derivable by recurrent applications of Sobolev embeddings and Hölder inequalities, and which is a particular case of \cite[Theorem 2.5]{MaxwellHolstRegularity}.
\begin{theo}\label{BesselMultLocal}
Consider a smooth bounded domain $\Omega\subset \mathbb{R}^n$. Let $s_1,s_2$ be real numbers and $s\in \mathbb{Z}$, satisfying $s_1+s_2\geq 0$, $s_i\geq s$ for $i=1,2$, and let $p_1,p_2$ and $p$ be real numbers $1<p, p_i< \infty$, $i=1,2$. If $s_1,s_2,s\geq 0$, then the following continuous multiplication property holds
\begin{align}\label{LocalMultiPropPositiveBessel}
H^{s_1,p_1}(\Omega)\otimes H^{s_2,p_2}(\Omega)\hookrightarrow H^{s,p}(\Omega)
\end{align}
as long as 
\begin{subequations}\label{MultiplicationConditions}
\begin{align}
s_i-s&\geq n\left(\frac{1}{p_i} - \frac{1}{p}\right) \text{ and } s_1+s_2-s> n\left(\frac{1}{p_1} + \frac{1}{p_2} - \frac{1}{p} \right), \text{ or }\label{MultiplicationConditionsBesell}
\\
s_i-s&> n\left(\frac{1}{p_i} - \frac{1}{p}\right) \text{ and } s_1+s_2-s\geq n\left(\frac{1}{p_1} + \frac{1}{p_2} - \frac{1}{p} \right).\label{MultiplicationConditions.2}
\end{align}
\end{subequations}
Furthermore, under the same conditions, it also holds that
\begin{align}\label{LocalMultiPropPositive0Bessel}
H^{s_1,p_1}(\Omega)\otimes H_0^{s_2,p_2}(\Omega)\hookrightarrow H_0^{s,p}(\Omega)
\end{align}
In case that $\min(s_1,s_2)<0$, then (\ref{LocalMultiPropPositiveBessel}) holds if additionally one imposes
\begin{subequations}\label{MultiplicationConditionsDuals}
\begin{align}\label{MultiplicationConditionsDualsBessel} 
s_1+s_2&\geq n\left( \frac{1}{p_1}+\frac{1}{p_2} - 1 \right) \text{ whenever (\ref{MultiplicationConditionsBesell}) is satisfied} , \text{ or }\\
s_1+s_2&> n\left( \frac{1}{p_1}+\frac{1}{p_2} - 1 \right) \text{ whenever (\ref{MultiplicationConditions.2}) is satisfied}.\label{MultiplicationConditionsDuals.2}
\end{align}
\end{subequations}
\end{theo}

Appealing to the above theorem and using our definitions for Sobolev sections of vector bundles, we directly obtain:
\begin{cor}\label{ContractionsSobReg}
Consider a closed manifold $M^n$, let $s_1,s_2\in\mathbb{R}$ and $s\in\mathbb{Z}$ satisfy $s_1+s_2\geq 0$, $s_i\geq s$ for $i=1,2$, and let $p_1,p_2$ and $p$ be real numbers $1<p, p_i< \infty$, $i=1,2$. If $s_1,s_2,s\geq 0$, then the following continuous multiplication property holds for sections of tensor bundles $T^l_mM$ and $T^q_rM$:
\begin{align}\label{LocalMultiPropManifolds}
H^{s_1,p_1}(T^l_mM)\otimes H^{s_2,p_2}(T^q_rM)\hookrightarrow H^{s,p}(T^l_mM\otimes T^q_rM)
\end{align}
as long as 
\begin{subequations}\label{MultiplicationConditionsManifoldsMain}
\begin{align}
s_i-s&\geq n\left(\frac{1}{p_i} - \frac{1}{p}\right) \text{ and } s_1+s_2-s> n\left(\frac{1}{p_1} + \frac{1}{p_2} - \frac{1}{p} \right), \text{ or } \label{MultiplicationConditionsManifolds}\\
s_i-s&> n\left(\frac{1}{p_i} - \frac{1}{p}\right) \text{ and } s_1 + s_2 - s\geq n\left(\frac{1}{p_1} + \frac{1}{p_2} - \frac{1}{p} \right).\label{MultiplicationConditionsManifolds.2}
\end{align}
\end{subequations}

In case that $\min(s_1,s_2)<0$, then (\ref{LocalMultiPropManifolds}) holds if additionally one imposes 
\begin{subequations}\label{MultiplicationConditionsDualsManifoldsMain}
\begin{align} 
s_1+s_2&\geq n\left( \frac{1}{p_1}+\frac{1}{p_2} - 1 \right) \text{ whenever (\ref{MultiplicationConditionsManifolds}) is satisfied, or} \label{MultiplicationConditionsDualsManifolds}\\ 
s_1+s_2&> n\left( \frac{1}{p_1}+\frac{1}{p_2} - 1 \right) \text{ whenever (\ref{MultiplicationConditionsManifolds.2}) is satisfied}.\label{MultiplicationConditionsTensorsDuals.2}
\end{align}
\end{subequations}
Furthermore, if $u\in H^{s_1,p_1}(T^l_mM)$, $v\in H^{s_2,p_2}(T^q_rM)$ and we denote an arbitrary contraction between these two tensors by $u_{\cdot}v$, then, under the same conditions as above, $u_{\cdot}v\in H^{s,p}$ and there is a fixed constant $C>0$ such that $\Vert u_{\cdot}v\Vert_{H^{s,p}}\leq C\Vert u\Vert_{H^{s_1,p_1}}\Vert v\Vert_{H^{s_2,p_2}}$.
\end{cor}

\subsubsection{Mapping properties of $\mathcal{L}^2(W^{2,q})$-operators - $q>\frac{n}{2}$}\label{SectionMappingProps}

In this section we recapitulate quite general mapping properties related to partial differential operators which apply in fairly low regularity. Related results have been analysed previously in the literature, for instance in relation to rough solutions to the Einstein constraint equations. Among such references, we highlight the recent paper \cite[Section 2]{MaxwellHolstRegularity}, where more general results than those presented here have been established.  

Let us start considering a closed manifold $M^n$ and a vector bundle $E\to M$ over $M$. We are interested in the class of second order linear operators $L$ on $C^{\infty}(E)$, which over any given coordinate patch trivialising $E$ via $(U,\varphi,\rho)$, can locally be written as
\begin{align}\label{Holst.1}
(Lu)\circ\varphi^{-1}=\sum_{|\alpha|\leq 2}A_{\alpha}\partial^{\alpha}(\tilde{\rho}\circ u\circ\varphi^{-1}), \text{ with } A^{\alpha}\in W_{loc}^{|\alpha|,q}(U,\mathbb{R}^{r\times r}),
\end{align}
where above $\alpha=(\alpha_1,\cdots,\alpha_n)$ denotes an arbitrary multi-index, $|\alpha|=\sum_{i}^n\alpha_i$, $\partial_{\alpha}=\partial_{(x^1)^{\alpha_1}\cdots (x^n)^{\alpha_n}}$ and $r$ is the dimension of the fibre of $E\to M$. Along the lines of \cite{MaxwellHolstRegularity}, we denote the set of such operators by $\mathcal{L}^2(W^{2,q})$. One can then establish the following lemma concerning the mapping properties of such an operator:
\begin{lem}\label{ContinuityPropsGeneralOps}
Let $L$ be a differential operator of class $\mathcal{L}^2(W^{2,q})$. Then, given $1<p< \infty$ and an integer $0\leq k\leq 2$, $L$ can be extended as a bounded linear map
\begin{align}
L:W^{k,p}(E)\to W^{k-2,p}(E)
\end{align}
as long as $q>\frac{n}{2}$ and $\frac{1}{q}+\frac{k-2}{n}\leq \frac{1}{p}\leq 1 - \frac{1}{q}+ \frac{k}{n}$.
\end{lem}
\begin{remark}\label{RemarkContinuityProps}
Concerning the condition $\frac{1}{q}+\frac{k-2}{n}\leq \frac{1}{p}\leq 1 - \frac{1}{q}+ \frac{k}{n}$, let us highlight that:
\begin{itemize}
\item If $k=2$ we we obtain $\frac{1}{q}\leq \frac{1}{p}\leq \frac{1}{q'}+\frac{2}{n}$, which is equivalent to $\frac{nq'}{n+2q'}\leq p\leq q$. Notice then that
\begin{align*}
\frac{nq'}{n+2q'}\leq 1 \Longleftrightarrow 1-\frac{1}{q}+\frac{2}{n}\geq 1\Longleftrightarrow \frac{1}{q}\leq \frac{2}{n},
\end{align*}
where the last inequality holds by hypothesis. That is, in this case the restrictions on $p$ are $1<p\leq q$;
\item If $k=1$, we find $\frac{1}{q}-\frac{1}{n}\leq \frac{1}{p}\leq \frac{1}{q'}+\frac{1}{n}$.\ Thus, if $q\geq n$, this translates into $\frac{nq'}{n+q'}\leq p<\infty $, while if $\frac{n}{2}<q<n$, we then find $\frac{nq'}{n+q'}\leq p<\frac{nq}{n-q}$;
\item If $k=0$, then $\frac{1}{q}-\frac{2}{n}\leq \frac{1}{p}\leq \frac{1}{q'}$ holds for all $p\geq q'$, since $\frac{1}{q}-\frac{2}{n}<0$ by hypothesis.
\item If $q\geq 2$, then $q'\leq 2\leq q$ and $\frac{1}{q}+\frac{k-2}{n}\leq \frac{1}{q}\leq  \frac{1}{q'}+ \frac{k}{n}$ for all $k=0,1,2$. Therefore, when $q\geq 2$, we see that $p=q$ is always allowed.
\end{itemize}
\end{remark}
\begin{proof}
Let $u\in C^{\infty}$ be compactly supported in a coordinate neighbourhood $(U,\varphi)$ where $L$ acts as in (\ref{Holst.1}), and let $(U,\varphi,\rho,\eta)$ be part of a trivialisation of $E$ by coordinate charts. Then,
\begin{align*}
\tilde{\rho}^l\circ (\eta Lu)\circ\varphi_x^{-1}&= \sum_{i=1}^N\tilde{\rho}^l\circ (\eta L(\eta_iu))\circ\varphi^{-1},\\
&= \sum_{i=1}^N\sum_{|\alpha|\leq 2}\tilde{\rho}^l\circ (\eta A_{\alpha})\circ \partial_x^{\alpha}(\tilde{\rho}_i(\eta_iu)\circ\varphi^{-1}),\\
&=\sum_{i=1}^N\sum_{|\alpha|\leq 2}\sum_{j=1}^r\eta\circ\varphi^{-1} (A_{\alpha})^l_j\partial_x^{\alpha}(\eta_iu^j\circ\varphi^{-1}).
\end{align*}
By hypothesis, $\rho^a_i\circ (\eta_iu)\circ \varphi^{-1}_i\in W^{k,p}(\varphi_i(U_i))$ for each $i=1,\cdots,N$ and $a=1,\cdots,r$. In fact, $\rho^a_i\circ (\eta_iu)\circ \varphi^{-1}_i\in W_0^{k,p}(\varphi_i(U_i))$, and so, in a neighbourhood of $\mathrm{supp}(\eta_i u)\subset U_i\cap U$,
\begin{align*}
\varphi_i\circ\varphi^{-1}:\varphi(U_i\cap U)\to \varphi_i (U_i\cap U)
\end{align*}
is a bounded diffeomorphism, and therefore by \cite[Theorem 3.35]{Adams1st}
\begin{align*}
\mathcal{A}_i\doteq (\varphi_i\circ\varphi^{-1})^{*}:W^{k,p}(\varphi_i(U_i\cap U))&\to W^{k,p}(\varphi(U_i\cap U)),\\
f&\mapsto f\circ \varphi_i\circ\varphi^{-1}
\end{align*}
is a bounded invertible map. Therefore, since 
\begin{align*}
\mathrm{supp}(\rho^a_i\circ (\eta_iu)\circ \varphi^{-1}_i)\subset\subset \varphi_i(U_i\cap U),
\end{align*}
we have
\begin{align*}
\mathcal{A}_i(\rho^a_i\circ (\eta_iu)\circ \varphi^{-1}_i)=\rho^a_i\circ (\eta_iu)\circ \varphi^{-1}\in W_0^{k,p}(\varphi(U_i\cap U)).
\end{align*}
Let us then chose a cut-off function $\chi_i$ such that $\mathrm{supp}(\chi_i)\subset\subset U_i\cap U$ and $\chi_i\equiv 1$ in a neighbourhood of $\mathrm{supp}(\eta_iu)\subset\subset U_i\cap U$. Then,
\begin{align*}
\tilde{\rho}^l\circ (\eta A_{\alpha}\circ \varphi^{-1})\circ \partial_x^{\alpha}(\tilde{\rho}(\eta_iu)\circ\varphi^{-1})&=\left(\eta A_{\alpha}\circ \varphi^{-1}\right)^l_a\partial_x^{\alpha}(\tilde{\rho}^a(\eta_iu)\circ\varphi^{-1}),\\
&=\chi_i \left(\eta A_{\alpha}\circ \varphi^{-1}\right)^l_a\partial_x^{\alpha}(\mathcal{A}_i(\rho^a_i\circ (\eta_iu)\circ \varphi^{-1}_i))\\
&\in W_0^{|\alpha|,q}(\varphi(U\cap U_i))\otimes W_0^{k-|\alpha|,p}(\varphi(U\cap U_i))
\end{align*}
Therefore, we need to establish the embedding $W_0^{|\alpha|,q}(\varphi(U\cap U_i))\otimes W_0^{k-|\alpha|,p}(\varphi(U\cap U_i))\hookrightarrow W^{k-2,p}(\varphi(U\cap U_i))$. Appealing to Theorem \ref{BesselMultLocal}, one needs to guarantee 
\begin{align*}
&|\alpha|+k-|\alpha|=k\geq 0,\:\:\: (s_1+s_2\geq 0)\\
&|\alpha|,k-|\alpha|\geq k-2 \:\:\: \forall |\alpha|=0,1,2 \Longleftrightarrow k\leq 2, \:\:\:\ (s_1,s_2\geq s)
\end{align*}
and then from (\ref{MultiplicationConditionsBesell}), the following conditions guarantee the embedding:
\begin{align*}
&|\alpha|-k+2\geq n\left(\frac{1}{q} - \frac{1}{p} \right)\Longleftrightarrow \frac{1}{p}\geq \frac{1}{q}+\frac{k-2-|\alpha|}{n},\\
&2-|\alpha|\geq 0,\\
&|\alpha|+k-|\alpha|-k+2>\frac{n}{q}\Longleftrightarrow q>\frac{n}{2},\\
&|\alpha|+k-|\alpha|\geq n\left(\frac{1}{p}+\frac{1}{q}-1 \right)\Longleftrightarrow \frac{1}{p}+\frac{1}{q}\leq 1+\frac{k}{n}.
\end{align*}
Notice that the last three conditions above are satisfied by hypotheses, and for the first one to hold for all $0\leq |\alpha|\leq 2$, we must have
\begin{align*}
\frac{1}{p}\geq \frac{1}{q}+\frac{k-2}{n}\geq \frac{1}{q}+\frac{k-2-|\alpha|}{n},
\end{align*}
which is satisfied by hypothesis.
Therefore $W_0^{|\alpha|,q}\otimes W_0^{k-|\alpha|,p}\hookrightarrow W^{k-2,p}$ holds and thus,\footnote{Below, the constant $C$ appearing in the estimates may vary from line to line, keeping its independence on $u$.} 
\begin{align*}
\Vert \tilde{\rho}^l\circ (\eta A_{\alpha}\circ \varphi^{-1})\circ \partial_x^{\alpha}(\tilde{\rho}(\eta_iu)\circ&\varphi^{-1})\Vert_{W^{k-2,p}(\varphi(U_i\cap U))}\leq \\ 
&\leq C\Vert \chi_i \left(\eta A_{\alpha}\circ \varphi^{-1}\right)^l_a\Vert_{W^{|\alpha|,q}(\varphi(U_i\cap U))} \Vert \partial_x^{\alpha}(\mathcal{A}_i(\rho^a_i\circ (\eta_iu)\circ \varphi^{-1}_i))\Vert_{W^{k-|\alpha|,p}(\varphi(U_i\cap U))}\\
&\leq C\Vert \chi_i \left(\eta A_{\alpha}\circ \varphi^{-1}\right)^l_a\Vert_{W^{|\alpha|,q}(\varphi(U_i\cap U))} \Vert \mathcal{A}_i(\rho^a_i\circ (\eta_iu)\circ \varphi^{-1}_i)\Vert_{W^{k,p}(\varphi(U_i\cap U))}\\
&\leq C\Vert \chi_i \left(\eta A_{\alpha}\circ \varphi^{-1}\right)^l_a\Vert_{W^{|\alpha|,q}(\varphi(U_i\cap U))} \Vert \rho^a_i\circ (\eta_iu)\circ \varphi^{-1}_i\Vert_{W^{k,p}(\varphi_i(U_i\cap U))}\\
&\leq C\Vert \left(\eta A_{\alpha}\circ \varphi^{-1}\right)^l_a\Vert_{W^{|\alpha|,q}(\varphi(U))} \Vert \rho^a_i\circ (\eta_iu)\circ \varphi^{-1}_i\Vert_{W^{k,p}(\varphi_i(U_i))}
\end{align*}
Hence
\begin{align*}
\Vert \tilde{\rho}\circ (\eta Lu)\circ\varphi^{-1}\Vert_{W^{k-2,p}(\varphi(U))}&=\sum_{l=1}^r\Vert \tilde{\rho}^l\circ (\eta Lu)\circ\varphi^{-1}\Vert_{W^{k-2,p}(\varphi(U))},\\
&\leq \sum_{l=1}^r\sum_{i=1}^N\sum_{|\alpha|\leq 2} \Vert \tilde{\rho}^l\circ (\eta A_{\alpha})\circ \partial_x^{\alpha}(\tilde{\rho}(\eta_iu)\circ\varphi^{-1})\Vert_{W^{k-2,p}(\varphi(U))},\\
&\leq C\sum_{l=1}^r\sum_{|\alpha|\leq 2}\sum_{a= 1}^r\Vert \left(\eta A_{\alpha}\circ \varphi^{-1}\right)^l_a\Vert_{W^{|\alpha|,q}(\varphi(U))} \sum_i\Vert \rho^a_i\circ (\eta_iu)\circ \varphi^{-1}_i\Vert_{W^{k,p}(\varphi_i(U_i))},\\
&\leq C\sum_{|\alpha|\leq 2}\sum_{l=1}^r\sum_{a= 1}^r\Vert \left(\eta A_{\alpha}\circ \varphi^{-1}\right)^l_a\Vert_{W^{|\alpha|,q}(\varphi(U))} \Vert u\Vert_{W^{k,p}(E)},\\
&\leq C\sum_{|\alpha|\leq 2}\Vert \tilde{\rho}\circ\left(\eta A_{\alpha}\right)\circ \varphi^{-1}\Vert_{W^{|\alpha|,q}(\varphi(U))} \Vert u\Vert_{W^{k,p}(E)}
\end{align*}

Appealing to a partition of unity to localise an arbitrary section $u\in C^{\infty}$ into a sum $u=\sum_iu_i$ with $u_i$ supported in coordinate patches trivialising the bundle, then
\begin{align}\label{OpContinuity2ndOrder}
\begin{split}
\Vert Lu\Vert_{W^{k-2,p}(E)}&\leq\sum_{i=1}^N\Vert Lu_i\Vert_{W^{k-2,p}(E)}= \sum_{i=1}^N\sum_{j=1}^N\Vert \tilde{\rho}_j\circ (\eta_jLu_i)\circ\varphi_j^{-1}\Vert_{W^{k-2,p}(\varphi_j(U_j))},\\
&\leq \sum_{i,j=1}^N\sum_{|\alpha|\leq 2}C_{i,j,\alpha}\Vert \tilde{\rho}_j\circ \left(\eta_j A_{\alpha}\right)\circ \varphi_j^{-1}\Vert_{W^{|\alpha|,q}(\varphi(U))} \Vert u\Vert_{W^{k,p}(E)}\\
&\leq C\left(\sum_{i,j=1}^N\sum_{|\alpha|\leq 2}\Vert \tilde{\rho}_j\circ \left(\eta_j A_{\alpha}\right)\circ \varphi_j^{-1}\Vert_{W^{|\alpha|,q}(\varphi(U))} \right)\Vert u\Vert_{W^{k,p}(E)}
\end{split}
\end{align}
which establishes that
\begin{align*}
L:(C^{\infty}(M;E),\Vert \cdot\Vert_{W^{k,p}})\to W^{k-2,p}(M;E)
\end{align*}
is continuous under our hypotheses on $L$, $p$ and $k$. Finally, since $C^{\infty}$ is dense in $W^{k,p}(M;E)$, the general claim follows.

\end{proof}

Let us consider also first order operators, acting on smooth sections like 
\medskip
\begin{align}\label{1stOrderOp}
(Lu)\circ\varphi^{-1}=\sum_{|\alpha|\leq 1}A_{\alpha}\partial^{\alpha}(\tilde{\rho}\circ u\circ\varphi^{-1}), \text{ with } A^{\alpha}\in W_{loc}^{|\alpha|,q}(U,\mathbb{R}^{r\times r}),
\end{align}
where above $\alpha=(\alpha_1,\cdots,\alpha_n)$ denotes an arbitrary multi-index, $|\alpha|=\sum_{i}^n\alpha_i$, $\partial_{\alpha}=\partial_{(x^1)^{\alpha_1}\cdots (x^n)^{\alpha_n}}$ and $r$ is the dimension of the fibre of $E\to M$. Following \cite{MaxwellHolstRegularity}, we denote the set of such operators by $\mathcal{L}^1(W^{1,q})$. One can then establish the following lemma concerning the mapping properties of such an operator:
\begin{lem}\label{ContinuityPropsGeneralOps1stOrder}
Let $L$ be a differential operator of class $\mathcal{L}^1(W^{1,q})$. Then, given $1<p< \infty$ and an integer $0\leq k\leq 1$, $L$ can be extended as a bounded linear map
\begin{align}
L:W^{k,p}(E)\to W^{k-1,p}(E)
\end{align}
as long as $q>n$ and $\frac{1}{q}+\frac{k-1}{n}\leq \frac{1}{p}\leq 1-\frac{1}{q}+ \frac{k}{n}$.
\end{lem}
\begin{proof}
This result follows along the same lines as detailed in the proof Lemma \ref{ContinuityPropsGeneralOps} above, and thus we merely outline the proof. Letting $u\in C^{\infty}$ be compactly supported in a coordinate neighbourhood $(U,\varphi)$ where $L$ acts as in (\ref{1stOrderOp}), and let $(U,\varphi,\rho,\eta)$ be part of a trivialisation of $E$ by coordinate charts. Then,
\begin{align*}
\tilde{\rho}^l\circ (\eta Lu)\circ\varphi_x^{-1}
&=\sum_{i=1}^N\sum_{|\alpha|\leq 1}\sum_{j=1}^r\eta\circ\varphi^{-1} (A_{\alpha})^l_j\partial_x^{\alpha}(\eta_iu^j\circ\varphi^{-1}).
\end{align*}
By hypothesis, $\rho^a_i\circ (\eta_iu)\circ \varphi^{-1}_i\in W_0^{k,p}(\varphi_i(U_i))$ for each $i=1,\cdots,N$ and $a=1,\cdots,r$. Thus, appealing to the same arguments and notations as in the proof of Lemma \ref{ContinuityPropsGeneralOps}, we find that
\begin{align*}
\mathcal{A}_i(\rho^a_i\circ (\eta_iu)\circ \varphi^{-1}_i)=\rho^a_i\circ (\eta_iu)\circ \varphi^{-1}\in W_0^{k,p}(\varphi(U_i\cap U)).
\end{align*}
Choosing a cut-off function $\chi_i$ such that $\mathrm{supp}(\chi_i)\subset\subset U_i\cap U$ and $\chi_i\equiv 1$ in a neighbourhood of $\mathrm{supp}(\eta_iu)\subset\subset U_i\cap U$, we see that
\begin{align*}
\tilde{\rho}^l\circ (\eta A_{\alpha}\circ \varphi^{-1})\circ \partial_x^{\alpha}(\tilde{\rho}(\eta_iu)\circ\varphi^{-1})&=\chi_i \left(\eta A_{\alpha}\circ \varphi^{-1}\right)^l_a\partial_x^{\alpha}(\mathcal{A}_i(\rho^a_i\circ (\eta_iu)\circ \varphi^{-1}_i))\\
&\in W_0^{|\alpha|,q}(\varphi(U\cap U_i))\otimes W_0^{k-|\alpha|,p}(\varphi(U\cap U_i)),
\end{align*}
and, with similar computations to those in the proof of Lemma \ref{ContinuityPropsGeneralOps}, an appeal to Theorem \ref{BesselMultLocal} grants that $W_0^{|\alpha|,q}(\varphi(U\cap U_i))\otimes W_0^{k-|\alpha|,p}(\varphi(U\cap U_i))\hookrightarrow W^{k-1,p}(\varphi(U\cap U_i))$ under our conditions on $k,p$ and $q$. Then, proceeding just as in the last part of proof of Lemma \ref{ContinuityPropsGeneralOps}, we deduce that
\begin{align*}
L:(C^{\infty}(M;E),\Vert \cdot\Vert_{W^{k,p}})\to W^{k-1,p}(M;E)
\end{align*}
is continuous under our hypotheses on $L$, $p$ and $k$, and since $C^{\infty}$ is dense in $W^{k,p}(M;E)$, the general claim follows.
\end{proof}

\subsection{AE Manifolds}\label{AESection}

Let us now introduce some definitions and technical results concerning AE manifolds. First, we will consider manifolds $M^n$ which consist of a compact core $K$ such that $M\backslash K$ is the disjoint union of a finite number of open sets $U_i$, such that each $U_i$ is diffeomorphic to the complement of a closed ball in Euclidean space. Such manifolds, in part of the classic literature, are referred to as \emph{Euclidean at infinity} (see \cite{CB-C}). The diffeomorphisms $\Phi_{i}:U_i\subset M\to \mathbb{R}^n\backslash\overline{B}$ induce charts, which are referred to as \emph{end coordinate systems} and are said to provide a \emph{structure of infinity} \cite{BartnikMass}. 

On these model manifolds we want to control the behaviour of fields near infinity. This can be done in different ways. For instance, some authors opt to fix the end coordinate systems and impose decay rates for fields written in those coordinates.  Another common option is to introduce function spaces with weights adapted to our manifold structure which provide good controls of the asymptotic behaviour of the fields (see, for instance, \cite{CB-C,Cantor-SplittingTensors,BartnikMass,Maxwell1}). For our purposes, the latter option is best, since, as we will see below, we can also tailor such weighted spaces to have good analytic properties useful for our PDE analysis. Such spaces have been investigated for a long time by different authors, such as \cite{NirenbergWalker,McOwen,Lockhart,CB-C,BartnikMass}. In what follows, we will adopt the conventions given in \cite{BartnikMass} for the weight parameters, which has become the most common one in current literature. Let us start with the following definition on $\mathbb{R}^n$.

\begin{defn}[Weighted Spaces]\label{WeightedSobolevNormsDefs}
Let $E\rightarrow \mathbb{R}^n$ be tensor bundle over $\mathbb{R}^n$. The weighted Sobolev space $W^{k,p}_{\delta}$, with $k$ a non-negative integer, $1<p<\infty$ and $\delta\in\mathbb{R}$, of sections $u$ of $E$, is defined as the subset of $W^{k,p}_{loc}$ for which the norm 
\begin{align}\label{WeightedSobolevNorm}
\Vert u\Vert_{W^{k,p}_{\delta}(\mathbb{R}^n)}\doteq \sum_{|\alpha|\leq k}\Vert \sigma^{-\delta-\frac{n}{p}+|\alpha|}\partial^{\alpha}u \Vert_{L^p(\mathbb{R}^n)}
\end{align}
is finite, where $\sigma(x)\doteq (1+|x|^2)^{\frac{1}{2}}$ and $\alpha$ denotes an arbitrary multi-index.

Similarly, the weighted $C^{k}_{\delta}$-spaces are given by sections $u\in \Gamma(E)$, whose components are $k$-times continuously differentiable and which satisfy
\begin{align}\label{WeightedC^k-Norm}
\Vert u\Vert_{C^{k}_{\delta}}\doteq \sum_{|\alpha|\leq k}\sup_{x\in\mathbb{R}^n}\sigma^{-\delta+|\alpha|}|\partial^{\alpha}u(x)|<\infty.
\end{align}
\end{defn}

Below we collect the main properties of these spaces, which proofs can be found in the previously cited references.
\begin{theo}\label{AEWeightedEmbeedings}
Let $E\rightarrow \mathbb{R}^n$ be a tensor bundle as in Definition \ref{WeightedSobolevNormsDefs}. Then, the following continuous embeddings hold:
\begin{enumerate}
\item If $1< p\leq q<\infty $ and $\delta_2<\delta_1$, then $L^{q}_{\delta_2}\hookrightarrow L^p_{\delta_1}$;
\item If $kp<n$, then $W^{k,p}_{\delta}\hookrightarrow L^{q}_{\delta}$ for all $p\leq q\leq \frac{np}{n-kp}$;
\item If $kp=n$, then $W^{k,p}_{\delta}\hookrightarrow L^q_{\delta}$ for all $q\geq p$;
\item If $kp>n$, then $W^{k+l,p}_{\delta}\hookrightarrow C^l_{\delta}$ for any $l=0,1,2\cdots$;
\item For any given $\epsilon>0$ there is a constant $C_{\epsilon}>0$ such that, for all $u\in W^{2,p}_{\delta}$, $1<p <\infty$ the following inequality holds
\begin{align}
\Vert u\Vert_{W^{1,p}_{\delta}}\leq \epsilon\Vert u\Vert_{W^{2,p}_{\delta}} + C_{\epsilon}\Vert u\Vert_{L^{p}_{\delta}}.
\end{align} 
\item If $1<p\leq q<\infty $ and $k_1+k_2>\frac{n}{q}+k$ where $k_1,k_2\geq k$ are non-negative integers, then, we have a continuous multiplication property $W^{k_1,p}_{\delta_1}\otimes W^{k_2,q}_{\delta_2}\hookrightarrow W^{k,p}_{\delta}$ for any $\delta>\delta_1+\delta_2$. In particular, $W^{k,p}_{\delta}$ is an algebra under multiplication for $k>\frac{n}{p}$ and $\delta<0$.
\end{enumerate}
\end{theo}

Let us now consider a manifold $M^n$ Euclidean at infinity (recall that this only fixed the differential structure of the ends of $M$). Notice that we have a finite number of end charts, say $\{U_i,\Phi_i \}_{i=1}^{N_0}$, with $\Phi(U_i)\simeq \mathbb{R}^n\backslash \overline{B}$, and a finite number of coordinate charts covering the compact region $K$, say $\{U_i,\Phi_i\}_{i=N_0+1}^N$. We can consider a partition of unity $\{\eta_i\}_{i=1}^N$ subordinate to the coordinate  cover $\{U_i,\Phi_i\}_{i=1}^N$. Then, given a tensor bundle $E\xrightarrow{\pi} M$, consider the tuple $\{U_i,\Phi_i,\eta_i,\rho_i\}_{i=1}^N$, where are usual $\rho_i$ denote the trivialisations of the tensor bundle $E$ over $U_i$, so that we can define $W^{k,p}_{\delta}(M;E)$ to be the subset of $W^{k,p}_{loc}(M;E)$ such that 
\begin{align}\label{GlobalWeightedAENorm}
\Vert u\Vert_{W^{k,p}_{\delta}}&\doteq\sum_{i=1}^{N_0}\Vert {\Phi^{-1}_i}^{*}(\tilde{\rho}_i\circ \eta_i u)\Vert_{W^{k,p}_{\delta}(\mathbb{R}^n)} + \sum_{i=N_0+1}^N\Vert{\Phi^{-1}_i}^{*}(\tilde{\rho}_i\circ\eta_i u)\Vert_{W^{k,p}(U_i)}<\infty. 
\end{align}

We can now extend the embedding and multiplication properties to a general manifold $M^n$ Euclidean at infinity by an appeal to localisation of fields using a partition of unity, to obtain:

\begin{theo}\label{SobolevPorpsAE}
Let $M^n$ be a manifold Euclidean at infinity and $E\rightarrow M$ a tensor bundle over $M$. Then, all the properties of Theorem \ref{AEWeightedEmbeedings} hold for $W^{k,p}_{\delta}(E)$ under the same conditions stated in those theorems. 
\end{theo}

Finally, let us highlight that the same type of localisation properties prove that $C^{\infty}_0(M)$ is dense in $W^{k,p}_{\delta}$ for all $k\geq 0$, $1<p<\infty$ and $\delta\in \mathbb{R}$.

Now that we can measure the behaviour of fields at infinity appealing to Sobolev spaces, let us introduce the key concept of AE manifolds.

\begin{defn}[AE manifolds]\label{AE-Manifolds}
	We will say that a Riemannian manifold $(M^n,g)$ is asymptotically Euclidean if:
\begin{enumerate}
\item $M^n$ is Euclidean at infinity, with end charts $\{\Phi_i\}$;
\item $g\in W^{k,p}_{loc}$ for some $k>\frac{n}{p}$;
\item $g-\Phi^{*}_ie\in W^{k,p}_{\delta}(E_i)$ for some $\delta<0$ and all end charts $\Phi_i$,
\end{enumerate}
where, above, ``e" denotes the Euclidean metric on $\mathbb{R}^n$.
\end{defn}

\subsubsection{ADM definitions}\label{SectionADM}

Let us recall that an initial data set for GR is a tuple of the form $\mathcal{I}\doteq (M^n,g,K,\mu,S)$, where $(M^n,g)$ is an $n$-dimensional Riemannian manifold; $K\in \Gamma(T^0_2M)$ is a symmetric tensor field, while $\mu$ and $S$ are, respectively, a function and a 1-form which stand for the energy and momentum densities induced by physical sources. These quantities are constrained via the ECE, which are given on $M^n$ by\!
\begin{align}\label{constraints}
\begin{split}
R_g-|K|^2_{g}+(\mathrm{tr}_gK)^2&=2\mu, \\
\mathrm{div}_{g}\left(K-\mathrm{tr}_gK\: g \right)&=S.
\end{split}
\end{align}
In many situations, the above constraints stand as necessary and sufficient conditions for the initial data to admit a well-posed evolution problem. In this context, isolated gravitational systems are modelled by initial data sets which are asymptotic to the initial data of Minkowski's space-time $\mathbb{M}^{n+1}$, which is given by  $(\mathbb{R}^n,\delta,K=0,\mu=0,S=0)$, where $\mathbb{E}^n=(\mathbb{R}^n,\delta)$ denotes the usual Euclidean space. These kinds of initial data sets are typically referred to as AE initial data sets. Although in general within GR notions of conserved total energy and momenta are quite subtle, in the case of AE initial data sets, these are well-understood.
 
Let us consider 3-dimensional AE initial data satisfying the following asymptotic conditions:
\begin{align}\label{decaying.1}
g_{ij}=\delta_{ij}+O_k(|x|^{-\tau}),\; \; \; K_{ij}=O_{k-1}(|x|^{-\tau-1}),\; \; \; \mu,S_i=O_{k-2}(|x|^{-\rho}),
\end{align}
for some $\tau,\rho>0$, where the above expressions are meant to hold on some fixed asymptotic chart $\{x^i\}_{i=1}^3$. In particular, following \cite{CederbaumSakovich}, we will focus our attention on the case where $\frac{1}{2}<\tau\leq 1$ and $\rho>3$ and refer to these initial data sets as $C^2_{\nicefrac{1}{2}+\epsilon}$-AE, with $\epsilon\in (\frac{1}{2},1]$. This choice of decaying conditions is made to guarantee that the ADM energy and momentum of the initial data sets are well-defined \cite{Bartnik1}. Let us recall that, on some end $E_i\cong\mathbb{R}^3\backslash \overline{B_{R_0}(0)}$, these quantities are defined in asymptotic rectangular coordinates by:
\begin{align}\label{energy-momentum}
\begin{split}
E\doteq \frac{1}{16\pi}\lim_{r\rightarrow\infty}\int_{S_r}\left(\partial_ig_{ij}-\partial_jg_{ii} \right)\nu^jd\omega_r \; \; ; \; \;
P_i\doteq \frac{1}{8\pi}\lim_{r\rightarrow\infty}\int_{S_r}\pi_{ij}\nu^jd\omega_r,
\end{split}
\end{align}
where $S_r$ denotes an euclidean sphere of radius $r$ embedded in $E_i$; $\nu$ stands for the Euclidean outward pointing unit normal; $d\omega_r$ the Euclidean area measure induced on $S_r$ and $\pi_{ij}\doteq K_{ij}-\mathrm{tr}_gK\:g_{ij}$. These quantities have been shown to be well-defined under clear geometric conditions, namely $L^1$-integrability of the sources $\mu$ and $S$ (see, for instance, \cite{Bartnik1}). On the other hand, the angular momentum $J$ and center of mass $C_{B\acute{O}M}$ are typically defined as \cite{Corvino1,BOM}\footnote{B\'OM stands for Beig - \'O Murchadha, who introduced the expression $C_{B\acute{O}M}$ in \cite{BOM}.}
\begin{align}
J_i&\doteq \frac{1}{8\pi}\lim_{r\rightarrow\infty}\int_{S_r}\pi_{jk}Y_i^j\nu^kd\omega_r \label{angularmomentum},\\
C^k_{B\acute{O}M}&\doteq\frac{1}{16\pi E} \lim_{r\rightarrow\infty}\left(\int_{S_r}x^k\left(\partial_ig_{ij}-\partial_jg_{ii} \right)\nu^jd\omega_r - \int_{S_r}\left(g_{ik}\nu^i - g_{ii}\nu^k \right)d\omega_r \right)\label{BOM},
\end{align}
whenever the limits exist, and where we have denoted by $Y_j$ the basic rotation fields in $\mathbb{R}^3$.\footnote{For example $Y_1=x^2\frac{\partial}{\partial x^3}-x^3\frac{\partial}{\partial x^2}$.} In contrast to the energy-momentum of an AE initial data set, these quantities are not well-defined for general $C^2_{\nicefrac{1}{2}+\epsilon}$-vacuum initial data sets. In fact, certain asymptotic symmetry conditions are typically imposed on the initial data set in order to guarantee that the limits exist. These symmetry assumptions are called the Regge-Teitelboim conditions and (for vacuum initial data sets) their most common form is given by (see \cite{RT-parity,Corvino1,Huang1})\!
\begin{align}\label{RT}
\begin{split}
g^{odd}_{ij}=O_k(|x|^{-2}),\;\;\; K^{even}_{ij}=O_{k-1}(|x|^{-3}),
\end{split}
\end{align}
where the even/odd superscripts refer to the even/odd parts of the corresponding functions.

\section{Mapping properties of $\nabla^2$ and $\Delta_g$ on $S_2M$}\label{SectionRegularity}

In this section we analyse the mapping properties of $\nabla^2$ and $\Delta_g$ on $S_2M$ when acting on sections of limited Sobolev regularity. These will be useful in subsequent sections to allow for manipulations such as those presented in Proposition \ref{WeakCommutation}, known for smooth sections. Along these lines, the following lemma is related to Lemma \ref{DualIsomorphismHolst}, but for the case of metrics of limited regularity.

\begin{lem}\label{DualPairingLowRegLemma}
Let $(M^n,g)$ be a closed Riemannian manifold with $g\in W^{2,q}(M)$, $q>\frac{n}{2}$, and let $p\in (1,\infty)$ and $k\in\mathbb{N}_0\cap \left[0,2\right]$ satisfy $\frac{1}{q}-\frac{2-k}{n}\leq \frac{1}{p}\leq \frac{1}{q'}+\frac{2+k}{n}$. Then, the inner product
\begin{align}\label{DualPairingLowReg.1}
\langle u,v\rangle_{L^2(M,dV_g)}=\int_M\langle u,v\rangle_gdV_g
\end{align}
extends from $C^{\infty}(T_rM)\times C^{\infty}(T_rM)$ to a bilinear map 
\begin{align}\label{DualPairingLowReg.2}
\langle \cdot,\cdot\rangle_{(M,g)}: W^{-k,p'}(T_rM)\times W^{k,p}(T_rM)\to \mathbb{R}.
\end{align}
Furthermore, given $(u,v)\in W^{-k,p'}(T_rM)\times W^{k,p}(T_rM)$, there is a sequence $\{u_k\}\subset L^{p'}(T_lM)$ such that
\begin{align}\label{LpDensity}
\langle u,v\rangle_{(M,g)}=\lim_k\langle u_k,v\rangle_{L^2(M,dV_g)}
\end{align}
and this pairing induces an isomoprhism $W^{-k,p'}(T_rM)\cong (W^{k,p}(T_rM))'$ via
\begin{alignat}{3}\label{DualIsomorphLowReg}
\mathcal{S}_{k,p}:W^{-k,p'}(T_rM)&\to (W^{k,p}(T_rM))',\nonumber\\
u&\mapsto \mathcal{S}_{k,p}(u): &W^{k,p}(T_rM) &\to  \mathbb{R},\\
 &							   &v &\mapsto  [\mathcal{S}_{k,p}u](v)=\langle u,v\rangle_{(M,g)}\nonumber
\end{alignat}
Finally, if $v\in \Gamma(T_rM)$ is compactly supported in a coordinate chart $(U,\varphi)$, then
\begin{align*}
\langle u,v\rangle_{(M,g)}=\langle \sqrt{\mathrm{det}(g)} u^{\sharp}{}^{i_1\cdots i_r},v_{i_1\cdots i_r}\rangle_{(U,\delta)},
\end{align*}
where $\langle \cdot,\cdot\rangle_{(U,\delta)}$ denotes the usual paring in $\mathbb{R}^n$ between $W^{-k,p'}(U)$ and $W^{k,p}(U)$ induced by the Euclidean metric; while $\tilde{\rho}\circ v\circ\varphi^{-1}= v_{i_1\cdots i_r}\in W^{k,p}(\varphi(U))$ denote the components of $v$ with respect to the trivialisation $(U,\varphi,\rho)$, while $u^{\sharp}\in W^{-k,p'}(T^{r}M)$ denotes the tensor field obtained from $U$ by raising its indices with $g$, and $u^{\sharp}{}^{i_1\cdots i_r}$ denote the components of $u^{\sharp}$ with respect to the corresponding trivialisation $(U,\varphi,\rho^{\sharp})$. 
\end{lem}
\begin{proof}
From Theorem \ref{DualIsomorphismHolst}, fixing a smooth background metric $\bar{g}$ on $M$, we know that (\ref{DualPairingLowReg.2}) and (\ref{DualIsomorphLowReg}) hold for the pairing induced by $dV_{\bar{g}}$. Notice then that given $u,v\in C^{\infty}(T_rM)$,
\begin{align*}
\langle u,v\rangle_{L^2(M,dV_{g})}&=\int_M\langle u,v\rangle_{g}dV_{g}=\int_M  u^{\sharp}(v)dV_{g}=\int_M\langle \frac{\sqrt{\mathrm{det}(g)}}{\sqrt{\mathrm{det}(\bar{g})}}\bar{u^{\sharp}}^{\flat},v\rangle_{\bar{g}}dV_{\bar{g}},\\
&=\langle \frac{\sqrt{\mathrm{det}(g)}}{\sqrt{\mathrm{det}(\bar{g})}}\bar{u^{\sharp}}^{\flat},v\rangle_{L^2(M,dV_{\bar{g}})}
\end{align*}
where $u^{\sharp}$ denotes the tensor field obtained by raising the indices of $u$ with $g$, and thus, by Corollary \ref{ContractionsSobReg}, $u^{\sharp}\in W^{2,q}(T^lM)$, while $\bar{u^{\sharp}}^{\flat}\in W^{2,q}(T_lM)$ denotes the the tensor field on obtained by lowering the indices of $u^{\sharp}$ with $\bar{g}$.  Thus, again by Corollary \ref{ContractionsSobReg}, $\frac{\sqrt{\mathrm{det}(g)}}{\sqrt{\mathrm{det}(\bar{g})}}\bar{u^{\sharp}}^{\flat}\in W^{2,q}(T_lM)$. In fact, the map
\begin{align*}
\Phi:\left(C^{\infty}(T_lM),\Vert\cdot\Vert_{W^{k,p}}\right)&\to W^{2,q}(T_lM),\\
u&\mapsto \frac{\sqrt{\mathrm{det}(g)}}{\sqrt{\mathrm{det}(\bar{g})}}\bar{u^{\sharp}}^{\flat},
\end{align*}
is a bounded injective map, which extends to a bounded isomorphism:
\begin{align*}
\Phi_{k,p}:W^{k,p}(T_lM)&\to W^{k,p}(T_lM),\\
u&\mapsto \frac{\sqrt{\mathrm{det}(g)}}{\sqrt{\mathrm{det}(\bar{g})}}\bar{u^{\sharp}}^{\flat},
\end{align*}
for each $(k,p)$ such that $k\in\mathbb{Z}$ and $0\leq k\leq 2$ as long as $2>\frac{n}{q}$ and
\begin{align*}
2-k\geq n\left(\frac{1}{q}-\frac{1}{p}\right)\Longleftrightarrow \frac{1}{p}\geq \frac{1}{q}+\frac{k-2}{n}.
\end{align*}
In the case of negative exponents $k<0$, then the above holds if in addition we have $k\geq -2$ and
\begin{align*}
2+k\geq n\left(\frac{1}{q}+\frac{1}{p} -1\right)\Longleftrightarrow \frac{1}{p}\leq \frac{1}{q'}+\frac{2+k}{n}.
\end{align*}
Therefore, if $k=0,1,2$, $q>\frac{n}{2}$ and $ \frac{1}{q}-\frac{2-k}{n}\leq \frac{1}{p}\leq \frac{1}{q'}+\frac{2+k}{n}$, then notice that
\begin{align*}
1-\frac{1}{q'}-\frac{2+k}{n}\leq \frac{1}{p'}\leq 1-\frac{1}{q}+\frac{2-k}{n}\Longrightarrow \frac{1}{q}-\frac{2-(-k)}{n}\leq \frac{1}{p'}\leq \frac{1}{q'}+\frac{2+(-k)}{n}
\end{align*}
That is, the pair $(-k,p')$ satisfied the hypotheses required for $\Phi_{-k,p'}:W^{-k,p'}(T_lM)\to W^{-k,p'}(T_lM)$ to be a bounded isomorphism. Thus, given $(u,v)\in W^{-k,p'}(T_lM)\times W^{k,p}(T_lM)$,
\begin{align}\label{SobolevDualityLowReg.1}
\langle u,v\rangle_{(M,g)}\doteq \langle \Phi_{-k,p'}(u),v\rangle_{(M,\bar{g})}
\end{align}
agrees with (\ref{DualPairingLowReg.1}) for $(u,v)\in C^{\infty}(T_lM)\times C^{\infty}(T_lM)$, and we know from Theorem \ref{DualIsomorphismHolst},
\begin{align*}
\vert \langle u,v\rangle_{(M,g)}\vert\leq C\Vert \Phi_{-k,p'}(u)\Vert_{W^{-k,p'}}\Vert v\Vert_{W^{k,p}}\leq C'\Vert u\Vert_{W^{-k,p'}}\Vert v\Vert_{W^{k,p}} , \;\; \forall\; u,v\in C^{\infty}.
\end{align*}
Being $C^{\infty}$ dense in both $W^{-k,p'}(T_lM)$ and $W^{k,p}(T_lM)$, it follows that (\ref{SobolevDualityLowReg.1}) extends uniquely by continuity to a bounded bilinear map as in (\ref{DualPairingLowReg.2}). 
 Also, given $\phi\in L^{p'}$ and $v\in W^{k,p}$, by definition one has
\begin{align*}
\langle \phi,v\rangle_{(M,g)}=\lim_{k,j}\langle \phi_k,v_j\rangle_{(M,g)}=\lim_{k,j}\langle \phi_k,v_j\rangle_{L^2(M,dV_{g})}
\end{align*}
where $\{\phi_k\},\{v_k\}\subset C^{\infty}$ and $\phi_k\xrightarrow[]{L^{p'}} \phi$, $v_k\xrightarrow[]{W^{k,p}} v$. It then follows that
\begin{align*}
\vert \langle \phi,v\rangle_{L^2(M,dV_{g})} - \langle \phi_k,v_k\rangle_{L^2(M,dV_{g})}\vert&\leq \vert \langle \phi - \phi_k,v\rangle_{L^2(M,dV_{g})}\vert  + \vert \langle \phi,v-v_k\rangle_{L^2(M,dV_{g)}}\vert,\\
&\leq \Vert \phi-\phi_k\Vert_{L^{p'}}\Vert v\Vert_{L^p}  + \Vert  \phi\Vert_{L^{p'}}\Vert v-v_k\Vert_{L^p}
\end{align*}
which shows that, whenever $\phi\in L^{p'}$ and $v\in W^{k,p}$:
\begin{align*}
\langle \phi,v\rangle_{(M,g)}= \lim_{k,j}\langle \phi_k,v_j\rangle_{L^2(M,dV_{g})}=\langle \phi,v\rangle_{L^2(M,dV_{g})}.
\end{align*}

Therefore, given $u\in W^{-k,p'}(T_lM)$, $v\in W^{k,p}(T_lM)$ and a sequence $\{u_k\}\subset  C^{\infty}(T_lM)$ such that $u_k\xrightarrow[]{W^{-k,p'}}u$, we have 
\begin{align*}
\langle u,v\rangle_{(M,g)}=\lim_k\langle u_k,v\rangle_{(M,g)}=\lim_k\langle u_k,v\rangle_{L^2(M,dV_g)}
\end{align*}
which establishes (\ref{LpDensity}). Now using (\ref{SobolevDualityLowReg.1}) together with the map $S$ defined in Theorem \ref{DualIsomorphismHolst}, we have the isomorphism claim of (\ref{DualIsomorphLowReg}) defining $\mathcal{S}_{k,p}\doteq S_{k,p}\circ \Phi_{-k,p'}:W^{-k,p'}(T_lM)\to \left(W^{k,p}(T_lM)\right)'$, which is a composition of bounded linear isomorphisms, which provides a topological isomorphism $W^{-k,p'}(T_lM)\cong (W^{k,p}(T_lM))'$.

Finally, if $v\in W^{k,p}(T_lM)$ is compactly supported in a coordinate chart as described in the theorem, then, from (\ref{SobolevDualityLowReg.1}) we have a sequence $u_k\subset L^{p'}$ such that
\begin{align*}
\langle u,v\rangle_{(M,g)}&=\lim_k\langle u_k,v\rangle_{L^2(M,dV_g)}=\lim_k \int_{U}\langle  u_k,v\rangle_{g}dV_g=\lim_k \int_{\varphi(U)}\langle  u_k,v\rangle_{g}\sqrt{\mathrm{det}(g)}dx,\\
&=\lim_k\int_{\varphi(U)}  \sqrt{\mathrm{det}(g)}u_k^{\sharp}(v)dx=\langle \sqrt{\mathrm{det}(g)}{u^{\sharp}}^{i_1\cdots i_l},v_{i_1\cdots i_l}\rangle_{(U,\delta)}.
\end{align*}
\end{proof}

\medskip
Let us now continue by noticing that, given a Riemannian manifold $(M^n,g)$, $g\in W^{2,q}$, $q>\frac{n}{2}$, and a tensor field $u\in L^p(M,dV_g)$, from Lemma \ref{ContinuityPropsGeneralOps} we can make sense of covariant derivatives of $u$ up to second order, which are defined from its local coordinate form, where both partial derivatives as well as multiplication by Christoffel symbols are operations defined by duality. For instance, let us explicitly present the following proposition:
\begin{prop}\label{WeakCommutation}
Let $(M^n,g)$ be a Riemannian manifold, with $n\geq 3$, $g\in W^{2,q}$ for $q>\frac{n}{2}$. If $u\in L^p(S_2M)$, where $S_2M$ denotes the bundle of symmetric $(0,2)$-tensor fields and $q'\leq p$, then $\nabla^2u\in W^{-2,p}(T_4M)$ and also the following usual commutation of second derivatives holds, so that for any $u\in L^p(M)$ compactly supported in a coordinate system $(U,\varphi)$:
\begin{align}\label{Local2ndDerComm}
\nabla_j\nabla_iu_{ab}=\nabla_i\nabla_ju_{ab} - R^{k}_{aji}u_{kb} - R^{k}_{bji}u_{ak}
\end{align}
\end{prop}
\begin{proof}
Using Remark \ref{RemarkContinuityProps}, we first notice that exponents $q$ and $p$ satisfy the conditions of Lemma \ref{ContinuityPropsGeneralOps} in the case $k=0$. Thus, to establish $\nabla^2u\in W^{-2,p}(T_4M)$ appealing to Lemma \ref{ContinuityPropsGeneralOps}, we need to consider the local expression of $\nabla^2u$ for $u\in C^{\infty}$ and then check they satisfy the hypotheses of the lemma. Thus, noticing that in a given local coordinate chart $(U,\varphi)$ the action of the linear operator $\nabla^2$ on $u$ is given by:
\begin{align}\label{WeakHessian}
\nabla_j\nabla_iu_{ab}&=\partial_{ij}u_{ab} + A^{cdl}_{abij}\partial_cu_{dl}+B^{dl}_{abij}u_{dl}
\end{align}
where
\begin{align*}
A^{cdl}_{abij}&=- \left(\Gamma^d_{ia}\delta^c_j\delta^l_b + \Gamma^l_{ib}\delta^c_j\delta^d_a + \Gamma^c_{ji}\delta^d_a\delta^l_b + \Gamma^d_{ja}\delta^c_i\delta^l_b + \Gamma^l_{jb}\delta^c_i\delta_a^d\right),\\
B^{dl}_{abij}&=\left(\Gamma^c_{ja}\Gamma^d_{ic}\delta^l_b  + \Gamma^d_{ja}\Gamma^l_{ib}  + \Gamma^l_{jb}\Gamma^d_{ia} + \Gamma^c_{jb}\Gamma^l_{ic}\delta^d_a - \partial_j\Gamma^d_{ia}\delta^l_b - \partial_j\Gamma^l_{ib}\delta^d_a + \Gamma^c_{ji}\Gamma^d_{ca}\delta^l_b + \Gamma^c_{ji}\Gamma^l_{cb}\delta^d_a\right).
\end{align*}
Thus, if we show that the above operator has coefficients satisfying (\ref{Holst.1}), then Lemma \ref{ContinuityPropsGeneralOps} shows that $\nabla^2$ extends to a bounded linear map from $L^p\to W^{-2,p}$ as long as $q'\leq p$. Since the top order coefficient is constant, it is clearly in $W^{2,q}_{loc}$, and thus we must only check that $A^{cdl}_{abij}\in W^{1,q}_{loc}$ and $B^{dl}_{abij}\in L^q_{loc}$. For the first one of these, just notice that since $\Gamma^i_{jk}(g)\in W_{loc}^{1,q}(U)$, then $A^{cdl}_{abij}\in W_{loc}^{1,q}(U)$. 

Finally, concerning the zero order term in (\ref{WeakHessian}), we have $\partial\Gamma\in L_{loc}^q(U)$, so we need only show that products of the form $\Gamma_{\cdot}\Gamma\in L^q_{loc}(U)$. For that pick any function $\chi\in C^{\infty}_0(U)$ and then chose $\eta\in C^{\infty}_0(U)$ a cut-off function such that $\eta\equiv 1$ on $\mathrm{supp}(\chi)$, so that 
\begin{align*}
\chi\Gamma_{\cdot}\Gamma=\chi\Gamma_{\cdot}\eta\Gamma\in W^{1,q}(U)\otimes W^{1,q}(U)\hookrightarrow L^q(U)
\end{align*}
which shows that $\Gamma_{\cdot}\Gamma\in L^q_{loc}(U)$ and thus $\nabla^2:L^p(S_2M)\to W^{-2,p}(S_2M)$ continuously as a consequence of Lemma \ref{ContinuityPropsGeneralOps}. 

\medskip 
Concerning the commutation rule for second covariant derivatives, we need only notice that, even for $g\in W^{2,q}$, this is an algebraic result which follows from the local expression on $(U,\varphi)$ for any $u\in C^{\infty}_0(U)$:
\begin{align*}
\nabla_j\nabla_iu_{ab}
&=\partial_{ji}u_{ab} - \Gamma^l_{ia}\partial_ju_{lb} - \Gamma^l_{ja}\partial_iu_{lb} - \Gamma^l_{ib}\partial_ju_{al} - \Gamma^l_{ji}\partial_lu_{ab}  - \Gamma^l_{jb}\partial_iu_{al}\\
&+ \Gamma^l_{ja}\Gamma^k_{ib}u_{lk} + \Gamma^l_{jb}\Gamma^k_{ia}u_{kl} + \Gamma^l_{ji}\Gamma^k_{la}u_{kb} + \Gamma^l_{ji}\Gamma^k_{lb}u_{ak}  \\
&- \partial_j\Gamma^l_{ia}u_{lb} - \partial_j\Gamma^l_{ib}u_{al} + \Gamma^l_{jb}\Gamma^k_{il}u_{ak} + \Gamma^l_{ja}\Gamma^k_{il}u_{kb}   ,
\end{align*}
where the first two lines in the last expression are explicitly symmetric under the interchange of indices $i\longleftrightarrow j$. Therefore,
\begin{align}\label{Commutation.1}
\nabla_j\nabla_iu_{ab}-\nabla_i\nabla_ju_{ab}
&=-R^{k}_{aji}u_{kb} - R^{k}_{bji}u_{ak}, \:\: \forall \: u\in C^{\infty}_0(U).
\end{align}
Then, since $\mathrm{Riem}_g\in L^q$ and $u\in C^{\infty}$, from Corollary \ref{ContractionsSobReg} we know the contractions in the right-hand side of the the above expression satisfy 
\begin{align*}
\Vert (\mathrm{Riem}_g)_{\cdot}u\Vert_{W^{-2,p}}\leq C\Vert \mathrm{Riem}_g\Vert_{L^{q}}\Vert u\Vert_{L^{p}}, \:\: \forall \: u\in C^{\infty}_0(U), \: p\geq q'
\end{align*}
Since both sides in (\ref{Commutation.1}) extend by continuity to $u\in L^{p}(U)$, the result follows. 
\end{proof}

The above will be particularly useful in the next section, but it also serves as a warm-up for the next proposition, where we study precise mapping properties of the tensor Laplacian, given by the trace of the second covariant derivative, on symmetric $(0,2)$-tensor fields. 

\begin{prop}\label{LapContPropertiesProp}
Let $(M^n,g)$ be a Riemannian manifold with $g\in W^{2,q}(M)$, $q>\frac{n}{2}$. Let us consider the bundle of $(0,2)$-symmetric tensor fields, denoted by $S_2M$, and the natural induced covariant derivative $\nabla$ on it by $g$. Consider then the Laplacian
\begin{align}\label{Lap.0}
\begin{split}
\Delta_g:C^{\infty}(S_2M)&\to W^{1,q}(S_2M),\\
u&\mapsto \Delta_gu=g^{ij}\nabla_i\nabla_ju=\nabla_i(g^{ij}\nabla_ju)
\end{split}
\end{align}
Then, $\Delta_g\in \mathcal{L}^2(W^{2,q})$ and given $1<p\leq q$, (\ref{Lap.0}) extends to a continuous operator
\begin{align}\label{LaplacianContinuityLp}
\begin{split}
\Delta_g:W^{2,p}(S_2M)&\to L^{p}(S_2M),\\
u&\mapsto \Delta_gu=g^{ij}\nabla_i\nabla_ju=\nabla_i(g^{ij}\nabla_ju).
\end{split}
\end{align}
Furthermore, if $\frac{1}{q}-\frac{1}{n}\leq \frac{1}{p}\leq \frac{1}{q'}+\frac{1}{n}$, then (\ref{Lap.0}) extends to a continuous operator
\begin{align}\label{LaplacianContinuityWeak}
\begin{split}
\Delta_g:W^{1,p}(S_2M)&\to W^{-1,p}(S_2M),\\
u&\mapsto \Delta_gu=g^{ij}\nabla_i\nabla_ju=\nabla_i(g^{ij}\nabla_ju).
\end{split}
\end{align}
Finally, if $q'\leq p$, then (\ref{Lap.0}) extends to a continuous operator
\begin{align}\label{LaplacianContinuityWeak2}
\begin{split}
\Delta_g:L^{p}(S_2M)&\to W^{-2,p}(S_2M),\\
u&\mapsto \Delta_gu=g^{ij}\nabla_i\nabla_ju=\nabla_i(g^{ij}\nabla_ju).
\end{split}
\end{align}

In all of the above cases, the following local formula, valid from smooth sections, extends to (\ref{LaplacianContinuityLp}), (\ref{LaplacianContinuityWeak}) and (\ref{LaplacianContinuityWeak2}) for $u$ compactly supported in a coordinate domain:
\begin{align}\label{TensorLapLoc}
\nabla_i(g^{ij}\nabla_ju^{ab})=\frac{1}{\sqrt{\mathrm{det}(g)}}\partial_i\left(\sqrt{\mathrm{det}(g)}\nabla^{i}u^{ab} \right) + \Gamma^a_{il}\nabla^{i}u^{lb} + \Gamma^b_{il}\nabla^{i}u^{al},
\end{align}
\end{prop}
\begin{remark}

Recall from Remark \ref{RemarkContinuityProps}, that in the case of (\ref{LaplacianContinuityWeak}) the restrictions on $p$ are given by
\begin{itemize}
\item $p\geq \frac{nq'}{n+q'}$ if $q\geq n$;
\item $\frac{nq'}{n+q'}\leq p\leq \frac{nq}{n-q}$ if $\frac{n}{2}<q<n$.
\end{itemize}

\end{remark}
\begin{proof}
Let $(U,x^i)_{i=1}^n$ be a local coordinate system, $U$ a bounded open set with smooth boundary, and let $u\in C^{\infty}(S_2M)$ be compactly supported in $U$. Then, locally (\ref{Lap.0}) is given by
\begin{align}\label{LaplacianContinuity.1}
\Delta_gu_{ab}&=g^{ij}\partial_{ij}u_{ab} + \mathcal{A}_{ab}^{cdl}\partial_cu_{dl} + \mathcal{B}_{ab}^{dl}u_{dl},
\end{align}
with
\begin{align}\label{LaplacianContinuity.2}
\begin{split}
\mathcal{A}_{ab}^{cdl}=g^{ij}A^{cdl}_{abij},\:\: \mathcal{B}_{ab}^{dl}=g^{ij}B_{abij}^{dl},
\end{split}
\end{align}
where the functions $A^{cdl}_{abij}$ and $B_{abij}^{dl}$ are those given in (\ref{WeakHessian}). In order to establish (\ref{LaplacianContinuityLp}),(\ref{LaplacianContinuityWeak}) and (\ref{LaplacianContinuityWeak2}), we appeal to Lemma \ref{ContinuityPropsGeneralOps}. Thus, let us first establish that $\Delta_g\in \mathcal{L}^2(W^{2,q})$. For this we need to establish the following claim:

\begin{claim}\label{ClaimMult1}
The coefficients in (\ref{LaplacianContinuity.2}) satisfy $\mathcal{A}_{ab}^{cdl}\in W_{loc}^{1,q}$ and $\mathcal{B}_{ab}^{dl}\in L_{loc}^{q}$
\end{claim}
\begin{proof}
From the analysis after (\ref{WeakHessian}), we know that $A^{cdl}_{abij}\in W_{loc}^{1,q}(U)$ and $B_{abij}^{dl}\in L^q_{loc}(U)$, and hence the claim amounts to showing that $W_{loc}^{2,q}\otimes W_{loc}^{1,q}\hookrightarrow W_{loc}^{1,q}$ and $W_{loc}^{2,q}\otimes L_{loc}^q\hookrightarrow L_{loc}^q$, both following directly from Theorem \ref{BesselMultLocal} and Corollary \ref{ContractionsSobReg}, due to $q>\frac{n}{2}$.
\end{proof}

Given the above claim, since $g^{ij}\in W_{loc}^{2,q}$, then we see that $\Delta_g\in \mathcal{L}^2(W^{2,q})$ and therefore Lemma \ref{ContinuityPropsGeneralOps} establishes our continuity properties (\ref{LaplacianContinuityLp}), (\ref{LaplacianContinuityWeak}) and (\ref{LaplacianContinuityWeak2}) are granted as long as $\frac{1}{q}+\frac{k-2}{n}\leq \frac{1}{p}\leq \frac{1}{q'}+\frac{k}{n}$ for $k=2,1,0$ respectively. Using Remark \ref{RemarkContinuityProps}, we know that for $k=2$ this holds as long as $1<p\leq q$, while in the case of $k=0$, since $q>\frac{n}{2}$, this reduces to $p\geq q'$. In the case of $k=1$ the condition is optimally expressed as $\frac{1}{q}-\frac{1}{n}\leq \frac{1}{p}\leq \frac{1}{q'}+\frac{1}{n}$.


Finally, to establish (\ref{TensorLapLoc}), let us notice that locally one can write the \emph{contravariant} version of (\ref{LaplacianContinuity.1}) more compactly as
\begin{align}\label{TensorLapLoc2}
\nabla_i(g^{ij}\nabla_ju^{ab})&=\partial_i\nabla^iu^{ab} + \Gamma^i_{il}\nabla^lu^{ab} + \Gamma^a_{il}\nabla^iu^{lb} + \Gamma^b_{il}\nabla^iu^{al}.
\end{align}
Now, the following formula for the derivative of a determinant extends to $W^{2,q}$ metrics, $q>\frac{n}{2}$,\footnote{The composition lemma for Sobolev functions allows us to use the chain rule in the same way as in the smooth case and therefore standard proofs extend to this level of regularity.}
\begin{align*}
\frac{\partial_l\sqrt{\mathrm{det}(g)}}{\sqrt{\mathrm{det}(g)}}=\frac{1}{2\mathrm{det}(g)}\partial_l\mathrm{det}(g)=\frac{1}{2}g^{ij}\partial_lg_{ij}=\Gamma^i_{li}.
\end{align*}
Putting this together with (\ref{TensorLapLoc2}), proves (\ref{TensorLapLoc}).
\end{proof}

In Section \ref{SectionCompactification}, we will be interested in regularity theory for the operators (\ref{LaplacianContinuityLp}) and (\ref{LaplacianContinuityWeak}). Related results have been the subject of extensive analysis under a variety of regularity conditions on the coefficients, specially when the a priori solution belongs to some $W^{1,p}$-space, $1<p<\infty$ (see, for instance, \cite[Chapter 3]{TaylorToolsForPDEs}, \cite[Chapters 5 and 6]{Morrey},\cite[Chapter 5]{Taylor1},\cite[Chapter 8 and 9]{GilbargTrudinger} as classical references, mainly devoted to the scalar case). Nevertheless, for general operators locally given by (\ref{LaplacianContinuity.1}) with the associated low regularity coefficients, it is not easy to pin down a reference containing a regularity theory for a priori solutions below the $W^{1,p}$-threshold. Notably, we will require a regularity theory for (\ref{LaplacianContinuityWeak}) allowing for a priori solutions in suitable $L^p$-spaces, which is out of the scope of the above highlighted references and aligns with more recent developments, such as those of \cite{Brezis1,Regularity1,Regularity2} mentioned in the introduction. 

With the above in mind, we will mainly appeal to \cite[Proposition 2.20 and Theorem 2.21]{MaxwellHolstRegularity} for regularity statements related to (\ref{LaplacianContinuityWeak}), which accommodate all our requirements and allow for very weak a priori solutions in $L^p$-spaces. In this last reference, the authors have developed substantial analytic machinery to deal with local regularity of solutions to linear elliptic equations with low regularity coefficients in a variety of Sobolev-type spaces. We also highlight that \cite[Proposition 5.4]{TaylorBochnerRough} is general enough to address our required local regularity theory for (\ref{LaplacianContinuityWeak}) in the case of the scalar Laplace operator for a priori solutions in $L^{q'}(M)$. In the case of the tensor Laplacian analysed in (\ref{LaplacianContinuityWeak}), this last result can be adapted to provide a local regularity theory relying on (\ref{LaplacianContinuity.1}) and applying the results \cite[Proposition 5.4]{TaylorBochnerRough} component wise. Then, the associated results would be enough for our purposes as well, although they would require some bootstrap arguments, while \cite[Theorem 2.21]{MaxwellHolstRegularity} provides a more general and ready to use result. Along these lines, these refined regularity results can also be used to analyse some conformal deformation problems related to the Yamabe problem for rough metrics, which we present below in Theorem \ref{YamabeLowRegThm.1}.

\subsection{Conformal properties of low regularity closed Riemannian manifolds}\label{SectionConformalProps}

In this section we review a low regularity Yamabe-type classification which is of interest to us for upcoming sections, where we will be interested in analysing the controls that conformal objects provide to bootstrap regularity and decay properties of a metric. Although the scalar curvature alone cannot fully control the metric, low regularity Yamabe-type classifications point in this direction. Such classifications have been studied in the literature under a variety of hypotheses, and, to the best of our knowledge, the original reference is \cite{MaxwellRoughClosed}, establishing a classification for rough metrics in $H^{s,2}(M)$-spaces, with $s>\frac{n}{2}$. Although the full Yamabe problem remains open for rough metrics, this classification establishes the existence of a conformal metric with continuous scalar curvature and sign given by the corresponding Yamabe invariant.\footnote{For recent important advances related to the full resolution of this problem, we highlight \cite{AubinThmRough}, where T. Aubin's criteria for the Yamabe problem to be solved was extended to rough metrics, and \cite{AvalosCogoAbregoYamabe} where the case of $C^{1,\alpha}$-Yamabe positive metrics is treated in three dimensions.} Thus, moving within the conformal class allows one to improve the a priori regularity of the scalar curvature. In this setting, we are interested in results of the type of \cite[Proposition 3.3]{MaxwellRoughClosed}, but within the regularity class of metrics $g\in W^{2,p}$, $p>\frac{n}{2}$, and our objective will be to guarantee such a conformal improvement on the scalar curvature regularity. 

It is worth noticing that, in the literature associated with the Einstein constraint equations, even lower regularity classes than those introduced above have been considered and corresponding Yamabe-type classifications can be found, for instance, in \cite{HolstFarCMC,HolstLichCompact,HolstFarCMCWithBoundary}. Nevertheless, for instance in the case of \cite[Theorem 12]{HolstFarCMC}, the classification is obtained via an appeal to \cite[Theorem 11]{HolstFarCMC}, which stands as a more general version of Theorem \ref{YamabeLowRegThm.1} presented below. In the case of \cite[Theorem 11]{HolstFarCMC}, the bootstrap of $\psi\in W^{1,2}(M)$ to $W^{s,p}(M)$ is claimed to follow from  \cite[Corollary 5]{HolstFarCMC}, which itself follows from \cite[Lemma 32]{HolstFarCMC}. This last lemma in \cite{HolstFarCMC} has a subtle mistake which has been pointed out in the introduction to \cite{MaxwellHolstRegularity}, but which can be fixed by the regularity theory developed in \cite{MaxwellHolstRegularity} (see the discussion after Theorem \ref{YamabeLowRegThm.1}).
 
With the above motivations in mind, let us introduce some definitions related to the Yamabe problem, and thus consider a closed Riemannian manifold $(M^n,\gamma)$ with $n\geq3$, $\gamma\in W^{2,p}$ and $p>\frac{n}{2}$, and define
\begin{align*}
\begin{split}
\mathcal{A}:W^{1,2}(M)\times W^{1,2}(M)&\to \mathbb{R}\\
(\varphi_1,\varphi_2)&\mapsto \int_M\left(a_n\langle \nabla\varphi_1,\nabla\varphi_2\rangle_{\gamma} + R_{\gamma}\varphi_1\varphi_2\right)dV_{\gamma}
\end{split}
\end{align*}
where $a_n\doteq 4\frac{n-1}{n-2}$. Sobolev embeddings guarantee that the above is a well-defined bilinear functional on $W^{1,2}(M)$. Related to the associated quadratic form, for each $1\leq q\leq \frac{n}{n-2}$, we define:
\begin{align}\label{q-Yamabe.1}
J_{\gamma,q}(\varphi)\doteq \frac{\mathcal{A}(\varphi,\varphi)}{\Vert\varphi\Vert^{2}_{L^{2q}(M,dV_{\gamma})}} =\frac{\int_M\left(a_n|\nabla\varphi|^2_{\gamma} + R_{\gamma}\varphi^2 \right)dV_{\gamma}}{\left(\int_M\varphi^{2q}dV_{\gamma} \right)^{\frac{1}{q}}}.
\end{align}

As commented in the proof of \cite[Theorem 11]{HolstFarCMC}, the same ideas of \cite[Section 3]{MaxwellRoughClosed} can be followed verbatim to show that $J_{\gamma,q}$ are all bounded from below for any $1\leq q\leq \frac{n}{n-2}$, and we can therefore introduce the following notation for their infima:
\begin{align*}
\mathcal{Y}_{\gamma,q}\doteq \inf_{\underset{\varphi\not\equiv 0}{\varphi\in W^{1,2}}}J_{\gamma,q}(\varphi). 
\end{align*}
Let us now highlight the role played by the following two especial infima.
\begin{defn}
Let $(M^n,\gamma)$ be a closed Riemannian manifold with $\gamma\in W^{2,p}(M)$, $p>\frac{n}{2}$ and $n\geq 3$. Then, we denote the \emph{first eigenvalue of the conformal Laplacian} by
\begin{align}\label{1stEigenVal}
\lambda_{\gamma}\doteq \mathcal{Y}_{\gamma,1},
\end{align}
while we denote the \emph{Yamabe invariant} by
\begin{align}\label{YamabeInvariant}
\mathcal{Y}([\gamma])\doteq \mathcal{Y}_{\gamma,\frac{n}{n-2}}
\end{align}
\end{defn}

We can now present the main result of this section, which is key in the low-regularity Yamabe classification:
\begin{theo}\label{YamabeLowRegThm.1}
Let $(M^n,\gamma)$ be a closed Riemannian manifold with $\gamma\in W^{2,p}$, $p>\frac{n}{2}$, and $n\geq 3$. Then, there exists a $W^{2,p}$ function $\varphi>0$ such that
\begin{align}\label{YamabeEigenValue.1}
-a_n\Delta_{\gamma}\varphi + R_{\gamma}\varphi=\lambda_{\gamma}\varphi.
\end{align} 
In particular, in the conformal class $[\gamma]$ there is a metric $g\in W^{2,p}(M)$ such that 
\begin{align}\label{ScalCurvSign}
R_g=\lambda_{\gamma}\varphi^{-\frac{4}{n-2}}.
\end{align}
\end{theo}

As highlighted in the beginning of this section, the above theorem follows as a special case of \cite[Theorem 11]{HolstFarCMC} if one can justify the regularity claims to which this theorem appeals. Following the proof of this theorem, and thus working along the lines of \cite[Section 3]{MaxwellRoughClosed}, one can show the existence of a non-negative $W^{1,2}(M)$ weak solution $\varphi$ to (\ref{YamabeEigenValue.1}) by standard variational methods, which is indeed in $W^{2,p}(M)$ by an application of \cite[Theorem 2.21]{MaxwellHolstRegularity}, fixing the regularity issue in the proof of \cite[Theorem 11]{HolstFarCMC}. Then, $\varphi>0$ follows as argued in \cite[Theorem 11]{HolstFarCMC} due to an application of the weak Harnack inequality of \textcite[Theorem 5.2]{Trudinger1}. This establishes the properties stated for $\varphi$ solving (\ref{YamabeEigenValue.1}), while the scalar curvature statement (\ref{ScalCurvSign}) follows by considering the Riemannian metric $g=\varphi^{\frac{4}{n-2}}\gamma$ and appealing to the conformal covariance formula for the scalar curvature. Therefore, Theorem \ref{YamabeLowRegThm.1} is indeed established.

\begin{remark}
With the aid of Theorem \ref{YamabeLowRegThm.1} above, one can actually parallel the discussion of \cite[Proposition 3.3]{MaxwellRoughClosed} concerning the weak Yamabe classification in the regularity classes $g\in W^{2,p}$, $p>\frac{n}{2}$.
\end{remark}

\section{Compactification of AE 3-manifolds - improved regularity}\label{SectionCompactification}

In this section we are concerned with the study of the Sobolev regularity of a conformal compactification $(\hat{M}^3,\hat{g})$ of an AE manifold $(M^3,g)$, which, for simplicity, will be assumed to have only one end. From \cite{MaxwellDiltsYamabeAE} one knows that under quite general conditions such a compactification can be obtained,\footnote{For the sake of completeness, the relevant results are presented within Appendix \ref{AppendixMaxwell}.} but that the regularity of $\hat{g}$ around the point of compactification $p_{\infty}$ is in general limited and it is closely related to the rate of decay of $g$ at infinity. Also, from \cite{Herzlich1997}, one knows that improved regularity is related to the decay of the Cotton tensor $C_{g}$ near infinity. Along this section, we will establish that certain mild controls over $C_g$ actually guarantee improved regularity of $\hat{g}$ into a $W^{2,q}$-metric, $q>3$.

\begin{lem}\label{ApriopriCompactifiedRegularity}
Let $(M^3,g)$ be a smooth $W^{k,p}_{\tau}$-AE manifold, relative to a structure of infinity with coordinates $\{z^i\}_{i=1}^3$ and $p> 2$, $\tau\in (-1,-\frac{1}{2})$, $k\geq 2$, and let $\hat{M}$ be the one point compactification of $M$. Then, there is a conformal factor $\varphi$ such that
\begin{enumerate}
\item $\hat{g}\doteq \varphi^4g$ extends to a $W^{2,q_0}$ metric on $\hat{M}$, with $q_0>2$. Moreover, $\varphi,\hat{g}\in C^{\infty}(\hat{M}\backslash\{ p_{\infty}\})$;
\item $\varphi=|z|^{-1}\phi$ in a neighbourhood of $p_{\infty}$, where $\phi\in W^{2,q_0}(\hat{M})$, $\phi\in C^{\infty}(\hat{M}\backslash\{ p_{\infty}\})$ and one can furthermore take $\phi(p_{\infty})=1$. Moreover, the inversion $x(z)=\frac{z}{|z|^2}$ for the coordinates $\{z^i\}_{i=1}^3$ near infinity provides a coordinate system around $p_{\infty}$ in which $x(p_{\infty})=0$ and $\hat{g}(\partial_{x^i},\partial_{x^j})|_{0}=\delta_{ij}$;
\item $R_{\hat{g}}\in W^{2,q_0}(\hat{M})$. Also, the scalar curvature has definite sign, with $\mathrm{sign}(R_{\hat{g}})=\lambda_{\hat{g}}$.
\end{enumerate}
Furthermore, if the Cotton tensor satisfies $C_g\in L^{p_1}_{\sigma}(M,dV_g)$, for some $-6<\sigma<-3$ and $p_1=\frac{3}{6+\sigma}$, then:
\begin{enumerate}
\item[4.] $C_{\hat{g}}\in L^{p_1}(\hat{M})$; 
\item[5.] If moreover $-6<\sigma<-4$, then $C_{\hat{g}}\in L^{q_1}(\hat{M})$ for some $q_1> \frac{3}{2}$, and $\Delta_{\hat{g}}\mathrm{Ric}_{\hat{g}}\in W^{-1,q_2}(\hat{M})$ for some $\frac{3}{2}<q_2\leq q_1$.
\end{enumerate}
\end{lem}
\begin{remark}\label{ApriopriCompactifiedRegularityRemark}
Notice that the fourth item guarantees $C_{\hat{g}}\in L^{q_1}(\hat{M})$ for some $q_1>\frac{3}{2}$ as long as $p_1>\frac{3}{2}$, while $C_{g}\in L^{p_1}_{\sigma}(M)$ satisfies $p_1\doteq \frac{3}{6+\sigma}>\frac{3}{2}$ as long as $\sigma$ satisfies:
\begin{align*}
\frac{3}{6+\sigma}>\frac{3}{2}\Longleftrightarrow 2>6+\sigma\Longleftrightarrow \sigma<-4.
\end{align*}
\end{remark}

\begin{remark}\label{RemarkComparisonHerzlich}
Assume $|C_g|=O(|z|^{-5})$, then
\begin{align*}
\int_{\mathbb{R}^3\backslash\overline{B}}|C_g|^p|z|^{-\sigma p- 3}dz&\lesssim \int_{\mathbb{R}^3\backslash\overline{B}}|z|^{-5p}|z|^{-\sigma p - 3}dz=\int_{\mathbb{R}^3\backslash\overline{B}}|z|^{-p(5+\sigma)}|z|^{ - 3}dz\lesssim \int_{r>1}r^{-p(5+\sigma)}r^{ - 1}dr,
\end{align*}
which is finite for any $\sigma<-5$. That is, $|C_g|=O(|z|^{-5})$ implies $C_g\in L^p_{\sigma}(M)$ for all $p\geq 1$ and $\sigma<-5$. Thus, the decaying conditions used in \cite{Herzlich1997} are strictly contained within the hypotheses $C_g\in L^p_{\sigma}(M)$ for some $-6<\sigma<-3$ and $p=\frac{3}{6+\sigma}$.
\end{remark}
\begin{proof}
We first claim that there is some $\tau<\delta<-\frac{1}{2}$ such that $W^{2,p}_{\tau}\hookrightarrow W^{2,q_0}_{\delta}$, where $q_0\doteq \frac{3}{2+\delta}$. This is granted by Theorem \ref{SobolevPorpsAE} if we can guarantee that
\begin{align*}
\tau<\delta<-\frac{1}{2} \text{ and } 1<q_0=\frac{3}{2+\delta}\leq p.
\end{align*}
Noticing that
\begin{align*}
1<\frac{3}{2+\delta}\leq p\Longleftrightarrow \frac{3}{p}-2\leq \delta<1,
\end{align*}
we see that if there some $\delta$ satisfying
\begin{align}
-\frac{1}{2}>\delta>\max(\tau,\frac{3}{p}-2),
\end{align}
then the desired inclusion $W^{2,p}_{\tau}\hookrightarrow W^{2,q_0}_{\delta}$ follows. Since $\tau<-\frac{1}{2}$, this is granted provided that
\begin{align*}
\frac{3}{p}-2<-\frac{1}{2}\Longleftrightarrow \frac{3}{p}<\frac{3}{2} \Longleftrightarrow 2<p,
\end{align*}
which is satisfied by hypothesis.

The first item above now follows from Lemma \ref{LemmaMaxwellDilts} for $q_0=\frac{3}{2+\delta}$ as long as $q_0>\frac{3}{2}$ and $g$ is $W^{2,q_0}_{\delta}$-AE with $\delta>-2$. Noticing that
\begin{align*}
q_0> 2 \Longleftrightarrow \frac{3}{2+\delta}> 2 \Longleftrightarrow \frac{3}{2}>2+\delta\Longleftrightarrow -\frac{1}{2}> \delta,
\end{align*}
from the inclusion $W^{2,p}_{\tau}\hookrightarrow W^{2,q_0}_{\delta}$ established above we see that $\bar{g}\in W^{2,q_0}(\hat{M})$, with $2<q_0\leq p$ and $\hat{M}$ compact. For the remaining statements, it is important to notice that the metric $\bar{g}$ obtained above from Lemma \ref{LemmaMaxwellDilts} is of the form $\bar{g}=u^{4}g$, $u\in C^{\infty}(\hat{M}\backslash\{p_{\infty}\})$, and, in a neighbourhood of the point of compactification $p_{\infty}\in \hat{M}$, the conformal factor satisfies $u=|z|^{-1}$. Moreover, using the inverted coordinates $x^i=\frac{z^i}{|z|^2}$, we also know from Lemma \ref{LemmaMaxwellDilts} that $\bar{g}(\partial_{x^i},\partial_{x^j})|_{x=0}=\delta_{ij}$.

Now, from Theorem \ref{YamabeLowRegThm.1}, there is some element $\tilde{g}\in [\bar{g}]$, such that $\tilde{g}=\tilde{\phi}^4\bar{g}$, where $\tilde{\phi}>0$ is a $W^{2,q_0}$-solution of (\ref{YamabeEigenValue.1}), and the scalar curvature satisfies $R_{\tilde{g}}=\lambda_{\bar{g}}\tilde{\phi}^{-4}$. Notice that, due to Sobolev multiplication, $\tilde{g}\in W^{2,q_0}$. Also, Lemma \ref{CompositionLemma} actually shows that $\tilde{\phi}^{-4}\in W^{2,q_0}$, so we see that $R_{\tilde{g}}\in W^{2,q_0}(\hat{M})$. The regularity statement away of $p_{\infty}$ follows since $\tilde{g}$ constructed above is smooth away of $p_{\infty}$, and thus, given a cut-off function $\eta$ supported in a coordinate ball $B$ which does not meet $p_{\infty}$, since $\tilde{\phi}$ solves (\ref{YamabeEigenValue.1}), we have that
\begin{align*}
-a_3\Delta_{\bar{g}}(\eta\tilde{\phi})=\eta(\lambda_{\bar{g}}-R_{\bar{g}})\tilde{\phi} +\Delta_{\bar{g}}\eta \tilde{\phi} + 2\bar{g}(\nabla\eta,\nabla\tilde{\phi})\in W_0^{1,q_0}(B).
\end{align*}  
Since the coefficients of $\Delta_{\bar{g}}$ are smooth on $B$, from a classical elliptic regularity bootstrap we find $\eta\tilde{\phi}\in C^{\infty}(B)$, and thus $\tilde{\phi}\in C^{\infty}(\hat{M}\backslash\{p_{\infty}\})$. This, in turn, also implies that $\tilde{g}\in C^{\infty}(\hat{M}\backslash\{p_{\infty}\})$.

We have therefore found that $(\hat{M},\tilde{g})$ satisfies $\tilde{g},R_{\tilde{g}}\in W^{2,q_0}(\hat{M})$, for some $q_0>2$, and, keeping track of our conformal transformations:
\begin{align*}
\tilde{g}&=\tilde{\phi}^{4}\bar{g}=\tilde{\varphi}^4g, \text{ with } \tilde{\phi}\in C^{\infty}(\hat{M}\backslash\{p_{\infty}\}),\\
\tilde{\varphi}&= |z|^{-1}\tilde{\phi},\text{ in an neighbourhood of } p_{\infty} \text{ and } \tilde{\phi}\in W^{2,q_0}(\hat{M}).
\end{align*}

Now, to establish the second claim in the lemma, one can consider $\hat{g}=\tilde{\phi}^{-4}(p_{\infty})\tilde{g}$, so that
\begin{align*}
\hat{g}=(\tilde{\phi}^{-1}(p_{\infty})\tilde{\phi})^{4}\bar{g}=(\tilde{\phi}^{-1}(p_{\infty})\tilde{\varphi})^4g. 
\end{align*}
Setting $\phi\doteq \tilde{\phi}^{-1}(p_{\infty})\tilde{\phi}\in W^{2,q_0}(\hat{M})$ and $\varphi\doteq \tilde{\phi}^{-1}(p_{\infty})\tilde{\varphi}=u\phi$, the metric $\hat{g}$ now obeys the same properties as $\tilde{g}$ did, namely properties $1.$ and $3.$ in the lemma, plus the regularity properties in the second statement, and now additionally that $\phi(p_{\infty})=1$, establishing the second property as well, since this also implies 
\begin{align*}
\hat{g}(\partial_{x^i},\partial_{x^j})|_{x=0}=\phi^4(p_{\infty})\bar{g}(\partial_{x^i},\partial_{x^j})|_{x=0}=\delta_{ij}.
\end{align*}

Now, to analyse the behaviour of the Cotton tensor, recall that
\begin{align}\label{CottonTensor}
C_{ijk}(g)\doteq \nabla_k\mathrm{Ric}_{ij} - \nabla_j\mathrm{Ric}_{ik} + \frac{1}{4}(\nabla_jR_gg_{ik} - \nabla_kR_gg_{ij}).
\end{align}
Since this is a conformal invariant in three dimensions, $C_{ijk}(\hat{g})=C_{ijk}(g)$ on $\hat{M}\backslash\{p_{\infty}\}$. Our first objective is to show that $C_{\hat{g}}$ is given by an $L^{q}(\hat{M})$-field. Notice that, a priori, $C_{\hat{g}}\in \mathcal{D}'$ is given by a distribution, which is represented by a smooth field on $\hat{M}\backslash\{p_{\infty}\}$. 

\begin{claim}\label{ClaimCotton1}
The Cotton tensor $C_{\hat{g}}$ satisfies $C_{\hat{g}}\in W^{-1,q_0}(\hat{M})$.
\end{claim}
\begin{proof}
Since $R_{\hat{g}}\in W^{2,q_0}(\hat{M})$, from the expression 
\begin{align}\label{CottonTensor}
C_{ijk}(\hat{g})\doteq \hat{\nabla}_k{\mathrm{Ric}_{\hat{g}}}_{ij} - \hat{\nabla}_i{\mathrm{Ric}_{\hat{g}}}_{ik} + \frac{1}{4}(\hat{\nabla}_jR_{\hat{g}}\hat{g}_{ik} - \hat{\nabla}_kR_{\hat{g}}\hat{g}_{ij}),
\end{align}
we see that we only need to control the Ricci terms, since the scalar curvature ones are already in $W^{1,q_0}(\hat{M})$. A priori, we know that $\mathrm{Ric}_{\hat{g}}\in L^{q_0}(\hat{M})$ and to guarantee that $\hat{\nabla} \mathrm{Ric}_{\hat{g}}\in W^{-1,q_0}(T_3M)$, we can appeal to Lemma \ref{ContinuityPropsGeneralOps1stOrder}. For that, we need to check that the coefficients satisfy the hypotheses of (\ref{1stOrderOp}), with $A_{\alpha}\in W_{loc}^{|\alpha|,r}$ for $r>3$, and that $\frac{1}{r}+\frac{1}{q_0}\leq 1$. Noticing that the top order coefficients in $\hat{\nabla}_k{\mathrm{Ric}_{\hat{g}}}_{ij}$ are constant, they are clearly in $W^{1,r}_{loc}(U)$, for any $r$. The regularity of the zero order coefficients is given by that of $\Gamma^l_{ij}(\hat{g})\in W^{1,q_0}_{loc}(U)$. Since $\frac{3}{2}< q_0<3$, then $W^{1,q_0}_{loc}(U)\hookrightarrow L^{\frac{3q_0}{3-q_0}}_{loc}(U)$ with $\frac{3q_0}{3-q_0}>3$. Thus, in the notations of (\ref{1stOrderOp}), we can find some $r>n> q_0$, such that $A_{|\alpha|}\in W^{|\alpha|,r}_{loc}$. We now need to guarantee that $\frac{1}{r}+\frac{1}{q_0}\leq 1$, but since $r\geq q_0$ and $q_0\geq q_0'$, then
\begin{align*}
\frac{1}{r}+\frac{1}{q_0}\leq \frac{1}{q_0}+\frac{1}{q_0'}=1.
\end{align*}
Thus, we are under the hypotheses of Lemma \ref{ContinuityPropsGeneralOps1stOrder} and we can deduce that $\hat{\nabla} \mathrm{Ric}_{\hat{g}}\in W^{-1,q_0}(T_3M)$.
\end{proof}

From the above claim, we know that $C_{\hat{g}}\in W^{-1,q_0}(\hat{M})$ a priori, but also $C_{\hat{g}}\in C^{\infty}(\hat{M}\backslash\{p_{\infty}\})$. We intend to show that, if $C_g$ decays fast enough, $C_{\hat{g}}$ can actually be extended to an $L^q$-field on $\hat{M}$.
\begin{claim}
If the Cotton tensor satisfies $C_g\in L^{p_1}_{\sigma}(M,dV_g)$ for some $-6<\sigma< -3$ and $p_1= \frac{3}{6+\sigma}$, then $C_{\hat{g}}\in L^{p_{1}}(\hat{M})$.
\end{claim}
\begin{proof}
We need only to establish the result in a neighbourhood of $p_{\infty}$. Notice that on $\hat{M}\backslash\{p_{\infty}\}$
\begin{align*}
\vert C_{\hat{g}}\vert^2_{\hat{g}}=\hat{g}^{ij}\hat{g}^{kl}\hat{g}^{ab}{C_{\hat{g}}}_{ika}{C_{\hat{g}}}_{jlb}=\varphi^{-12}g^{ij}g^{kl}g^{ab}{C_{g}}_{ika}{C_{g}}_{jlb}=\varphi^{-12}\vert C_g\vert^2_{g}.
\end{align*}
Thus, defining $\tilde{C}\in \Gamma(T_3\hat{M})$ by $\tilde{C}\doteq C_{\hat{g}}$ on $\hat{M}\backslash\{p_{\infty}\}$ and (for instance) $\tilde{C}|_{p_{\infty}}=0$, let us show that $\tilde{C}\in L^{p_1}(\hat{M})$. By the above computation, $\vert \tilde{C}\vert_{\hat{g}}=\varphi^{-6}\vert C_g\vert_{g}$. Since near $p_{\infty}$ we know that $\varphi=|z|^{-1}\phi$, with $\phi\in W^{2,q_0}(\hat{M})\hookrightarrow C^{0}(\hat{M})$, then given $\epsilon>0$ sufficiently small:
\begin{align*}
\int_{B_{\epsilon}}\vert \eta \tilde{C}\vert^{p_{1}}_{\hat{g}}dV_{\hat{g}}(x)
&\lesssim \int_{\mathbb{R}^3\backslash B_{\epsilon^{-1}}}\eta\vert C_g\vert^{p_{1}}_{g}|z|^{6p_{1}-6}dV_{g}(z)\lesssim \Vert C_{g}\Vert^{p_1}_{L^{p_1}_{\sigma}}
\end{align*}
where $\eta\in C^{\infty}_{0}(B_{\epsilon}(p_{\infty}))$ is a cut-off function, $0\leq \eta\leq 1$, equal to one in a neighbourhood of $p_{\infty}$, and the last estimate holds because $6p_{1}-6=-3-\sigma p_1$, which follows from our hypothesis $p_1=\frac{3}{6+\sigma}$. Since by hypothesis $C_g\in L^{p_1}_{\sigma}(M)$, then the integral in the right hand side of the above expression is finite, and thus $\tilde{C}\in L^{p_1}(\hat{M})$. Now, we claim that $L^{p_1}(\hat{M})\hookrightarrow W^{-1,t}(\hat{M})$ for some $t(\sigma)>\frac{3}{2}$. This is equivalent to $W^{1,t'}(\hat{M})\hookrightarrow L^{p'_1}(\hat{M})$ for some $t'(\sigma)<3$, which then reduces to granting that $\frac{3t'}{3-t'}\geq p'_1$ for some $t'(\sigma)<3$ and fixed $p_1(\sigma)=\frac{3}{6+\sigma}$, $-6<\sigma< -3$.\footnote{Notice that such $p_1(\sigma)$ parametrise the interval $(1,\infty)$.} Being $1<p'_1(\sigma)<\infty$ fixed and finite, clearly there always exists some $t'=t'(p_1)<3$ and close enough to $3$ such that $\frac{3t'}{3-t'}\geq p'_1$, which grants the desired embedding. Thus, $\tilde{C}\in W^{-1,\min\{t,q_0\}}(\hat{M})$ where $\min\{t,q_0\}>\frac{3}{2}$ which, together with Claim \ref{ClaimCotton1}, shows that $C_{\hat{g}}-\tilde{C}\in W^{-1,\min\{t,q_0\}}(\hat{M})$. Since $\mathrm{supp}(C_{\hat{g}}-\tilde{C})$, if non-empty, equals $\{p_{\infty}\}$, Proposition \ref{PropOnePointSupport} implies $C_{\hat{g}}-\tilde{C}\equiv 0$ and thus $C_{\hat{g}}\in L^{p_1}(\hat{M})$.
\end{proof}

\medskip

The claim above thus establishes item $4$ in the lemma, and finally to establish the item $5$, we first consider the following claim:
\begin{claim}
If $C_{\hat{g}}\in L^{q_1}$ for some $\frac{3q'_0}{3+q'_0}\leq q_1\leq q_0$, then $\hat{\nabla}C_{\hat{g}}\in W^{-1,q}(\hat{M})$ for all $1<q\leq q_1$.
\end{claim}
\begin{proof}
In Claim \ref{ClaimCotton1} we have already seen that the operator $\hat{\nabla}$ acting on tensor fields satisfies the hypotheses of Lemma \ref{ContinuityPropsGeneralOps1stOrder}, with coefficients $A_{\alpha}\in W_{loc}^{|\alpha|,r}$ for $r\doteq \frac{3q_0}{3-q_0}>3$. Therefore, from Lemma \ref{ContinuityPropsGeneralOps1stOrder} one knows $C_{\hat{g}}\in W^{-1,q_1}$ as long as $\frac{1}{r}+\frac{1}{q_1}\leq 1$. Notice then that
\begin{align*}
\frac{1}{r}+\frac{1}{q_1}=\frac{3-q_0}{3q_0} +\frac{1}{q_1}=\frac{1}{q_0} - \frac{1}{3} + \frac{1}{q_1}\leq 1 \Longleftrightarrow \frac{1}{q_1}\leq \frac{1}{q'_0}+\frac{1}{3}\Longleftrightarrow q_1\geq \frac{3q'_0}{3+q_0'}.
\end{align*}

\end{proof}

Let us now assume $C_{\hat{g}}\in L^{q_1}$, $q_1>\frac{3}{2}$, and notice this implies $q_1>\frac{3q'_0}{3+q'_0}$, since $\frac{3}{2}>\frac{3q'_0}{3+q'_0}$ iff $3>q'_0$, which is equivalent to $q_0>\frac{3}{2}$. Thus, we can use the above claim and $W^{2,q_0}(\hat{M})\otimes W^{-1,q_1}(\hat{M})\hookrightarrow W^{-1,q_{1}}(\hat{M})$, to deduce that $\hat{g}^{kl}\hat{\nabla}_{l}{C_{\hat{g}}}_{ijk}\in W^{-1,q_1}(U)$ for any bounded coordinate neighbourhood. Furthermore, from Proposition \ref{WeakCommutation}, we know that $\hat{\nabla}^2\mathrm{Ric}_{\hat{g}}\in W^{-2,q_0}(S_2\hat{M})$, and since $W^{2,q_0}(\hat{M})\otimes W^{-2,q_0}(\hat{M})\hookrightarrow W^{-2,q_0}(\hat{M})$, then traces of $\hat{\nabla}^2\mathrm{Ric}_{\hat{g}}$ belong to $ W^{-2,q_0}(\hat{M})$ a priori.\footnote{Notice these last two claims require $q_0\geq 2$.} Along the lines of some of the computations done above, before moving further, let us establish the following claim, which concerns the Schur lemma in this level of regularity:
\begin{claim}\label{SchurLemma}
Given our closed $W^{2,q_0}$-Riemannian manifold $(\hat{M},\hat{g})$, it holds that $\mathrm{div}_{\hat{g}}\left(\mathrm{Ric}_{\hat{g}}-\frac{1}{2}R_{\hat{g}}\hat{g}\right)=0$.
\end{claim}
\begin{proof}
From previous analysis we know that $\mathrm{Ric}_{\hat{g}}\in L^{q_0}$, $\hat{\nabla}\mathrm{Ric}_{\hat{g}}\in W^{-1,q_0}$, $R_{\hat{g}}\in W^{2,q_0}$ and hence $V\doteq \mathrm{div}_{\hat{g}}\left(\mathrm{Ric}_{\hat{g}}-\frac{1}{2}R_{\hat{g}}\hat{g}\right)\in W^{-1,q_0}$, and actually $V=0$ on $\hat{M}\backslash\{p_{\infty}\}$. Thus, $\mathrm{supp}(V)$, if non-empty, equals $\{p_{\infty}\}$. But Proposition \ref{PropOnePointSupport} shows that any $W^{-1,s}(B_1(0))$-function with $s>\frac{3}{2}$ supported at a single point vanishes identically, and hence since $q_0>\frac{3}{2}$ we deduce $V\equiv 0$. 

\end{proof}

Since $R_{\hat{g}}\in W^{2,q_0}(M)$, then we also know that $\hat{\nabla}^2R_{\hat{g}}\in L^{q_0}(M)$ a priori, and we can make the following local computation, where the right-hand side is a priori understood (term by term) in $W^{-2,q_0}(M)$:
\begin{align*}
\hat{\nabla}^{k}{C_{\hat{g}}}_{ijk}&=\hat{\nabla}^{k}\hat{\nabla}_k{\mathrm{Ric}_{\hat{g}}}_{ij} - \hat{\nabla}^{k}\hat{\nabla}_j{\mathrm{Ric}_{\hat{g}}}_{ik} + \frac{1}{4}(\hat{\nabla}^{k}\hat{\nabla}_jR_{\hat{g}}\hat{g}_{ik} - \hat{\nabla}^{k}\hat{\nabla}_kR_{\hat{g}}\hat{g}_{ij}),\\
&=\Delta_{\hat{g}}{\mathrm{Ric}_{\hat{g}}}_{ij} - (\hat{\nabla}_j\hat{\nabla}^{k}{\mathrm{Ric}_{\hat{g}}}_{ik} - \hat{g}^{ka}\hat{R}^l_{iaj}{\mathrm{Ric}_{\hat{g}}}_{lk} -\hat{g}^{ka}\hat{R}^l_{kaj}{\mathrm{Ric}_{\hat{g}}}_{il}) + \frac{1}{4}(\hat{\nabla}_{i}\hat{\nabla}_jR_{\hat{g}} - \Delta_{\hat{g}}R_{\hat{g}}\hat{g}_{ij}),\\
&=\Delta_{\hat{g}}{\mathrm{Ric}_{\hat{g}}}_{ij} + \hat{R}_{liaj}{\mathrm{Ric}_{\hat{g}}}^{la} - {\mathrm{Ric}_{\hat{g}}}^l_{j}{\mathrm{Ric}_{\hat{g}}}_{il} - \frac{1}{2}\hat{\nabla}_j\hat{\nabla}_{i}{R_{\hat{g}}} + \frac{1}{4}(\hat{\nabla}_{i}\hat{\nabla}_jR_{\hat{g}} - \Delta_{\hat{g}}R_{\hat{g}}\hat{g}_{ij}),\\
&=\Delta_{\hat{g}}{\mathrm{Ric}_{\hat{g}}}_{ij} + \hat{R}_{liaj}{\mathrm{Ric}_{\hat{g}}}^{la} - {\mathrm{Ric}_{\hat{g}}}^l_{j}{\mathrm{Ric}_{\hat{g}}}_{il} -  \frac{1}{4}(\hat{\nabla}_{i}\hat{\nabla}_jR_{\hat{g}} + \Delta_{\hat{g}}R_{\hat{g}}\hat{g}_{ij}).
\end{align*}
Notice that in the second line we have used $\mathrm{Ric}_{\hat{g}}\in L^{q_0}(M)$, with $q_0>2$, so as to appeal to Proposition \ref{WeakCommutation} and apply the commutation rule (\ref{Local2ndDerComm}), while in the third line we have appealed to Claim \ref{SchurLemma} to rewrite $\mathrm{div}_g\mathrm{Ric}_{\hat{g}}=\frac{1}{2}dR_{\hat{g}}$. Therefore,
\begin{align}\label{LapRicciEstimate}
\Delta_{\hat{g}}{\mathrm{Ric}_{\hat{g}}}_{ij}&=-\hat{R}_{liaj}{\mathrm{Ric}_{\hat{g}}}^{la} + {\mathrm{Ric}_{\hat{g}}}^l_{j}{\mathrm{Ric}_{\hat{g}}}_{il} + \frac{1}{4}\underbrace{(\hat{\nabla}_{i}\hat{\nabla}_jR_{\hat{g}} + \Delta_{\hat{g}}R_{\hat{g}}\hat{g}_{ij})}_{\in L^{q_0}} + \underbrace{\hat{\nabla}^{k}{C_{\hat{g}}}_{ijk}}_{\in W^{-1,q_1}}.
\end{align}
Since $\hat{g}\in W^{2,q_0}$, $q_0>2$, then $|\mathrm{Ric}_{\hat{g}}|^2_{\hat{g}}\in L^{\frac{q_0}{2}}$, $\frac{q_0}{2}>1$. This also implies that $\mathrm{Riem}_{\hat{g}}\in L^{q_0}$ and hence $-{R_{\hat{g}}}_{liaj}{\mathrm{Ric}_{\hat{g}}}^{la}+{\mathrm{Ric}_{\hat{g}}}^l_{j}{\mathrm{Ric}_{\hat{g}}}_{il}\in L^{\frac{q_0}{2}}$. Let us now consider the following claim: 
\begin{claim}
If $q_0>2$ then $L^{\frac{q_0}{2}}(\hat{M})\hookrightarrow W^{-1,q}(\hat{M})$ for some $q> \frac{3}{2}$.
\end{claim}
\begin{proof}
First, since $W^{-1,q}(\hat{M})= (W^{1,q'}(\hat{M}))'$, then 
\begin{align*}
L^{\frac{q_0}{2}}(\hat{M})\hookrightarrow W^{-1,q}(\hat{M})\Longleftrightarrow W^{1,q'}(\hat{M})\hookrightarrow \left(L^{\frac{q_0}{2}}(\hat{M})\right)'\cong L^{\left(\frac{q_0}{2}\right)'}(\hat{M}).
\end{align*}
We first notice that 
\begin{align*}
q>\frac{3}{2}\Longleftrightarrow \frac{1}{q}=1-\frac{1}{q'}<\frac{2}{3}\Longleftrightarrow q'<3,
\end{align*}
and therefore we are only interested in analysing the embedding $W^{1,q'}(\hat{M})\hookrightarrow L^{\left(\frac{q_0}{2}\right)'}(\hat{M})$ for $q'<3$. Since $\frac{1}{\left(\frac{q_0}{2}\right)'}=1-\frac{2}{q_0}$ and  $W^{1,q'}(\hat{M})\hookrightarrow L^{\frac{3q'}{3-q'}}(\hat{M})$ whenever $q'<3$, we then need to satisfy:
\begin{align*}
\frac{3q'}{3-q'}\geq \left(\frac{q_0}{2}\right)'\Longleftrightarrow \frac{1}{q'}-\frac{1}{3}\leq 1-\frac{2}{q_0} \Longleftrightarrow \frac{1}{q}\geq \frac{2}{q_0}-\frac{1}{3}.
\end{align*}
Since $q_0>2$, then $\frac{2}{q_0}-\frac{1}{3}<\frac{2}{3}$ and then the interval $[\frac{2}{q_0}-\frac{1}{3},\frac{2}{3})$ is non-empty. Thus, any $q\in (1,\infty)$ such that   $\frac{1}{q}\in [\frac{2}{q_0}-\frac{1}{3},\frac{2}{3})$ satisfies
\begin{align*}
\frac{2}{q_0}-\frac{1}{3}\leq \frac{1}{q}<\frac{2}{3}.
\end{align*}
Hence, for such $q>\frac{3}{2}$, the embedding $W^{1,q'}(\hat{M})\hookrightarrow L^{\left(\frac{q_0}{2}\right)'}(\hat{M})$ holds, and therefore $L^{\frac{q_0}{2}}(\hat{M})\hookrightarrow W^{-1,q}(\hat{M})$.


\end{proof}

Putting the above claim together with (\ref{LapRicciEstimate}), if $q_1>\frac{3}{2}$, we find
\begin{align}
\Delta_{\hat{g}}{\mathrm{Ric}_{\hat{g}}}_{ij}\in W^{-1,q_2}, \text{ for some } q_2>\frac{3}{2},
\end{align}
where $\hat{g}\in W^{2,q_0}(\hat{M})$. 
\end{proof}

In order to use the above to improve the regularity of the metric $\hat{g}$ on $\hat{M}$, we will need the following result concerning the existence of harmonic coordinates for low regularity metrics:

\begin{theo}\label{HarmonicCoord}
Let $(M^n,g)$ be a Riemannian manifold with $g\in W^{1,q}_{loc}$ and $q>n$. Then, given a point $p\in M$ and a coordinate system $\{x^i\}_{i=1}^n$ around $p$, such that $x(p)=0$ and $g(\partial_{x^i},\partial_{x^j})|_{x=0}=\delta_{ij}$, there exists a harmonic coordinate system $(U,y)$ centred at $p$, \textit{i.e} $y(p)=0$, with coordinate functions $y^i\in W_{loc}^{2,q}(U)$. Finally, one can choose such coordinates satisfying $\frac{\partial y^i}{\partial x^j}(0)=\delta^i_j$.
\end{theo}
\begin{proof}
Start by fixing the coordinate system $\{x^i\}_{i=1}^n$ around the chosen point $p\in M$. Since under our conditions $g\in C_{loc}^{0,\gamma}(M)$, for some $\gamma\in (0,1)$, then the existence of a change of coordinates $y^i=y^i(x)$ for $C^{1,\alpha}$-harmonic coordinates $\{y^i\}_{i=1}^n$ in a neighbourhood of $p\in M$ such that $y(p)=0$ and $\alpha\in (0,\alpha)$ follows from \cite[Section 9 in Chapter 3]{TaylorToolsForPDEs}. In an even more general setting, this follows from \cite[Corollary 2.8]{Salo}. The exact $W^{2,q}$-regularity for these coordinate functions follows from the local regularity of $W^{1,p}$-solutions, $1<p\leq q$, to $\Delta_gu=0$ for a $W^{1,q}$-metric (see, for instance, \cite[Chapter 3, Proposition 1.12]{TaylorToolsForPDEs}).

Finally, notice that $A^{i}_j\doteq \frac{\partial y^i}{\partial x^j}(0)\in GL(\mathbb{R}^n)$. Therefore, if necessary, one can consider the change of coordinates $\bar{y}^i\doteq (A^{-1})^{i}_jy^j\in W_{loc}^{2,q}$, which are still harmonic and also satisfy
\begin{align*}
\bar{y}^i(0)&=0,\\
\frac{\partial\bar{y}^i}{\partial x^j}(0)&=(A^{-1})^{i}_l\frac{\partial y^l}{\partial x^j}(0)=\delta^i_j.
\end{align*}
\end{proof}

From the above general result and Lemma \ref{ApriopriCompactifiedRegularity}, we get the following:
\begin{cor}\label{HarmoniCoordCorollary}
Let $(M^3,g)$ be a smooth $W^{k,p}_{\tau}$-AE manifold, relative to a structure of infinity with coordinates $\{z^i\}_{i=1}^3$ and $p> 2$, $\tau\in (-1,-\frac{1}{2})$, $k\geq 2$, and let $(\hat{M}^3,\hat{g})$ be the conformal compactification obtained for it in Lemma \ref{ApriopriCompactifiedRegularity}. Then, around $p_{\infty}$ there is a $\hat{g}$-harmonic coordinate system $(U,y^i)_{i=1}^3$ with coordinate function $y^i=y^i(x)$, $y^i\in W^{3,q_0}_{loc}(U)$, where $x^i$ denote the inverted coordinates $x^i\doteq \frac{z^i}{|z|^2}$. Furthermore, we have that $\hat{g}(\partial_{y^i},\partial_{y^j})\vert_{y=0}=\delta_{ij}$ and the functions $y^j\in C^{\infty}(U\backslash \{p_{\infty}\})$ satisfy $y^i(p_{\infty})=0$ and  $\frac{\partial y^i}{\partial x^j}(0)=\delta^i_j$.
\end{cor}
\begin{proof}
From Lemma \ref{ApriopriCompactifiedRegularity}, we know that $\hat{g}\in W^{2,q_0}(\hat{M})$ for some $2<q_0<3$. Thus, appealing to the embedding $W^{2,q_0}(\hat{M})\hookrightarrow W^{1,\bar{q}}(\hat{M})$, we know that $\bar{q}>3$ due to $q_0>\frac{3}{2}$. Setting the a priori $x$-coordinates there to be our inverted coordinates $x^i=\frac{z^i}{|z|^2}$, we furthermore know from Lemma \ref{ApriopriCompactifiedRegularity} that $\hat{g}(\partial_{x^i},\partial_{x_j})|_{x=0}=\delta_{ij}$, and thus appealing to Theorem \ref{HarmonicCoord} we know that there exist a $\hat{g}$-harmonic coordinate system $\{y^i\}$, where the functions $y^i=y^i(x)\in W_{loc}^{2,\bar{q}}(U)$ and also satisfy $y^i(p_{\infty})=0$ and  $\frac{\partial y^i}{\partial x^j}(0)=\delta^i_j$. This furthermore implies that 
\begin{align*}
\hat{g}(\partial_{y^i},\partial_{y^j})\vert_{y=0}=\frac{\partial x^a}{\partial y^{i}}(0)\frac{\partial x^b}{\partial y^{j}}(0)\hat{g}(\partial_{x^a},\partial_{x^b})\vert_{x=0}=\hat{g}(\partial_{x^j},\partial_{x^i})\vert_{x=0}=\delta_{ij}.
\end{align*}
Thus, we are now only left with improving the regularity of the harmonic coordinates $y^i=y^i(x)$, both around $p_{\infty}$ and away of it. For the second of these statements, since each function $y^i$ satisfies
\begin{align}
\Delta_{\hat{g}}y^i(x)=0 \text{ on } U\backslash\{p_{\infty}\},
\end{align}
and $\hat{g}(\partial_{x^j},\partial_{x^i})\in C^{\infty}(U\backslash\{p_{\infty}\})$ by Lemma \ref{ApriopriCompactifiedRegularity}, then classical local elliptic regularity guarantees that $y^i\in C^{\infty}(U\backslash\{p_{\infty}\})$.

Finally, if necessary by shrinking $U$, we may consider it to be bounded with smooth boundary so that $\hat{g}(\partial_{x^j},\partial_{x^i})\in W^{2,q_0}(U)$, and thus concerning the $W^{3,q_0}_{loc}(U)$ regularity for the harmonic coordinates, we also have:
\begin{align}\label{HarmoCoordsBootstrap}
\hat{g}(\partial_{x^j},\partial_{x^i})\in W^{2,q_0}(U), y^i\in W^{2,\bar{q}}_{loc}(U) \text{ and } \Delta_{\hat{g}}y^i(x)=0 \text{ on } U.
\end{align}
Then, either by a direct application of \cite[Theorem 2.21]{MaxwellHolstRegularity} to (\ref{HarmoCoordsBootstrap}) or by an application of \cite[Proposition A.1(b)]{Salo} to the equation satisfied by $\partial_{x^l}y^i$ for any $l=1,\cdots,n$, we find that indeed $y^i\in W^{3,q_0}_{loc}(U)$.\footnote{In order to apply \cite[Theorem 2.21]{MaxwellHolstRegularity}, notice that $\Delta_{\hat{g}}\in \mathcal{L}^2_1(H^{2,q_0},U)$ according to \cite[Definition 2.4]{MaxwellHolstRegularity}.}
\end{proof}

Let us now highlight the following improved regularity for the conformally compactified $\mathrm{Ric}_{\hat{g}}$, which follows from Lemma \ref{ApriopriCompactifiedRegularity} and a bootstrap argument:
\begin{lem}\label{RicciBootstrapLemma}
Assume the same hypotheses as in Lemma \ref{ApriopriCompactifiedRegularity} and consider the metric $\hat{g}$ obtained in it on $\hat{M}$. If $C_g\in L^{p_1}_{\sigma}(M,dV_g)$ with $-6<\sigma<-4$ and $p_1=\frac{3}{6+\sigma}$, then $\mathrm{Ric}_{\hat{g}}\in L^{q}(\hat{M})$ for some $q>3$.
\end{lem}
\begin{proof}
Under our conditions, from Lemma \ref{ApriopriCompactifiedRegularity} we know that $\hat{g}\in W^{2,q_0}(\hat{M})$, $\mathrm{Ric}_{\hat{g}}\in L^{q_0}(\hat{M})$, with $2<q_0<3$, and $\Delta_{\hat{g}}\mathrm{Ric}_{\hat{g}}\in W^{-1,q_2}(\hat{M})$, with $q_2>\frac{3}{2}$. To bootstrap the regularity of $\mathrm{Ric}_{\hat{g}}$, we appeal to \cite[Theorem 2.21]{MaxwellHolstRegularity}. In the same setting as in the proof of Proposition \ref{LapContPropertiesProp}, in a given coordinate system $(U,x^i)_{i=1}^3$, $U\subset\subset M$ with smooth boundary, we may use (\ref{LaplacianContinuity.1}) to write $\Delta_{\hat{g}}\mathrm{Ric}_{\hat{g}}\in W^{-1,q_2}(\hat{M})$ locally, with coefficients satisfying the regularity properties of Claim \ref{ClaimMult1} in that proposition. This implies that $\Delta_{\hat{g}}\in \mathcal{L}^2(W^{2,q_0})$ over $U$, which corresponds to saying that $\Delta_{\hat{g}}\in \mathcal{L}^2_0(H^{2,q_0},U)$ in the notations of \cite[Theorem 2.21]{MaxwellHolstRegularity}. Moreover, $\mathrm{Ric}_{\hat{g}}\in L^{q_0'}(U)$ since $q_0>q_0'$, and $\Delta_{\hat{g}}\mathrm{Ric}_{\hat{g}}$ is locally in $H^{-1,q_2}(U)=H^{1-2,q_2}(U)$. Thus, in order to apply \cite[Theorem 2.21]{MaxwellHolstRegularity} one needs only check that the set of indices $(1,q_2)\in \mathcal{S}^2_0(H^{2,q_0})$ as introduced in \cite[Definition 2.7]{MaxwellHolstRegularity}. One may directly check this to be the case as long as $\frac{1}{q_0}-\frac{1}{3}\leq \frac{1}{q_2}\leq \frac{1}{q'_0}+\frac{1}{3}$. Noticing that $\frac{1}{q_2}<\frac{2}{3}$ and
\begin{align*}
\frac{2}{3}\leq \frac{1}{q'_0}+\frac{1}{3}=1-\frac{1}{q_0}+\frac{1}{3}\Longleftrightarrow \frac{1}{q_0}\leq \frac{2}{3},
\end{align*}
which holds, since actually $\frac{1}{q_0}<\frac{1}{2}$ and thus $\frac{1}{q_2}<\frac{2}{3}\leq \frac{1}{q'_0}+\frac{1}{3}$. Also notice that $\frac{1}{q_0}-\frac{1}{3}<\frac{1}{6}$ and thus $\frac{1}{q_0}-\frac{1}{3}\leq \frac{1}{q_2}$ is granted as long as $q_2\leq 6$. Therefore, taking $\frac{3}{2}<q_3=\min\{q_0,q_2\}<3$ we have $(1,q_3)\in \mathcal{S}^2_0(H^{2,q})$, and from \cite[Theorem 2.21]{MaxwellHolstRegularity} the following bootstrap follows:\footnote{As highlighted before starting Section \ref{SectionRegularity}, this regularity implication could also be derived from (\ref{LaplacianContinuity.1}) using \cite[Proposition 5.4]{TaylorBochnerRough} and a bootstrap argument.}
\begin{align}\label{RicciBootstrap}
\Delta_{\hat{g}}{\mathrm{Ric}_{\hat{g}}}\in W^{-1,q_3}(\hat{M})\Longrightarrow {\mathrm{Ric}_{\hat{g}}}\in W^{1,q_3}(\hat{M})\hookrightarrow L^{\frac{3q_3}{3-q_3}}(\hat{M}),
\end{align}
Noticing that 
\begin{align*}
\frac{3q_3}{3-q_3}>3\Longleftrightarrow q_3>3-q_3 \Longleftrightarrow q_3>\frac{3}{2},
\end{align*}
shows that $\mathrm{Ric}_{\hat{g}}\in L^{q}(\hat{M})$ for some $q>3$.
\end{proof}

Below we intend to exploit the extra regularity just gained for $\mathrm{Ric}_{\hat{g}}$ to improve the regularity of the metric $\hat{g}$ itself. Since the work of \cite{Deturck}, it is known that the regularity of the Ricci tensor provides optimal regularity for the metric by appealing to harmonic coordinates, where $\mathrm{Ric}_{\hat{g}}$ reads as a second order quasi-linear elliptic operator on the metric. If one can properly estimate the quadratic terms on $\partial\hat{g}$ appearing in such operator, then local elliptic regularity allows us to bootstrap the regularity of the metric, as long as $\hat{g}$ starts with some minimal $W^{1,q}$-regularity for $q>\mathrm{dim}(\hat{M})$.  

In our present case, there is a subtlety involved in the above procedure, which arises from Corollary \ref{HarmoniCoordCorollary}, which grants the existence of harmonic coordinates $y^i:U\to \mathbb{R}$ in a neighbourhood of $p_{\infty}$ which are $W^{3,q_0}$-functions of an a priori (smooth) coordinate system around $p_{\infty}$, which we have taken to be the inverted coordinates $x^{i}=\frac{z^i}{|z|^2}$, with $\{z^i\}_{i=1}^3$ defining the structure of infinity of $M^3$. Therefore, the coordinates $y^i$ are \emph{only} $W^{3,q_0}$-\emph{compatible} with the differentiable structure $\mathcal{D}_0(\hat{M})$ for $\hat{M}$ induced by this compactification. But we intend to use the $\{y^i\}$-coordinates constructed in Corollary \ref{HarmoniCoordCorollary} to write both $\hat{g}$ and $\mathrm{Ric}_{\hat{g}}$ locally and improve the regularity of $\hat{g}$ with respect to these coordinates $y^i$ as explained in the previous paragraph, and therefore also with respect to any smoothly related coordinate system. That is, if we define $\mathcal{D}_{\mathrm{Har}}(\hat{M})$ to be the maximal differentiable structure for $\hat{M}$ compatible with the harmonic coordinates $\{y^i\}$ around $p_{\infty}$ constructed from Corollary \ref{HarmoniCoordCorollary}, then the improved regularity will (a priori) be only with respect to coordinate systems in $\mathcal{D}_{\mathrm{Har}}(\hat{M})$. This procedure requires us to preserve the Sobolev regularity gained for $\hat{g}$ in Lemma \ref{ApriopriCompactifiedRegularity}, but, since the changes of coordinates (at least around $p_{\infty}$) between coordinate systems in $\mathcal{D}_{\mathrm{Har}}(\hat{M})$ and $\mathcal{D}_{0}(\hat{M})$ are given by $W^{3,q_0}$-functions, then their Jacobian will be $W^{2,q_0}$, which will allow us to preserve up to $W^{2,q_0}$-regularity when passing from  $\mathcal{D}_{0}(\hat{M})$ to $\mathcal{D}_{\mathrm{Har}}(\hat{M})$. This claim actually reduces to granting that a $W^{3,q_0}$-diffeomorphism has a $W^{3,q_0}_{loc}$ inverse and then using Theorem \ref{BesselMultLocal} (see Proposition \ref{InverseRegProp} and Corollary \ref{SummaryRegularityCor} below). 

In order to simplify the above problem, notice we intend to show that, for certain $1\leq p\leq \infty$ and $k\in\mathbb{N}_0$, if $f\in W^{k,p}(\mathcal{D}_0)$ then $f\in W^{k,p}(\mathcal{D}_{Har})$. For this one needs to check that for any coordinate system $\{V_{\alpha},y^{i}\}\subset \mathcal{D}_{Har}(\hat{M})$ and $\eta_{\alpha}\in C^{\infty}_0(V_{\alpha})$, $\eta_{\alpha}f\in W^{k,p}(V_{\alpha})$. If $p_{\infty}\not\in V_{\alpha}$, then $\{V_{\alpha},y^{i}\}\subset \mathcal{D}_{0}(\hat{M})$ and hence the result is true by hypothesis. Similarly, if $p_{\infty}\not\in \mathrm{supp}(\eta_{\alpha})$, then one may shrink $V_{\alpha}$ to a neighbourhood of $\mathrm{supp}(\eta_{\alpha})$ avoiding $p_{\infty}$ and then the result again follows by hypothesis. So, we need only consider the case $p_{\infty}\in V_{\alpha}\cap \mathrm{supp}(\eta_{\alpha})$. In that case, we only need to prove $f\in W^{k,p}(\mathcal{D}_{Har}(B_{r}(p_{\infty})))$ for some small ball $B_r(p_{\infty})$. With all this in mind, let us introduce the following notations to transfer the problem to the preservation of Sobolev regularity under coordinate transformations in $\mathbb{R}^n$, which is a classic topic, as can be seen from \cite[Theorem 3.41]{Adams}. 

Let $U_1,U_2\subset \mathbb{R}^n$ be two open sets and $\Phi:U_1\to U_2$ be a diffeomorphism. Consider then the operator $A:\mathcal{L}(U_1)\to \mathcal{L}(U_2)$, where $\mathcal{L}(\Omega)$ denotes the set of Lebesgue measurable functions on the domain $\Omega\subset \mathbb{R}^n$, given by:
\begin{align}\label{CoordinateChangeOp}
(A_{\Phi}u)(y)\doteq u(\Phi^{-1}(y)),
\end{align}
which is the operator inducing the coordinate change $y=\Phi(x)$, $x\in U_1$. In the case of the harmonic coordinates $y$ constructed in Corollary \ref{HarmoniCoordCorollary}, consider $\{U_1,\Phi_x\}$, $\Phi_x:U_1\to \Omega_1\subset\mathbb{R}^n$, as the chart in $\mathcal{D}_0(\hat{M})$ given by the inverted coordinates $\Phi_x(p)=x=\frac{z}{|z|^2}$, for $p$ in a neighbourhood of $p_{\infty}$. Then, maybe by making $\Omega_1$ smaller, Corollary \ref{HarmoniCoordCorollary} guarantees the existence of a coordinate change  $\Omega_1\to\Omega_2\subset\mathbb{R}^n$ given by functions $y^i=y^i(x)$. This coordinate change between subsets of $\mathbb{R}^n$ induces the new coordinate system on $U_1$ (again maybe making this a smaller nerighbourhood of $p_{\infty}$) $\Phi_y:U_1\to \Omega_2$, $\Phi^i_y(p)\doteq y^i\circ \Phi_x(p)$, and thus we still denote the coordinates by $y^i$. Noticing that the coordinate change $\Phi_y\circ\Phi^{-1}_x(x)=y(x)$, then, we may limit ourselves to establishing the invariance of $W^{k,p}$-spaces by the diffeomorphism $\Phi_y\circ\Phi^{-1}_x(x)=y(x)$, for certain given $1\leq p\leq \infty$ and $k\in\mathbb{N}_0$. The first step in this process concerns the a priori regularity of the inverse transformation $\Phi^{-1}(y)=x(y)$.

\begin{prop}\label{InverseRegProp}
Let $U_{i}\subset \mathbb{R}^n$, $i=1,2$, be two open sets with smooth boundary and let $\varphi:U_1\to U_2$ be a $W^{3,q}(U_1)$ diffeomorphism, $q>\frac{n}{2}$. Then, $\varphi^{-1}:U_2\to U_1$ is of class $W^{3,q}_{loc}(U_2)$.
\end{prop}
\begin{proof}
Firstly, since $W^{3,q}(U_1)\hookrightarrow C^{1,\alpha}(U_1)$ for $q>\frac{n}{2}$, then by the inverse function theorem $\varphi^{-1}\in C^{1}_{loc}(U_2)$ and, for any $y\in U_2$, we have
\begin{align}\label{IFT1}
D(\varphi^{-1})(y)=[(D\varphi)(\varphi^{-1}(y))]^{-1},
\end{align}
where in the right-hand side $[\: \cdot\: ]^{-1}$ stands for the inverse of the associated matrix. Take $V'_y\subset\subset U_2$ a neighbourhood of $y$ and $V_y\doteq \varphi^{-1}(V'_y)$, so that $\varphi^{-1}:V'_y\to V_y$ is a $C^1(V_y)$ map with $C^1$-bounded inverse. Let us now notice that $D\varphi\in W^{2,q}(V_y)\hookrightarrow W^{1,p}(V_y)$ for some $p\geq q$ and $p>n$, $A_{\varphi}(D\varphi)\in W^{1,p}(V'_y)$ by \cite[Theorem 3.41]{Adams}, where $A_{\varphi}$ is defined as in (\ref{CoordinateChangeOp}). Since the map $GL(\mathbb{R}^n)\to GL(\mathbb{R}^n)$, $B\mapsto B^{-1}$ is smooth, (\ref{IFT1}) implies that $D(\varphi^{-1})\in W^{1,p}(V_y)$. We may then consider an arbitrary weak derivative of (\ref{IFT1}), which will be of the form:
\begin{align*}
\partial_{y^iy^j}(\varphi^{-1})^l(y)=\partial_{y^i}(\varphi^{-1})^k(y)\big(\partial_{x^k}[(D\varphi)^{-1}]^l_j\big)(\varphi^{-1}(y)).
\end{align*}
Concerning the last factors, $(D\varphi)^{-1}:V_y\to GL(\mathbb{R}^n)$ is seen to be in $W^{2,q}(V_y)$ since this space is an algebra under multiplication for $q>\frac{n}{2}$ and an appeal to Lemma \ref{CompositionLemma}. Thus, $\partial_{x^k}[(D\varphi)^{-1}]^l_j\in W^{1,q}(V_y)$, which by another application of \cite[Theorem 3.41]{Adams} implies $A_{\varphi}\left(\partial_{x^k} [(D\varphi)^{-1}]^l_j\right)\in W^{1,q}(V'_y)$. That is, we find:
\begin{align}\label{IFT2}
\partial_{y^iy^j}(\varphi^{-1})^l\in W^{1,p}(V'_y)\otimes W^{1,q}(V'_y)\hookrightarrow W^{1,q}(V'_y), 
\end{align} 
where the last inclusion is justified by Theorem \ref{BesselMultLocal} due to $p\geq q$ and $p>n$. Finally, (\ref{IFT2}) implies that $\varphi^{-1}\in W^{3,q}(V_y)$ and establishes that claim.
\end{proof}

The above proposition, for instance, guarantees that the inverse transformation of the coordinate change $x\mapsto y(x)$ provided in Corollary \ref{HarmoniCoordCorollary} is locally a $W_{loc}^{3,q_0}$-map. We may now use this information to guarantee that the change of coordinates around $p_{\infty}$ induced by these harmonic coordinates indeed preserves the regularity controls obtained in Lemma \ref{ApriopriCompactifiedRegularity} for both $\hat{g}$ and $\mathrm{Ric}_{\hat{g}}$.
\begin{cor}\label{SummaryRegularityCor}
Let $(M^3,g)$ be an AE manifold satisfying the hypotheses of Lemma \ref{ApriopriCompactifiedRegularity} with $C_g\in L^{p_1}_{\sigma}(M,dV_g)$ with $-6<\sigma<-4$ and $p_1=\frac{3}{6+\sigma}$, and let $(\hat{M},\hat{g})$ be the conformal compactification obtained in it. Then, $\hat{g}\in W^{2,q_0}(\mathcal{D}_{\mathrm{Har}}(\hat{M}))\hookrightarrow W^{1,\bar{q}}(\mathcal{D}_{\mathrm{Har}}(\hat{M}))$ for $\bar{q}\doteq \frac{3q_0}{3-q_0}$ and $\mathrm{Ric}_{\hat{g}}\in L^{q}(\mathcal{D}_{\mathrm{Har}}(\hat{M}))$ for some $3<q\leq \bar{q}$. Furthermore, $\hat{g}\in C^{\infty}(\hat{M}\backslash\{p_{\infty}\},\mathcal{D}_{\mathrm{Har}}(\hat{M}))$.
\end{cor}
\begin{proof}
Let $\{U,x^i\}_{i=1}^3$ and $\{U,y^j\}_{j=1}^3$ be coordinate systems around $p_{\infty}$ as in the statement of Corollary \ref{HarmoniCoordCorollary}, defined on a bounded domain $U\subset \hat{M}$ with smooth boundary. Thus, $\{y^j\}_{j=1}^3$ stand for the harmonic coordinates giving rise to $\mathcal{D}_{\mathrm{Har}}(\hat{M})$. Then, the change of coordinates $\varphi: x(U)\to y(U)$, $x\mapsto y(x)$, is a $W^{3,q_{0}}(U)$-diffeomorphism by Corollary \ref{HarmoniCoordCorollary} with inverse $\varphi^{-1}\in W^{3,q_0}_{loc}(y(U))$ by Proposition \ref{InverseRegProp}. Thus, we may directly compute that:
\begin{align}\label{CompactMetrHarmCoords}
\hat{g}_y(\partial_{y^l},\partial_{y^k})&=\frac{\partial (\varphi^{-1})^a}{\partial y^l}(y)\frac{\partial (\varphi^{-1})^b}{\partial y^k}(y)\hat{g}_{\varphi^{-1}(y)}(\partial_{x^a},\partial_{x^b}),
\end{align}
Since $\frac{\partial (\varphi^{-1})^a}{\partial y^l}\in W^{2,q_0}_{loc}(y(U))$, which is an algebra under multiplication due to $q_0>\frac{3}{2}$, then
\begin{align*}
\frac{\partial (\varphi^{-1})^a}{\partial y^l}\frac{\partial (\varphi^{-1})^b}{\partial y^k}\in W^{2,q_0}_{loc}(y(U))
\end{align*}
Now, since $\hat{g}(\partial_{x^a},\partial_{x^b})\in W^{2,q_0}(x(U))\hookrightarrow W^{1,p}(x(U))$ for some $p>3> q_0$, then $A_{\varphi}\left( \hat{g}(\partial_{x^a},\partial_{x^b})\right)\in W^{1,p}_{loc}(y(U))$ by \cite[Theorem 3.41]{Adams}. Therefore, just as in Proposition  \ref{InverseRegProp}, by chain rule we see that
\begin{align*}
\partial_{y^c}A_{\varphi}\left( \hat{g}(\partial_{x^a},\partial_{x^b})\right)(y)=\frac{\partial (\varphi^{-1})^d}{\partial y^c}(y)A_{\varphi}\left(\partial_{x^d}\left(\hat{g}(\partial_{x^a},\partial_{x^b})\right)\right)(y).
\end{align*}
Since $\partial_{x^d}\left(\hat{g}(\partial_{x^a},\partial_{x^b})\right)\in W^{1,q_0}(x(U))$, then $A_{\varphi}\left(\partial_{x^d}\left(\hat{g}(\partial_{x^a},\partial_{x^b})\right)\right)\in W^{1,q_0}_{loc}(y(U))$ again by \cite[Theorem 3.41]{Adams} and thus, for all $c=1,\cdots,n$:
\begin{align*}
\partial_{y^c}A_{\varphi}\left( \hat{g}(\partial_{x^a},\partial_{x^b})\right)\in W^{2,q_0}_{loc}(y(U))\otimes  W^{1,q_0}_{loc}(y(U))\hookrightarrow  W^{1,q_0}_{loc}(y(U))\Longrightarrow A_{\varphi}\left( \hat{g}(\partial_{x^a},\partial_{x^b})\right)\in W^{2,q_0}_{loc}(y(U)),
\end{align*}
which follows by an appeal to Theorem \ref{BesselMultLocal} and $q_0>\frac{3}{2}$. Using all this information in (\ref{CompactMetrHarmCoords}), we find that the right-hand side belongs to $W^{2,q_0}_{loc}(y(U))$ since this space is an algebra under multiplication. 

The above analysis grants that $\hat{g}(\partial_{y^l},\partial_{y^k})\in W^{2,q_0}_{loc}(y(U))$ around $p_{\infty}$, and from Corollary \ref{HarmoniCoordCorollary} we know that the harmonic coordinates $y^i$ are $C^{\infty}$-compatible with $\mathcal{D}_{0}(\hat{M}\backslash\{p_{\infty}\})$ and therefore the Jacobians of the transformation $y^j=y^j(x)$ are smooth away of $p_{\infty}$, with $x$ a coordinate system in $\mathcal{D}_{0}(\hat{M})$. Thus, since $\hat{g}\in C^{\infty}(\hat{M}\backslash\{p_{\infty}\})$ can be checked for coordinates in $\mathcal{D}_{0}(\hat{M})$, the same follows with coordinates in $\mathcal{D}_{\mathrm{Har}}(\hat{M})$, and we conclude that $\hat{g}\in W^{2,q_0}(\mathcal{D}_{\mathrm{Har}}(\hat{M}))$. 

Finally, a similar analysis to the one above for (\ref{CompactMetrHarmCoords}) can be done for the $\mathrm{Ric}_{\hat{g}}$ leading one to
\begin{align*}
\mathrm{Ric}_{\hat{g}}(\partial_{y^l},\partial_{y^k})\in W^{2,q_0}_{loc}(y(U)) \otimes L^q_{loc}(y(U))\hookrightarrow L^q_{loc}(y(U)),
\end{align*}
again by an appeal to Theorem \ref{BesselMultLocal} and $q_0>\frac{3}{2}$, and then we conclude $\mathrm{Ric}_{\hat{g}}\in L^q(\mathcal{D}_{\mathrm{Har}}(\hat{M}))$ with the same arguments as those given above for $\hat{g}$.
\end{proof}

Having established above that, in harmonic coordinates, we have $\hat{g}\in W^{1,q}_{loc}$ and $\mathrm{Ric}_{\hat{g}}\in L^{q}_{loc}$ for some $q>3$ around $p_{\infty}$, proves to be sufficient to bootstrap the Ricci regularity to a $W^{2,q}$-regularity statement for $\hat{g}$. Actually, according to rather subtle regularity results, this could be established under even weaker conditions, for instance appealing to \cite[Chapter 14, Proposition 4.10]{Taylor3}. Below, we first state a more elementary regularity result, which is enough to provide a self-contained short proof of the bootstrap for $\hat{g}$ in our setting.
\begin{prop}\label{RegularityGT}
Let $\Omega\subset \mathbb{R}^n$ be a bounded domain with smooth boundary and $g\in W^{1,q}(\Omega)$ a Riemannian metric, with $q>n$. Assume that $\{y^i\}_{i=1}^n$ are $g$-harmonic coordinates, and that $\psi\in W^{1,p}(\Omega)$, $p>n$, satisfies the equation
\begin{align}\label{RegularityGT.1}
g^{ij}\partial_{y^iy^j}\psi=f\in L^{\frac{q}{2}}(\Omega).
\end{align}
Then $\psi\in W_{loc}^{2,\frac{q}{2}}(\Omega)$.
\end{prop}

The above result follows by a direct application of existence and uniqueness results for equations of the type of (\ref{RegularityGT.1}) which can be found, for instance, in \cite[Corollary 9.18 and Theorem 8.30]{GilbargTrudinger}, and can also be derived from \cite[Theorem 2.21]{MaxwellHolstRegularity}. We are now in a position to prove the main result of this section:

\begin{theo}\label{MetricBootstrapThm}
Let $(M^3,g)$ be a smooth $W^{k,p}_{\tau}$-AE manifold, relative to a structure of infinity with coordinates $\{z^i\}_{i=1}^3$ and $p> 2$, $\tau\in (-1,-\frac{1}{2})$, $k\geq 2$, and let $\hat{M}$ be the one point compactification of $M$. If $C_g\in L^{p_1}_{\sigma}(M,dV_g)$ with $-6<\sigma<-4$ and $p_1=\frac{3}{6+\sigma}$, then $(M^3,g)$ can be conformally compactified into $(\hat{M},\hat{g})$, and $\hat{M}$ can be equipped with a preferred differentiable structure $\mathcal{D}_{\mathrm{Har}}(\hat{M})$ which is $W^{3,q_0}$-compatible with the differentiable structure provided by the inverted coordinates $x=\frac{z}{|z|^2}$, such that $\hat{g}\in W^{2,q}(\mathcal{D}_{\mathrm{Har}}(\hat{M}))$ for some $q>3$. In particular $\hat{g}\in C^{1,\alpha}(\mathcal{D}_{\mathrm{Har}}(\hat{M}))$ for some $\alpha\in (0,1)$ and the following properties holds at $p_{\infty}$:
\begin{enumerate}
\item $y^i(p_{\infty})=0$;
\item $\frac{\partial y^i}{\partial x^j}(0)=\delta^i_{j}$;
\item $\hat{g}(\partial_{y^i},\partial_{y^j})\vert_{0}=\delta_{ij}$,
\end{enumerate}
where the coordinates $\{y^i\}_{i=1}^3$ above are the ones constructed in Corollary \ref{HarmoniCoordCorollary}.
\end{theo}
\begin{proof}
First, from Corollary \ref{SummaryRegularityCor} we know that $\hat{g}\in W^{1,\bar{q}}(\mathcal{D}_{\mathrm{Har}}(\hat{M}))$ with $\bar{q}=\frac{3q_0}{3-q_0}>3$. Since we need only check this improved regularity statement around $p_{\infty}$, invoking Corollary \ref{HarmoniCoordCorollary} we can consider $\{U,y^i\}_{i=1}^{3}\in \mathcal{D}_{\mathrm{Har}}(\hat{M})$ to be local harmonic coordinates with $U$ a bounded neighbourhood of $p_{\infty}$. Writing down the Ricci tensor $\mathrm{Ric}_{\hat{g}}$ in such harmonic coordinates around $p_{\infty}$, we find that:
\begin{align}\label{MetricBootstrap.1}
{\mathrm{Ric}_{\hat{g}}}_{ab}=\hat{g}^{ij}\partial_{ij}\hat{g}_{ab} + f_{ab}(\hat{g},\partial\hat{g}),
\end{align}
where $f_{ab}(\hat{g},\partial\hat{g})$ is quadratic on $\partial \hat{g}$. If necessary, considering an open set with smooth boundary $\Omega\subset\subset U$, we have $\partial \hat{g}\in L^{\bar{q}}(\Omega)$, and thus $f_{ij}\in L^{\frac{\bar{q}}{2}}(\Omega)$. This, put together with (\ref{MetricBootstrap.1}) implies that
\begin{align}\label{MetricBootstrap.3}
\hat{g}^{ij}\partial_{ij}\hat{g}_{ab}\in L^{\frac{\bar{q}}{2}}(\Omega), 
\end{align}
where we have used Corollary \ref{SummaryRegularityCor} above so as to guarantee that $\mathrm{Ric}_{\hat{g}}\in L^q(\Omega)$ for some $q>3$. Furthermore, since $\hat{g}^{ij}\in W^{1,\bar{q}}(\Omega)$, we can apply Proposition \ref{RegularityGT} to obtain first $\hat{g}_{ij}\in W_{loc}^{2,\frac{\bar{q}}{2}}(\Omega)$, which implies $\hat{g}\in W^{2,\frac{\bar{q}}{2}}(\mathcal{D}_{\mathrm{Har}}(\hat{M}))$ since away of $p_{\infty}$ this is already known. Thus, as long as $\frac{\bar{q}}{2}<3$, we have $\partial \hat{g}\in W^{1,\frac{\bar{q}}{2}}\hookrightarrow L^{q_1}$, where $q_1=\frac{3\bar{q}}{6-\bar{q}}$. Then, we can go back to (\ref{MetricBootstrap.1}) where now we control the quadratic terms in $L^{\frac{q_1}{2}}$, which would give us $\hat{g}\in W^{2,\frac{q_1}{2}}$. Therefore, we can start a bootstrap, where at each step we begin with $\hat{g}\in W^{2,\frac{q_i}{2}}$ and get $g\in W^{2,\frac{q_{i+1}}{2}}$ with $q_{i+1}=\frac{3q_i}{6-q_i}$, and $q_0=\bar{q}$, which works as long as $\frac{q_i}{2}<3$. Notice that
\begin{align*}
q_{i+1}=\frac{3q_i}{6-q_i}>q_i \Longleftrightarrow 3>6-q_i\Longleftrightarrow q_i>3,
\end{align*}
since $q_0=\bar{q}>3$, we then see that $q_1>\bar{q}>3$ and hence, inductively, $q_{i+1}>q_i>\bar{q}>3$. Furthermore, as long as the bootstrap works we have $\frac{3}{2}<\frac{q_i}{2}<3$, implying $0<6-q_i<3$, and thus
\begin{align*}
q_{i+1}-q_i=\left(\frac{3}{6-q_i}-1\right)q_i=\left(\frac{q_i-3}{6-q_i}\right)q_i>\frac{q_i-3}{3}q_i>\frac{\bar{q}-3}{3}\bar{q}>0.
\end{align*}
Hence, we see that $q_{i+1}>\bar{q}+i\frac{\bar{q}-3}{3}\bar{q}$. This implies that after a finite number of iterations we must get $\frac{q_{i+1}}{2}>3$, at which step the argument stops and we achieve $\hat{g}\in W^{2,q}$ for some $q>3$. Finally, notice that the properties of $1. - 3.$ at $p_{\infty}$ hold due to Corollary \ref{HarmoniCoordCorollary}.
\end{proof}


For future purposes, we would like to keep detailed track of the dependence of coordinates around $p_{\infty}$ in $\mathcal{D}_{\mathrm{Har}}(\hat{M})$ to those in $\mathcal{D}_{0}(\hat{M})$. In particular, we will make use of the following result which deals with normal coordinates associated with $\mathcal{D}_{\mathrm{Har}}(\hat{M})$ around $p_{\infty}$. 

\begin{cor}\label{NormalCoordinateChangeCompact}
Let $(M^3,g)$ be a $W^{k,p}_{\tau}$-AE manifold relative to a structure of infinity with coordinates $\{z\}_{i=1}^3$, $-1<\tau<-\frac{1}{2}$, $p>2$ and $k\geq 2$. Let $(\hat{M}^3,\hat{g})$ be the one point conformal compactification obtained for it in Lemma \ref{ApriopriCompactifiedRegularity}, so that $\hat{g}=\varphi^{4}g\in W^{2,q_0}(\hat{M})$, $2<q_0<3$, with $\varphi$ given as in this lemma. Then, in a neighbourhood of $p_{\infty}$, there is a normal coordinate system $\{\bar{y}^i\}_{i=1}^3\in \mathcal{D}_{\mathrm{Har}(\hat{M})}$ such that 
\begin{enumerate}
\item $\bar{y}^i(p_{\infty})=0$;
\item $\hat{g}(\partial_{\bar{y}^i},\partial_{\bar{y}^j})\vert_{0}=\delta_{ij}$ and $\frac{\partial \hat{g}(\partial_{\bar{y}^i},\partial_{\bar{y}^j})}{\partial \bar{y}^k}\vert_{0}=0$;
\item The functions $\bar{y}^i=\bar{y}^i(y)$ are smooth with $\bar{y}^i=y^i+O_{\infty}(|y|^2)$;
\end{enumerate}
where above the coordinates $\{y^i\}_{i=1}^3$ refer to harmonic coordinates in a neighbourhood of $p_{\infty}$ constructed in Corollary \ref{HarmonicCoord}. In particular, this implies that on a punctured ball $B_{\epsilon}(p_{\infty})\backslash\{p_{\infty}\}$:
\begin{align}\label{FullCoordinateChageCompact}
\bar{y}^i&=x^i+O_{1}(|x|^{1+\alpha})=\frac{z^i}{|z|^2} + O_{1}(|z|^{-1-\alpha}),
\end{align}
for some $\alpha>0$, where $x^i=\frac{z^i}{|z|^2}$ denotes the inverted coordinate system around $p_{\infty}$ defining $\mathcal{D}_0$.
\end{cor}
\begin{proof}
Since by Theorem \ref{MetricBootstrapThm} $\hat{g}\in C^{1,\alpha}(\mathcal{D}_{\mathrm{Har}}(\hat{M}))$, the first part of the corollary is achieved following step by step the proof of \cite[Proposition 1.25]{AubinBook}, and choosing therein the starting coordinate system as the harmonic coordinates constructed in Corollary \ref{HarmonicCoord}. After that, we are only left with the proof of (\ref{FullCoordinateChageCompact}). To establish this expansion, notice that from Theorem \ref{HarmonicCoord} we know that on a punctured ball $B_{\epsilon}(p_{\infty})\backslash\{p_{\infty}\}$
\begin{align*}
y^i(x)=x^{i} + O_1(|x|^{1+\alpha})=\frac{z^i}{|z|^2}+O_1(|z|^{-1-\alpha}).
\end{align*}
Therefore,
\begin{align*}
|y|=|x|+O_1(|x|^{1+\alpha})=\frac{1}{|z|} + O_1(|z|^{-1-\alpha}).
\end{align*}
Putting this together with $\bar{y}^i=y^i+O_{\infty}(|y|^2)$ implies
\begin{align*}
\bar{y}^i=x^i+O_1(|x|^{1+\alpha})=\frac{z^i}{|z|^2} + O_{1}(|z|^{-1-\alpha}).
\end{align*}
\end{proof}

\section{Decompactification}\label{SectionDecompactification}

In this section we present the proof of Theorem \ref{MainAEthemIntro}. This will build on the analysis of the previous section and, as described in the introduction, we now intend to decompactify the manifold $(\hat{M}^3,\hat{g})$ obtained in Theorem \ref{MetricBootstrapThm} to obtain an element $\gamma\in [g]$ which decays rapidly near infinity, and then we will estimate the conformal factor relating $\gamma$ to $g$. The first of these steps is achieved by the following lemma, which is consequence of Lemma \ref{LemmaMaxwellDilts}.

\begin{lem}\label{LemmDecompactMaxwell}
Let $(\hat{M}^n,\hat{g})$ be a closed manifold with $\hat{g}\in W^{2,q}(\hat{M})$, $q>n$. Given a point $p\in \hat{M}$, let $\{x^i\}_{i=1}^n$ be a normal coordinate system around $p$. Then, there is a conformal factor $\hat{\phi}$, smooth on $\hat{M}\backslash\{p\}$, equal to $|x|^{2-n}$ in a neighbourhood of $p$ and such that $\gamma\doteq \hat{\phi}^{\frac{4}{n-2}}\hat{g}$ is a $W^{2,q}_{\delta}(M,\Phi_{\bar{z}})$-AE metric on $M\doteq\hat{M}\backslash\{p\}$ with $\delta=\frac{n}{q}-2$. In particular, $\delta<-1$, so that
\begin{align}\label{ExpansionConfoMetricAE}
\gamma(\partial_{\bar{z}^i},\partial_{\bar{z}^j})=\delta_{ij}+O_1(|\bar{z}|^{-1-\alpha}),
\end{align}
for some $\alpha>0$, and where $\bar{z}^i=\frac{x^i}{|x|^2}$ stands for the Kelvin transform of the normal coordinates $x$ and provides the structure of infinity $\Phi_{\bar{z}}$.
\end{lem}
\begin{proof}
All of the above lemma but for (\ref{ExpansionConfoMetricAE}) is a restatement of the second half of Lemma \ref{LemmaMaxwellDilts}. To establish (\ref{ExpansionConfoMetricAE}), just notice that the first part of the lemma and the weighted Sobolev embeddings of Theorem \ref{SobolevPorpsAE}  imply that
\begin{align*}
\gamma(\partial_{\bar{z}^i},\partial_{\bar{z}^j})-\delta_{ij}\in W^{2,q}_{\delta}(\mathbb{R}^n\backslash \overline{B_1(0)})\hookrightarrow C^{1}_{\delta}(\mathbb{R}^n\backslash \overline{B_1(0)}),
\end{align*}
which produces the desired expansion for any $\delta<-1-\alpha<-1$. That is, for $0<\alpha<-(\delta+1)$.  
\end{proof}

Let us now apply this result to the manifold $(\hat{M}^3,\hat{g})$ obtained from Theorem \ref{MetricBootstrapThm} and keep track of the relation between the initial structure of infinity $\Phi_z$ and the one obtained by the inversion of the normal coordinates $\{\bar{y}^i\}_{i=1}^3$ obtained in Corollary \ref{NormalCoordinateChangeCompact}.

\begin{lem}\label{ConformalMetricExpansion}
Let $(M^3,g)$ be a smooth $W^{k,p}_{\tau}$-AE manifold, relative to a structure of infinity with coordinates $\{z^i\}_{i=1}^3$ and $p> 2$, $\tau\in (-1,-\frac{1}{2})$, $k\geq 2$, such that $C_g\in L^{p_1}_{\sigma}(M,dV_g)$ with $-6<\sigma<-4$ and $p_1=\frac{3}{6+\sigma}$. Let $(\hat{M},\hat{g})$ be the conformal compactification achieved for it in Theorem \ref{MetricBootstrapThm}, with $\hat{g}\in W^{2,q}(\mathcal{D}_{\mathrm{Har}}(\hat{M}))$ with $q>3$. Assume that $\gamma=\hat{\phi}^4\hat{g}$ is the decompactification of $\hat{g}$ achieved in Lemma \ref{LemmDecompactMaxwell} on $M=\hat{M}\backslash\{p_{\infty}\}$. Then, for some $\alpha>0$, it holds that
\begin{align}
\gamma(\partial_{z^i},\partial_{z^j})=\delta_{ij} + O(|z|^{-\alpha}).
\end{align}
Furthermore, letting $\bar{y}$ stand for the normal coordinates around $p_{\infty}\in \hat{M}$ constructed in Corollary \ref{NormalCoordinateChangeCompact} and defining the asymptotic coordinates $\bar{z}^i\doteq \frac{\bar{y}^i}{|\bar{y}|^2}$ on $M$, then 
\begin{align}\label{AsymptoticCoordChange}
\bar{z}(z)-\mathrm{Id}(z)\in C^{1}_{1-\alpha}(\mathbb{R}^n\backslash\overline{B_R(0)})
\end{align}
for some large $R$.
\end{lem}
\begin{proof}
Under our hypotheses, Corollary \ref{NormalCoordinateChangeCompact} gives us
\begin{align*}
\bar{y}^i=\frac{z^i}{|z|^2}+O_1(|z|^{-1-\alpha}).
\end{align*}
Therefore, $|\bar{y}|=\frac{1}{|z|} + O_1(|z|^{-1-\alpha})$, which implies
\begin{align*}
\bar{z}^i&=\frac{\bar{y}^i}{|\bar{y}|^2}=\frac{\frac{z^i}{|z|^2}+O_1(|z|^{-1-\alpha})}{\frac{1}{|z|^2} + O_1(|z|^{-2-\alpha})}
=z^i +O_1(|z|^{1-\alpha}),
\end{align*}
which already establishes (\ref{AsymptoticCoordChange}). Also,
\begin{align*}
\frac{\partial\bar{z}^i}{\partial z^j}&=\delta_{ij} + O(|z|^{-\alpha}).
\end{align*}
Finally, using Lemma \ref{LemmDecompactMaxwell} with the normal coordinates $x^i\doteq \bar{y}^i(z)$ on $\hat{M}$ around $p_{\infty}$ so as to write $\gamma(\partial_{\bar{z}^a},\partial_{\bar{z}^b})=\delta_{ab} + O_1(|\bar{z}|^{-1-\alpha})$, we finally get: 
\begin{align*}
\gamma(\partial_{z^i},\partial_{z^j})&=\frac{\partial\bar{z}^a}{\partial z^i}\frac{\partial\bar{z}^b}{\partial z^j}\gamma(\partial_{\bar{z}^a},\partial_{\bar{z}^b}),\\
&=\left(\delta_{ai} + O(|z|^{-\alpha}) \right)\left(\delta_{bj} + O(|z|^{-\alpha}) \right)\left(\delta_{ab} + O(|z|^{-1-\alpha}) \right),\\
&=\delta_{ij} + O(|z|^{-\alpha}),
\end{align*}
where we have used that $|\bar{z}|=|z| + O_1(|z|^{1-\alpha})$.
\end{proof}

We now have enough information to start estimating the conformal factor $u$ between our initial AE-metric $g$ on $M^3$ and $\gamma\in [g]$ as obtained in Lemma \ref{ConformalMetricExpansion} above. The first step is to grant to $u-1$ has some minimum decay, as shown below.
\begin{cor}\label{ConformalFactorAssymptotics.0}
Consider the same setting as in the previous lemma and let $u$ be the conformal factor relating $g=u^4\gamma$. Then, $u-1=O(|z|^{-\mathrm{min}(-\tau,\alpha)})$. 
\end{cor}
\begin{proof}
The above Lemma implies
\begin{align*}
g(\partial_{z^i},\partial_{z^j})=\delta_{ij}+O(|z|^{\tau})=u^4(\delta_{ij}+O(|z|^{-\alpha})).
\end{align*}
Setting $\sigma_{ij}\doteq \delta_{ij}+O(|z|^{-\alpha})$ we see that for $|z|$ sufficiently large this matrix is invertible, with $\sigma^{-1}=\delta_{ij} + O(|z|^{-\alpha}) $, which can be derived using a Neumann series for $\sigma^{-1}$. Therefore, 
\begin{align}
u^{4}\delta_{ij}&=\sigma^{-1}_{il}\left(\delta_{lj}+O(|z|^{\tau}) \right)=\delta_{ij} + O(|z|^{-\mathrm{min}(-\tau,\alpha)}).
\end{align}
Tracing the equation we obtain $u^4=1+O(|z|^{-\mathrm{min}(-\tau,\alpha)})$, from which the result follows.
\end{proof}

Let us now present the following lemma, which is an adaptation to the inhomogeneous case of \cite[Theorem 1.17]{BartnikMass} and will allow us to improve the decay of solutions to Poisson-type equations when the source of the equation decays \emph{fast}. In this analysis, certain weight parameters in Sobolev spaces where the Laplacian looses its Fredholm properties are distinguished due to their relevance. On an $n$-dimensional manifold they are given by the so called \emph{exceptional values}, which are those in $\mathbb{Z}$ if $n=3$ while for $n\geq 4$ these are given by $\mathbb{Z}\backslash \{3-n,\cdots,-1\}$. Thus, one says that $\delta$ is a non-exceptional weight if it does not lie in the set of exceptional values. Furthermore, given $\delta\in \mathbb{R}$, one defines $k^{-}(\delta)$ to be the maximum exceptional value such that $k^{-}(\delta)\leq \delta$. 

\begin{lem}\label{ImprovedDecayBartnik}
Let $(M^n,g)$ be a $W^{2,q}_{\tau}(M)$-AE manifold, with $q>n$ and $\tau<0$. Assume $u\in W^{2,q}_{\delta}(M)$ with $\delta<0$. Assume that $f\in L^q_{\delta'-2}(M)$ for some $\delta'< \delta$ and that $\Delta_g u=f$. If $1-n<\delta'<2-n$, then there is some constant $C$ and a real number $\rho\in [\delta',2-n)$ such that
\begin{align}\label{MassExpanssion}
u-\frac{C}{|x|^{n-2}}\in W^{2,q}_{\rho}(\mathbb{R}^n\backslash\overline{B_1(0)}).
\end{align}
\end{lem}
\begin{proof}
Let us first notice that $\Delta_g:W^{2,q}_{\eta}(M)\to L^q_{\eta-2}(M)$ is an isomorphism for $\eta\in (2-n,0)$ due to, for instance, \cite[Proposition 1]{Maxwell1}.\footnote{The cited result is presented under our regularity conditions for manifolds with boundary, and the translation to the case of $M$ without boundary is straightforward. Also, this result in contained in \cite[Proposition 2.2]{BartnikMass}. In order to appeal to this latter result, one should be careful and substitute the Fredholm analysis which is there presented by Corollary \ref{WeightedFredholmCoro}, which is enough to prove the isomorphism claim under our present hypotheses.} Then, $f\in L^q_{\delta'-2}(M)$ for some $1-n<\delta'<2-n$ directly implies $u\in W^{2,q}_{\eta}(M)$ for all $\eta>2-n$. Thus, we can start assuming that $u\in W^{2,q}_{\delta}(M)$ for all $\delta\in (2-n,0)$ without any loss of generality. With this assumption, to proceed further, let us start by transferring the problem to $\mathbb{R}^n$. With this in mind, we chose a cut off function $\eta$, supported in the end of $M$ and identically equal to one in a neighbourhood of infinity. Then,
\begin{align*}
\Delta_g(\eta u)&=\eta\Delta_gu + 2\langle \nabla\eta,\nabla u \rangle_g + u\Delta_g\eta=\eta f + 2\langle \nabla\eta,\nabla u \rangle_g  + u\Delta_g\eta
\end{align*}
One can also locally compute 
\begin{align*}
\Delta_g(\eta u)&=g^{ij}(\partial_{ij}(\eta u) - \Gamma^l_{ij}\partial_l(\eta u))=\Delta(\eta u) - \Gamma^l_{ii}\partial_l(\eta u) + (g^{ij}-\delta_{ij})(\partial_{ij}(\eta u) - \Gamma^l_{ij}\partial_l(\eta u)),
\end{align*}
which gives us
\begin{align}\label{Bartnik1}
\Delta(\eta u)&=\eta f + F\in L^q_{\mathrm{max}(\delta'-2,\delta-2+\tau)}(\mathbb{R}^n),
\end{align}
where
\begin{align}\label{Bartnik2}
\!\!\! F&\doteq - (g^{ij}-\delta_{ij})\partial_{ij}(\eta u) + 2\langle \nabla\eta,\nabla u \rangle_g + u\Delta_g\eta + \Gamma^l_{ii}\partial_l(\eta u) + (g^{ij}-\delta_{ij})\Gamma^l_{ij}\partial_l(\eta u)\in L^q_{\delta-2+\tau}(\mathbb{R}^n).
\end{align}
Recall that $\Delta:W^{2,q}_{\rho}(\mathbb{R}^n)\to L^q_{\rho-2}(\mathbb{R}^n)$ is Fredholm as long as $\rho$ is non-exceptional due to Corollary \ref{LapFredCoro1}, and, if necessary by increasing $\tau$ while keeping it negative, one has that $\delta_0\doteq \mathrm{max}(\delta',\delta + \tau)$ is non-exceptional. From the Fredholm property, we know that $\Delta(\eta u)\in L^{q}_{\delta_0-2}(\mathbb{R}^n)$ is in $\Img(\Delta:W^{2,q}_{\delta_0}(\mathbb{R}^n)\to L^{q}_{\delta_0-2}(\mathbb{R}^n))$ iff $\Delta(\eta u)\in \Ker^{\perp}(\Delta^{*}:L^{q'}_{2-\delta_0-n}(\mathbb{R}^n)\to W^{-2,q'}_{-\delta_0-n}(\mathbb{R}^n))$. When $\delta_0<2-n$, this cokernel will be non-empty and this property can fail for $\bar{f}\doteq \eta f+F$. But, being $\Ker(\Delta^{*}\vert_{L^{q'}_{2-\delta_0-n}(\mathbb{R}^n)})$ finite dimensional, we can modify $\bar{f}$ in a fixed compact set, say the ball $\overline{B_1(0)}$, by subtracting its projections onto the elements of the cokernel and localising these contributions with a fixed cut-off function $\chi$ supported in $B_1(0)$,\footnote{In a general situation, where $\Ker(\Delta^{*}\vert_{L^{q'}_{2-\delta_0-n}(\mathbb{R}^n)})$ may not be one dimensional, this procedure would look like a Gram-Schmidt-type procedure.} thus obtaining $\hat{f}$ which agrees with $\bar{f}$ outside a compact set but is actually an element of $\Ker^{\perp}(\Delta^{*}\vert_{L^{q'}_{2-\delta_0-n}(\mathbb{R}^n)})$. Therefore, there is some element $\varphi\in W^{2,q}_{\delta_0}(\mathbb{R}^n)$ such that $\Delta\varphi=\hat{f}$ and thus
\begin{align*}
\Delta(\eta u-\varphi)=0 \text{ on } \mathbb{R}^n\backslash\overline{B_R(0)} \text{ for } R>1. 
\end{align*}
Since $h\doteq \eta u-\varphi\rightarrow 0$ as $|x|\rightarrow \infty$, we may appeal to well-understood characterisation of harmonic functions vanishing at infinity to expand $h$ into spherical harmonics $Y_k$:\footnote{For reference, see \cite[Appendix A Corollary A.19]{LeeDanBook} and further details can be found in \cite[Chapter 2, Sections H and I]{FollandPDEs}.}
\begin{align}\label{SphericalHarmonics}
h(x)=\sum_{k=0}^{\infty}|x|^{2-n-k}Y_k\left(\frac{x}{|x|}\right),
\end{align}
where $Y_k$ stand for the standard spherical harmonics, and the series above converges uniformly and absolutely. Since  $1-n<\delta'<2-n$, using (\ref{SphericalHarmonics}) we see all terms except for the $k=0$ one are in $W^{2,q}_{\delta'}(\mathbb{R}^n\backslash\overline{B_1(0)})$, and recognising that the $k=0$ contribution in (\ref{SphericalHarmonics}) is of the form $\frac{C}{|x|^{n-2}}$, we have that
\begin{align}\label{FredholmBootstrap.2}
\eta u-\frac{C}{|x|^{n-2}}=\varphi\in W^{2,q}_{\delta_0}(\mathbb{R}^n\backslash\overline{B_1(0)}).
\end{align}
Since $\delta$ can be taken arbitrarily close to $2-n$, we can chose it such that $\delta+\tau<2-n$, so that $\delta_0=\max\{\delta',\delta+\tau\}<2-n$, and we therefore have achieved the expansion (\ref{MassExpanssion}).

\end{proof}

We can now combine Corollary \ref{ConformalFactorAssymptotics.0} and Lemma \ref{ImprovedDecayBartnik} to extract an improved expansion of the conformal factor $u$ relating $\gamma$ and $g$ on $M$ under a scalar curvature decay condition.

\begin{lem}\label{HarmonicExpansionLemma}
Consider the same setting as in Lemma \ref{ConformalMetricExpansion} and let $u$ be the conformal factor relating $g=u^{4}\gamma$. Assume $R_g\in L^r_{-3-\epsilon}(M,\Phi_z)$ for some $\epsilon>0$ and $r>3$. Then, 
\begin{align}\label{HarmonicExpansion}
u=1+\frac{C}{|\bar{z}|} + O_1(|\bar{z}|^{-1-\alpha})
\end{align}
for some constant $C$, which is implicitly given as
\begin{align}\label{MassComputationGeneral}
C=\frac{1}{8\omega_2}\int_M\left(R_gu^{5}-R_{\gamma}u \right)dV_{\gamma},
\end{align}
where above $\omega_2$ stands for the volume of the standard $2$-unit sphere.
\end{lem}
\begin{proof}
First, from Lemma \ref{ConformalMetricExpansion}, in particular equation (\ref{AsymptoticCoordChange}), one knowns that $L^r_{\delta}$-spaces are preserved by the coordinate change $z\mapsto \bar{z}$, implying that $R_g\in L^r_{-3-\epsilon}(M,\Phi_{\bar{z}})$. Also, from Lemma \ref{LemmDecompactMaxwell} $\gamma$ is $W^{2,q}_{\delta}(M,\Phi_{\bar{z}})$-AE, with $\delta<-1$ and $q>3$, and from conformal covariance of the conformal Laplacian we know that
\begin{align}\label{FinalConformalChange}
-8\Delta_{\gamma}u=R_gu^{5}-R_{\gamma}u.
\end{align}
Thus, from Corollary \ref{ConformalFactorAssymptotics.0} and our first observation above, we know that $u-1\in W^{2,q}_{loc}\cap L^{q}_{\rho}(M,\Phi_{\bar{z}})$ for some $\tau<\rho<0$.\footnote{The $W^{2,q}_{loc}$-condition actually follows since $u$ is smooth.} Furthermore, from Corollary \ref{ConformalFactorAssymptotics.0} we know that $u=1+O(|\bar{z}|^{-\rho})$, which implies that $R_gu^5\in L^r_{-3-\epsilon}(M,\Phi_{\bar{z}})$ and $R_{\gamma}u\in L^q_{\delta-2}(M,\Phi_{\bar{z}})$ with $\delta<-1$. Now, since $L^r_{-3-\epsilon}(M,\Phi_{\bar{z}})\hookrightarrow L^{r'}_{-3-\epsilon'}(M,\Phi_{\bar{z}})$ for any $3<r'\leq r$ and $0<\epsilon'<\epsilon$, we can (if necessary) decrease $\epsilon$ and $r$ and assume that $3<r\leq q$ and $\delta-2<-3-\epsilon$, so that $L^q_{\delta-2}(M,\Phi_{\bar{z}})\hookrightarrow L^r_{-3-\epsilon}(M,\Phi_{\bar{z}})$, implying $\Delta_{\gamma}(u-1)\in L^r_{-3-\epsilon}(M,\Phi_{\bar{z}})$ with $u-1\in W^{2,q}_{loc}\cap L^{q}_{\rho}(M,\Phi_{\bar{z}})$, with $\rho>-1$ by hypothesis ($\tau>-1$). Since $L^r_{-3-\epsilon}(M,\Phi_{\bar{z}})\hookrightarrow L^r_{-\epsilon}(M,\Phi_{\bar{z}})$, again by decreasing $\epsilon$ if necessary this implies $L^q_{\rho}(M,\Phi_{\bar{z}})\hookrightarrow L^r_{-\epsilon}(M,\Phi_{\bar{z}})$. That is, we have found
\begin{align*}
\Delta_{\gamma}(u-1)\in L^r_{-3-\epsilon}(M,\Phi_{\bar{z}}) \text{ with } u-1\in W^{2,r}_{loc}\cap L^{r}_{-\epsilon}(M,\Phi_{\bar{z}}).
\end{align*}
Then, Theorem \ref{BartniksProp1.6} implies that $u-1\in W^{2,r}_{-\epsilon}(M,\Phi_{\bar{z}})$.  Thus, we find $u-1\in W^{2,r}_{-\epsilon}(M,\Phi_{\bar{z}})$ with $\Delta_{\gamma}(u-1)\in L^r_{-3-\epsilon}$ and $\gamma$ a $W^{2,q}_{\delta}(M,\Phi_{\bar{z}})$-AE metric. But then, Lemma \ref{ImprovedDecayBartnik} implies 
\begin{align}\label{SharpAsympt.0}
u-1-\frac{C}{|\bar{z}|}\in W^{2,r}_{-1-\alpha}(\mathbb{R}^3\backslash\overline{B_1(0)},\Phi_{\bar{z}}),\: \text{ for some } \alpha>0 \text{ and } r>3. 
\end{align}
This, in particular, implies that
\begin{align}
u=1+\frac{C}{|\bar{z}|}+O_1(|\bar{z}|^{-1-\alpha}).
\end{align}

Finally, to compute the constant $C$, integrate (\ref{FinalConformalChange}) on a large compact set $D_R$ such that $\partial D_R=S^2_R$, where $S^2_R$ denotes a topological $2$-sphere of radius $|\bar{z}|=R$, for some large $R$. Doing this, one gets:
\begin{align*}
\int_{D_R}\left(R_gu^{5} - R_{\gamma}u\right)dV_{\gamma}&=-8\int_{D_R}\Delta_{\gamma}udV_{\gamma}=-8\int_{S_R}\langle \nabla u,\nu_{\gamma}\rangle_{\gamma}d\omega_{\gamma},
\end{align*}
where $\nu_{\gamma}$ denotes the outward-pointing $\gamma$-unit normal to $S^2_R\hookrightarrow M$, while $d\omega_{\gamma}$ denotes the induced volume form on $S^2_R$ by $\gamma$. Since asymptotically one has
\begin{align*}
\langle \nabla u,\nu_{\gamma}\rangle_{\gamma}&=\langle \nabla u,\nu_{\gamma}\rangle_{\delta} + (\gamma_{ij}-\delta_{ij})\nabla^iu\nu^j_{\gamma}=\langle \nabla u,\nu_{\delta}\rangle_{\delta} + \langle \nabla u,\nu_{\gamma}-\nu_{\delta}\rangle_{\delta} + (\gamma_{ij}-\delta_{ij})\nabla^iu\nu^j_{\gamma},\\
&=\partial_{\nu_{\delta}}u + O(|\bar{z}|^{-3}).
\end{align*}
where above $\nu_{\delta}=\frac{\bar{z}^i}{|\bar{z}|}$ denotes the Euclidean unit normal to $S^2_{R}$ and $\partial_{\nu_{\delta}}u$ denotes the radial derivative of $u$ in the $\nu_{\delta}$-direction. Thus, we see that
\begin{align*}
\int_{D_R}\left(R_gu^{5}-R_{\gamma}u\right)dV_{\gamma}&=-8\int_{S_R}\partial_{\nu_{\delta}}ud\omega_{\gamma} + O(R^{-1})=-8\int_{S_R}\partial_{\nu_{\delta}}ud\omega_{\delta} + O(R^{-1}),\\
&=8C\omega_2 + O(R^{-\alpha})
\end{align*}
Passing to the limit establishes (\ref{MassComputationGeneral}):
\begin{align*}
C=\frac{1}{8\omega_2}\int_M\left(R_gu^{5}-R_{\gamma}u\right)dV_{\gamma}
\end{align*}
\end{proof}

We can now present the main result of this section:

\begin{theo}\label{MainThmAE}
Let $(M^3,g)$ be a smooth $W^{2,p}_{\tau}$-AE manifold with respect to a structure of infinity with coordinates $\{z^i\}_{i=1}^3$, with $\tau\in (-1,-\frac{1}{2})$ and $p>2$. Assume furthermore that: 
\begin{enumerate}
\item $R_g\in L^r_{-3-\epsilon}(M,dV_g)$ for some $r>3$ and $\epsilon>0$;
\item $C_{g}\in L^{p_1}_{\sigma}(M,dV_g)$ for some $-6<\sigma<-4$ and $p_1=\frac{3}{6+\sigma}$.
\end{enumerate}
Then, there is a structure of infinity with coordinates $\{\bar{z}^i\}_{i=1}^3$, which is $C^{1,\alpha}$-compatible with the original one, such that
\begin{align}
g(\partial_{\bar{z}^i},\partial_{\bar{z}^j})&=\left(1+\frac{4C}{|\bar{z}|}\right)\delta_{ij} + O_1(|\bar{z}|^{-1-\alpha}),\label{ImprovedDecayExpansion01}\\
\bar{z}(z)-\mathrm{Id}(z)&\in C^1_{1-\alpha}(\mathbb{R}^n\backslash\overline{B_{R_0}(0)}),\label{ImprovedDecayExpansion02}
\end{align}
for some $\alpha>0$ and where the constant $C$ is given by (\ref{MassComputationGeneral}).
\end{theo}
\begin{proof}
Under our hypotheses Theorem \ref{MetricBootstrapThm} follows. From this we obtain a conformal compactification $(\hat{M},\hat{g})$ satisfying the hypotheses of Lemma \ref{LemmDecompactMaxwell}, and therefore we know that $g$ is conformal to $\gamma\in W_{\delta}^{2,q}(M)$, $q>3$, satisfying (\ref{ExpansionConfoMetricAE}). Writing $g=u^{4}\gamma$, Lemma \ref{HarmonicExpansionLemma} implies $u$ satisfies (\ref{HarmonicExpansion}). Therefore, we can expand
\begin{align}
u^4=1+\frac{4C}{|\bar{z}|} + O_1(|\bar{z}|^{-1-\alpha}),
\end{align}
which put together with (\ref{ExpansionConfoMetricAE}) implies (\ref{ImprovedDecayExpansion01}), where the constant $C$ is given by (\ref{MassComputationGeneral}). Finally, the expansion (\ref{ImprovedDecayExpansion02}) follows from Lemma \ref{ConformalMetricExpansion}.
\end{proof}


It is interesting to contrast the above result with the decay control one can achieve via the Ricci tensor. In particular, from \cite[Proposition 3.3]{BartnikMass}, one knowns the Ricci tensor provides \emph{optimal} control for the decay of the metric to the Euclidean one in harmonic coordinates, which in turn are known to provide optimal decay control among compatible end coordinate systems (see \cite[Theorem 3.1]{BartnikMass} and for related results \cite[Theorem 5.2]{ALM}). Thus, one may think that expansions of the type of (\ref{ImprovedDecayExpansion01}) should be readable in harmonic coordinates directly from controls on the Ricci tensor. In fact, if one knew that $\mathrm{Ric}_g\in L^p_{\delta-2}(M,\Phi_z)$ for some $p>3$ and $\delta<-1$, then would be able to extract (\ref{ImprovedDecayExpansion01}) in harmonic coordinates directly from the Ricci tensor. We would like to highlight that, in fact, such an a priori decay condition on the Ricci tensor is quite strong. To see this, let us explore the case of a metric which is known to be asymptotically \emph{Schwarzschildian}, and show that such a decay cannot be extracted from the Ricci tensor if one does not a priori know it.  That is, let us consider $(M^3,g)$ AE with the following asymptotics:
\begin{align}\label{SchwarzExpan}
g_{ij}=\delta_{ij}+\frac{A\delta_{ij}}{|z|} + O_{\infty}(|z|^{-2})
\end{align}
Then, using (\ref{gauge0}), we find:
\begin{align*}
{\mathrm{Ric}_g}_{ij}=-\frac{1}{2}g^{ab}\partial_{ab}g_{ij}+\frac{1}{2}(g_{ia}\partial_{j}F^{a}+g_{ja}\partial_{i}F^{a}) + O_{\infty}(|z|^{-4}),
\end{align*}
with $F^a=g^{kl}\Gamma^a_{kl}(g)$. Then, since $g^{ab}\partial_{ab}g_{ij}=A\delta_{ij}\Delta|z|^{-1} + O_{\infty}(|z|^{-4})$, we find
\begin{align*}
{\mathrm{Ric}_g}_{ij}=\frac{1}{2}(g_{ia}\partial_{j}F^{a}+g_{ja}\partial_{i}F^{a}) + O_{\infty}(|z|^{-4}).
\end{align*}
Since
\begin{align*}
\Gamma^a_{kl}(g)&=\frac{g^{ba}}{2}\left( \partial_kg_{bl} + \partial_lg_{bk} - \partial_bg_{kl} \right)=\frac{1}{2}\left( \partial_kg_{al} + \partial_lg_{ak} - \partial_ag_{kl} \right) + O_{\infty}(|z|^{-3}),\\
&=\frac{A}{2}\left( - \delta_{al}\frac{z^k}{|z|^3} - \delta_{ak}\frac{z^l}{|z|^3} + \delta_{kl}\frac{z^a}{|z|^3} \right) + O_{\infty}(|z|^{-3}),\\
F^a&=g^{kl}\Gamma^a_{kl}(g)=\frac{A}{2}\left( - \delta_{ak}\frac{z^k}{|z|^3} - \delta_{al}\frac{z^l}{|z|^3} + 3\frac{z^a}{|z|^3} \right) + O_{\infty}(|x|^{-3})=\frac{A}{2}\frac{z^a}{|z|^3}  + O_{\infty}(|z|^{-3}),\\
\partial_jF^a&=\frac{A}{2}\left(\frac{\delta_{aj}}{|z|^3} - 3\frac{z^az^j}{|z|^5} \right) + O_{\infty}(|z|^{-4}),\\
g_{ia}\partial_jF^a&=\frac{A}{2}\left(\frac{\delta_{ij}}{|z|^3} - 3\frac{z^iz^j}{|z|^5} \right) + O_{\infty}(|z|^{-4}),\\
g_{ia}\partial_jF^a+g_{ja}\partial_iF^a&=A\left(\frac{\delta_{ij}}{|z|^3} - 3\frac{z^iz^j}{|z|^5} \right) + O_{\infty}(|z|^{-4}),
\end{align*}
which finally implies
\begin{align}\label{SchwarzRicci}
{\mathrm{Ric}_g}_{ij}=\frac{A}{2}\left(\frac{\delta_{ij}}{|z|^3} - 3\frac{z^iz^j}{|z|^5} \right) + O(|z|^{-4}).
\end{align}
Notice then that, as long as $A\neq 0$, the decay rate for the Ricci tensor is critical, in the sense that $\mathrm{Ric}_g$ can be seen to be in $L^p_{\delta-2}$ for any $\delta>-1$, but it is not in any $\delta\leq -1$. Thus, if one were to have the metric $g$ in any other coordinate system $\{y^i\}_{i=1}^3$ such that the transformation $y=y(z)$ is asymptotic to the identity to a sufficiently high order, but where the Schwarzschildian expansion (\ref{SchwarzExpan}) is not explicit,\footnote{Notice that any transformation of the form $y^i=z^i+O_{\infty}(|z|^0)$ will in general introduce terms at order $\frac{1}{|y|}$ for $g(\partial_{y^i},\partial_{y^j})$, which would spoil the Schwarzschildian decay.} then one would not be able to deduce the existence of coordinates where (\ref{SchwarzExpan}) holds analysing the decay of the Ricci tensor. Nevertheless, notice that
\begin{align}\label{SchwarzScal}
R_g=O_{\infty}(|z|^{-4})
\end{align}
and thus
\begin{align*}
C_{ijk}(g)&=\nabla_k{\mathrm{Ric}_g}_{ij} - \nabla_j{\mathrm{Ric}_g}_{ik} + \frac{1}{4}\left( \nabla_jR_gg_{ik} - \nabla_kR_gg_{ij} \right),\\
&=\nabla_k{\mathrm{Ric}_g}_{ij} - \nabla_j{\mathrm{Ric}_g}_{ik} + O(|z|^{-5})=\partial_k{\mathrm{Ric}_g}_{ij} - \partial_j{\mathrm{Ric}_g}_{ik} + O(|z|^{-5}),
\end{align*}
and from (\ref{SchwarzRicci}), we can compute
\begin{align*}
\partial_k{\mathrm{Ric}_g}_{ij}&=\frac{A}{2}\left(-3\frac{\delta_{ij}x^k}{|z|^5} - 3\left(\frac{\delta_{ik} z^j + \delta_{jk} z^i }{|z|^5} \right) +15\frac{z^iz^jz^k}{|z|^{7}} \right) + O(|z|^{-5}),\\
&=\frac{A}{2}\left( - 3\frac{\delta_{ij}x^k +\delta_{ik} z^j + \delta_{jk} z^i }{|z|^5} +15\frac{z^iz^jz^k}{|z|^{7}} \right) + O(|z|^{-5}).
\end{align*}
Noticing the leading order in the above expression is symmetric under the interchange $j\longleftrightarrow k$, we find
\begin{align}\label{SchwarzCotton}
C_{ijk}(g)&= O(|z|^{-5}).
\end{align}

We then see from (\ref{SchwarzScal})-(\ref{SchwarzCotton}) that $g$ satisfies the hypotheses of Theorem \ref{MainThmAE}, and therefore from it we can deduce the existence of a Schwarzschild-type expansion up to first order in some other coordinate system. Notice this implies that Theorem \ref{MainThmAE} is able to detect if a metric has an asymptotically Schwarzschild expansion in some coordinate system from the decay of $R_g$ and $C_g$, something we have seen is not possible from the analysis of the Ricci-tensor alone. Some interesting and physically relevant examples related to this analysis will be presented in subsequent work.

\section{The center of mass of an AE 3-manifold}\label{SectionCOM}

In this section we aim to address a slightly strengthened version of Conjecture \ref{CSConjectureWeak} stated in the introduction and motivated by \cite{CederbaumSakovich}. The proof of our result, which is an $L^p$-version ($p>1$) of Conjecture \ref{CSConjectureWeak}, is based on the results of the previous section, in particular on Theorem \ref{MainThmAE}.

Let $(M^3,g)$ be an AE manifold, with $\tau>\frac{1}{2}$, and assume then there is a structure of infinity with coordinates $x^i$ where metric obeys the following decaying conditions:
\begin{align}\label{ImprovedAsymptotics}
g_{ij}=\delta_{ij}+\frac{A_{ij}}{|x|}+O_{1}(|x|^{-1-\alpha})
\end{align}

We would like to prove that the asymptotics (\ref{ImprovedAsymptotics}) imply the associated COM
\begin{align}\label{COMreminder}
C^k_{B\acute{O}M}&\doteq\frac{1}{16\pi E} \lim_{r\rightarrow\infty}\left(\int_{S_r}x^k\left(\partial_ig_{ij}-\partial_jg_{ii} \right)\nu^jd\omega_r - \int_{S_r}\left(g_{ik}\nu^i - g_{ii}\nu^k \right)d\omega_r \right)
\end{align}
is well-defined, where above we are using the same notations adopted in Section \ref{SectionADM}.
\begin{lem}\label{COMConvergenenceLemma}
If $(M^3,g)$ is a smooth AE manifold satisfying the decaying conditions (\ref{ImprovedAsymptotics}) for some $\alpha>0$ and $R_g\in L^1_{-4}(\mathbb{R}^3\backslash\overline{B_1(0)})$, then the center of mass (\ref{COMreminder}) is well-defined. 
\end{lem}
\begin{proof}
Let us start by writing the scalar curvature as
\begin{align}\label{COM.1}
R_g=g^{ij}g^{kl}\partial_{ik}g_{jl} - g^{ij}g^{kl}\partial_{ij}g_{kl} + O(g_{\cdot}(\partial g)^2),
\end{align}
and consider an annulus $A_R=B_{R}\backslash \overline{B_{R_0}}$ for $R_0,R$ sufficiently large, $R>R_0$. Then, due to symmetry of $A_R$:
\begin{align}
\int_{A_R}x^aR_gdx=\int_{A_R}x^aR^{\:\Odd}_gdx.
\end{align}
To compute the above integrals, notice that in (\ref{COM.1}) we are interested in the odd part of the $O(g_{\cdot}(\partial g)^2)$, but to analyse it one needs a more explicit form. From (\ref{gauge0}) and (\ref{f-tensor}), we know this expression is a linear combination of quadratic forms on $\partial g$, which one can generically write as
\begin{align*}
Q&=(g^{-1}_{\cdot} g^{-1}_{\cdot} g^{-1})_{\cdot}(\partial g\otimes \partial g)=(\delta_{\cdot} \delta_{\cdot} \delta)_{\cdot}(\partial g\otimes \partial g) + O(|x|^{-5}),
\end{align*}
where the lower dots denote some specific contraction. So,
\begin{align*}
Q^{\Odd}&=(\delta_{\cdot} \delta_{\cdot} \delta)_{\cdot}(\partial g\otimes \partial g)^{\Odd} + O(|x|^{-5})=O((\partial g)^{\Odd})O((\partial g)^{\Ev}) + O(|x|^{-5}).
\end{align*}
Noticing that (\ref{ImprovedAsymptotics}) implies $\partial_kg_{ij}=-\frac{A_{ij}x^k}{|x|^3}+O(|x|^{-2-\alpha})$, we see that
\begin{align*}
(\partial g)^{\Odd}&=-\frac{A_{ij}x^k}{|x|^3}+O(|x|^{-2-\alpha})=O(|x|^{-2}),\\
(\partial g)^{\Ev}&=O(|x|^{-2-\alpha}),
\end{align*}
which implies 
\begin{align}\label{QuadraticTermsEstimate}
Q^{\Odd}&=O(|x|^{-4-\alpha}).
\end{align}
Therefore,
\begin{align}\label{COM.2}
\int_{A_R}x^a\left(g^{ij}g^{kl}\partial_{ik}g_{jl} - g^{ij}g^{kl}\partial_{ij}g_{kl} \right)dx=\int_{A_R}x^aR_gdx + O(R^{-\alpha}) + C,
\end{align}
where the constant $C$ above depends on $R_0$, which is fixed. Below, we will change this constant appropriately, as long as it remains fixed once $R_0$ is fixed.

Now, it also holds that
\begin{align*}
\left(g^{ij}g^{kl}\partial_{ik}g_{jl} - g^{ij}g^{kl}\partial_{ij}g_{kl} \right)&=\partial_i\left(g^{ij}g^{kl}\partial_{k}g_{jl} - g^{ij}g^{kl}\partial_{j}g_{kl} \right) - g^{ij}\partial_ig^{kl}\partial_kg_{jl} - g^{kl}\partial_ig^{ij}\partial_{k}g_{jl} \\
&+ g^{kl}\partial_ig^{ij}\partial_jg_{kl} + g^{ij}\partial_ig^{kl}\partial_jg_{kl},\\
&=\partial_i\left(g^{ij}g^{kl}\partial_{k}g_{jl} - g^{ij}g^{kl}\partial_{j}g_{kl} \right) + f(g,\partial g).
\end{align*}
One can make the above useful by noticing that
\begin{align*}
\int_{A_R}x^a\left(g^{ij}g^{kl}\partial_{ik}g_{jl} - g^{ij}g^{kl}\partial_{ij}g_{kl} \right)dx&=\int_{A_R}x^a\partial_i\left(g^{ij}g^{kl}\partial_{k}g_{jl} - g^{ij}g^{kl}\partial_{j}g_{kl} \right)dx + \int_{A_R}x^af^{\Odd}(g,\partial g)dx,\\
&=\int_{A_R}x^a\partial_i\left(g^{ij}g^{kl}\partial_{k}g_{jl} - g^{ij}g^{kl}\partial_{j}g_{kl} \right)dx + O(R^{-\alpha}) + C,\\
&=\int_{S_R}x^a\left(g^{ij}g^{kl}\partial_{k}g_{jl} - g^{ij}g^{kl}\partial_{j}g_{kl} \right)\nu^id\omega_r \\
&- \int_{A_R}\left(g^{aj}g^{kl}\partial_{k}g_{jl} - g^{aj}g^{kl}\partial_{j}g_{kl} \right)dx + O(R^{-\alpha})+C
\end{align*}
where the $f^{\Odd}$-terms have been estimated using using (\ref{QuadraticTermsEstimate}) and $\partial_kg^{ij}=-g^{ia}g^{jb}\partial_kg_{ab}$. One can similarly see that
\begin{align*}
g^{aj}g^{kl}\partial_{k}g_{jl} - g^{aj}g^{kl}\partial_{j}g_{kl}&=\partial_lg_{al} - \partial_ag_{ll} + h(\delta,(g^{-1}-\delta),\partial g).
\end{align*}
Above, the function $h(\delta,(g^{-1}-\delta),\partial g)$ is a linear combination of terms which have either the form $\delta_{\cdot}(g^{-1}-\delta)_{\cdot}\partial g$ or $(g^{-1}-\delta)_{\cdot}(g^{-1}-\delta)_{\cdot}\partial g$. Then, notice that
\begin{align}\label{COM.3}
\begin{split}
(\delta_{\cdot}(g^{-1}-\delta)_{\cdot}\partial g)^{\Ev}&=\delta_{\cdot}(g^{-1}-\delta)^{\Ev}_{\cdot}(\partial g)^{\Ev} + \delta_{\cdot}(g^{-1}-\delta)^{\Odd}_{\cdot}(\partial g)^{\Odd}=O(|x|^{-3-\alpha}),\\
(g^{-1}-\delta)_{\cdot}(g^{-1}-\delta)_{\cdot}\partial g&=O(|x|^{-4}).
\end{split}
\end{align}
Therefore $h^{\Ev}=O(|x|^{-3-\alpha})$
implying
\begin{align*}
\int_{A_R}\left(g^{aj}g^{kl}\partial_{k}g_{jl} - g^{aj}g^{kl}\partial_{j}g_{kl} \right)dx&=\int_{A_R}\left(\partial_lg_{al} - \partial_ag_{ll} \right)dx + \int_{A_R}h^{\Ev}dx,\\
&=\int_{A_R}\left(\partial_lg_{al} - \partial_ag_{ll} \right)dx + O(R^{-\alpha}) + C,\\
&=\int_{S_R}\left(g_{al}\nu^l - g_{ll}\nu^a \right)d\omega_r + O(R^{-\alpha}) + C.
\end{align*}
Putting all the above together, we find
\begin{align*}
\int_{A_R}x^a\left(g^{ij}g^{kl}\partial_{ik}g_{jl} - g^{ij}g^{kl}\partial_{ij}g_{kl} \right)dx&=\int_{S_R}x^a\left(g^{ij}g^{kl}\partial_{k}g_{jl} - g^{ij}g^{kl}\partial_{j}g_{kl} \right)\nu^id\omega_r - \int_{S_R}\left(g_{al}\nu^l - g_{ll}\nu^a \right)d\omega_r \\
&+ O(R^{-\alpha}) + C.
\end{align*}
Notice now that
\begin{align*}
(g^{ij}g^{kl}\partial_{k}g_{jl} - g^{ij}g^{kl}\partial_{j}g_{kl})\nu^i&=(\partial_{l}g_{jl} - \partial_{j}g_{ll})\nu^j + \tilde{h}_i(\delta,g^{-1}-\delta,\partial g)\nu^{i},
\end{align*}
where again the functions $\tilde{h}_i(\delta,g^{-1}-\delta,\partial g)$ have the same structure as the function $h$ studied above. Since $\nu^ix^a$ is even, we again care about the even parts, which satisfy (\ref{COM.3}), implying
\begin{align*}
\int_{S_R}x^a\left(g^{ij}g^{kl}\partial_{k}g_{jl} - g^{ij}g^{kl}\partial_{j}g_{kl} \right)\nu^id\omega_r&=\int_{S_R}(\partial_{l}g_{jl} - \partial_{j}g_{ll})\nu^jx^ad\omega_r + \int_{S_R}\tilde{h}^{\Ev}_i\nu^{i}x^ad\omega_r,\\
&=\int_{S_R}(\partial_{l}g_{jl} - \partial_{j}g_{ll})\nu^jx^ad\omega_r + O(R^{-\alpha}).
\end{align*}
Putting the above together with (\ref{COM.2}), we finally find
\begin{align}\label{COM.4}
\int_{S_R}(\partial_{l}g_{jl} - \partial_{j}g_{ll})\nu^ix^ad\omega_r - \int_{S_R}\left(g_{al}\nu^l - g_{ll}\nu^a \right)d\omega_r=\int_{A_R}x^aR_gdx + O(R^{-\alpha}) + C.
\end{align}

Let us then define, for each $a=1,2,3$, the sequence $\{C^{a}_{R_k}\}_{k=1}^{\infty}$ by
\begin{align}
C^a_{R_k}\doteq \int_{S_{R_k}}(\partial_{l}g_{jl} - \partial_{j}g_{ll})\nu^ix^ad\mu - \int_{S_{R_k}}\left(g_{al}\nu^l - g_{ll}\nu^a \right)d\mu
\end{align}
for $R_k$ a sequence of monotonically increasing real numbers. Since $R_g\in L^1_{-4}(\mathbb{R}^3\backslash\overline{B}_1)$, notice that
\begin{align*}
\Vert R_g\Vert_{L^{1}_{-4}(\mathbb{R}^3\backslash\overline{B_{R_0}})	}= \int_{\mathbb{R}^3\backslash\overline{B}_{R_0}}|R_g|\sigma^{-(-4+3)}dx=\int_{\mathbb{R}^3\backslash\overline{B}_{R_0}}|R_g|\sigma dx,
\end{align*}
where $\sigma(x)=(1+|x|^2)^{\frac{1}{2}}$. Thus,
\begin{align*}
\int_{\mathbb{R}^3\backslash\overline{B_{R_0}}}|R_g||x|dx&\lesssim \Vert R_g\Vert_{L^{1}_{-4}(\mathbb{R}^3\backslash\overline{B_{R_0}})},
\end{align*}

Thus, we see that (\ref{COM.4}) together with $R_g\in L^1_{-4}(\mathbb{R}^3\backslash\overline{B}_1)$ implies that $\{C^{a}_{R_k}\}_{k=1}^{\infty}$ is Cauchy, and therefore the limit (\ref{COMreminder}) exists.
\end{proof}

We can now establish the main result of this section, which is basically a restatement of Theorem \ref{MainThmAE}, accompanied with the convergence of the associated center of mass.
\begin{theo}\label{ThmCederbaumSakovichConjecture}
Let $(M^3,g)$ be a smooth $W^{2,p}_{\tau}$-AE manifold with respect to a structure of infinity with coordinates $\{z^i\}_{i=1}^3$, satisfying 
\begin{enumerate}
\item[1.] $\tau\in (-1,-\frac{1}{2})$;
\item[2.] $R_g\in L^r_{-4-\epsilon}(M,\Phi_z)$ for some $r>3$ and $\epsilon>0$;
\item[3.] $C_{g}\in L^{p_1}_{\sigma}(M,\Phi_z)$ for some $-6<\sigma<-4$ and $p_1=\frac{3}{6+\sigma}$.
\end{enumerate}
Then, there is a structure of infinity with coordinates $\{\bar{z}^i\}_{i=1}^3$, which is $C^{1,\alpha}$-compatible with the original one, such that
\begin{align}\label{ImprovedDecayExpansion.2}
\begin{split}
g(\partial_{\bar{z}^i},\partial_{\bar{z}^j})&=\left(1+\frac{4C}{|\bar{z}|}\right)\delta_{ij} + O_1(|\bar{z}|^{-1-\alpha}),\\
\bar{z}(z)-\mathrm{Id}(z)&\in C^1_{1-\alpha}(\mathbb{R}^n\backslash\overline{B_{R_0}(0)}),
\end{split}
\end{align}
for some $\alpha>0$ and where the constant $C$ is given by (\ref{MassComputationGeneral}). Moreover, in these coordinates the center of mass (\ref{COMreminder}) is well defined.
\end{theo}
\begin{proof}
The first part of the theorem is derived from Theorem \ref{MainThmAE} since our hypotheses are even stronger in this case. Once (\ref{ImprovedDecayExpansion.2}) has been established, we have the following implication\footnote{For details about this, the reader can consult \cite[Corollary 3.1]{ALM}.}
\begin{align*}
R_g\in L^p_{-4-\epsilon}(M,\Phi_z) \text{ and } \bar{z}(z)-\mathrm{Id}(z)&\in C^1_{1-\alpha}(\mathbb{R}^n\backslash\overline{B_{R_0}(0)})\Longrightarrow R_g\in L^p_{-4-\epsilon}(M,\Phi_{\bar{z}}).
\end{align*}
Noticing that $L^p_{-4-\epsilon}(M,\Phi_{\bar{z}})\hookrightarrow L^1_{-4}(M,\Phi_{\bar{z}})$, then Lemma \ref{COMConvergenenceLemma} establishes that the center of mass is well-defined with respect to the structure of infinity given by $\Phi_{\bar{z}}$, in the sense that the limit in (\ref{COMreminder}) exists.
\end{proof}


\appendix
\markboth{Appendix}{Appendix}
\addcontentsline{toc}{section}{Appendices}
\renewcommand{\thesubsection}{\Alph{subsection}}
\numberwithin{equation}{section}
\numberwithin{theo}{section}
\numberwithin{remark}{section}

\section{Appendix I: Fredholmness of $\Delta_g$ on certain weighted spaces}\label{AppendixBartnik}


Let $(M^n,g)$ be a $W^{2,q}_{\tau}$-AE manifold with $q>\frac{n}{2}$ and $\tau<0$. The objective of this appendix is to review the Fredholm properties of the operator $\Delta_g$ acting on $W^{2,p}_{\delta}$-spaces. In order to motivate the relevance of such a revision, let us first notice that these properties have been the focus of attention of several highly influential papers. We would like to stress that precise Fredholm properties of the Euclidean Laplacian are known since the work of Robert C. McOwen in \cite{McOwen}. In the literature one can find several generalisations of such properties to general AE-manifolds with metrics of different degree of regularity, specially in the case where $\Delta_g$ acts on spaces $W^{2,p}_{\delta}$ with $\delta\in (2-n,0)$. In this last case, under mild assumptions $\Delta_g$ actually becomes an isomorphism as can be consulted, for instance, in \cite{Maxwell1}. Nevertheless, we are interested in understanding such properties outside of the interval $\delta\in (2-n,0)$, and the classical reference for such a treatment clearly is the celebrated paper by Robert Bartnik \cite{BartnikMass}, which covers such generalisations contemplating also low regularity metrics (see \cite[Proposition 2.2]{BartnikMass}). In this last paper, one of the ingredients along the analysis of the Fredholm properties of $\Delta_g$ is \cite[Proposition 1.6]{BartnikMass}, which concerns regularity properties of a priori very weak solutions of a class of second order equations with coefficients of very weak regularity. Although the proof of this result appears as a result of standard elliptic regularity combined with scaling techniques, actually, in our opinion, the regularity statement seems to be out of the scope of the regularity theory one may find in standard references, such as \cite{GilbargTrudinger} (compare with Theorem 8.8 in Chapter 8 and Theorems in Chapter 9 of \cite{GilbargTrudinger}). Typically, the results in such references require an a priori solution of $W^{1,2}$-regularity, which are most of the time presented with stronger regularity conditions on the coefficients (such as local $L^{\infty}$-bounds on all of them). Actually, finding a result which accommodates all the local regularity hypotheses of the operators treated in \cite[Definition 1.5]{BartnikMass} is not easy. 

Let us also comment that the analysis of the regularity of weak solutions to equations with coefficients of the regularity classes proposed in \cite[Definition 1.5]{BartnikMass} seems to be rather close to the recent treatment of regularity of solutions to elliptic problems with coefficients of very limited regularity given in \cite{Brezis1,Regularity1,Regularity2} (among others), which was motivated by a conjecture of J. Serrin \cite{Serrin}. Within the PDE program associated with these papers, there seem to be very interesting counterexamples limiting the extent of possible regularity theory, such as \cite{Serrin,JIN}. In that spirit, we believe it is instructive to present a self-contained version of the regularity properties stated in \cite[Proposition 1.6]{BartnikMass} and their application to obtain Fredholm statements for $\Delta_g:W^{2,p}_{\delta}\to L^p_{\delta-2}$, for metrics of low regularity and outside the interval $\delta\in (2-n,0)$. We anticipate that we are unable recover the full statement of \cite[Proposition 1.6]{BartnikMass}, which should be compared with Theorem \ref{BartniksProp1.6} below. Nevertheless, the regularity we do obtain seems to be enough to get the full extent of \cite[Proposition 2.2]{BartnikMass} (compare with Corollary \ref{WeightedFredholmCoro}) .

\medskip
With the above motivations in mind, we start by noticing that locally:
\begin{align}\label{LapaclianLocApp1}
\Delta_g=g^{ij}\partial_{ij}-g^{ij}\Gamma^l_{ij}\partial_l.
\end{align}
Therefore, using the notations of Section \ref{Preliminaries}, on any bounded chart with smooth boundary $\{U,x^i\}_{i=1}^n$ we have $\Delta_g\in \mathcal{L}^2(W^{2,q})$, and hence $\Delta_g:W_0^{2,p}(U)\to L^p(U)$ is a bounded map for any $1<p\leq q$ due to the same reasoning as in Lemma \ref{ContinuityPropsGeneralOps}. Furthermore, the multiplication property of Theorem \ref{AEWeightedEmbeedings} also shows that $\Vert \Delta_g(\eta u)\Vert_{L^{p}_{\delta-2}(M)}\leq C\Vert \eta u\Vert_{W^{2,p}_{\delta}(M)}$ for any $u\in W^{2,p}_{\delta}(M)$, $\eta$ a cut-off function supported in an end of $M$ equal to one near infinity, $1<p\leq q$ and $\delta\in \mathbb{R}$. Thus, appealing to a partition of unity, we see that:
\begin{prop}
Let $(M^n,g)$ be a $W^{2,q}_{\tau}$-AE manifold with $q>\frac{n}{2}$ and $\tau<0$. Then, the operator $\Delta_g$ is a continuous map from $W^{2,p}_{\delta}(M)\to L^p_{\delta-2}(M)$ for all $1<p\leq q$ and $\delta\in \mathbb{R}$.
\end{prop}

We will now concentrate in proving a priori estimates, regularity and decay properties for solutions to $\Delta_gu=f$. Our presentation will be based on the classical paper \cite{NirenbergWalker}, adapting their results to our notations and extending them to our context.

\begin{lem}\label{RegularityAtInfinityCteCoef}
For any given real number $\rho$, if $u\in W^{2,p}_{loc}(\mathbb{R}^n)\cap L^p_{\rho}(\mathbb{R}^n)$, $1<p<\infty$, and $\Delta u\in L^p_{\rho-2}(\mathbb{R}^n)$, then $u\in W^{2,p}_{\rho}(\mathbb{R}^n)$ and there is a fixed constant $C>0$, independent of $u$, such that:
\begin{align}\label{EstimateAtInfinityCteCoef}
\Vert u\Vert_{W^{2,p}_{\rho}(\mathbb{R}^n)}\leq C\left( \Vert \Delta u\Vert_{L^p_{\rho-2}(\mathbb{R}^n)} + \Vert u\Vert_{L^p_{\rho}(\mathbb{R}^n)}\right)
\end{align}
\end{lem}
\begin{proof}
We will appeal to the scaling properties associated with weighted spaces. First, let $A_r=B_{2r}\backslash \overline{B_{r}}$ and notice that using $(S_ru)(x)=u(rx)$ and $\partial^{\beta}_x(S_ru)(x)=r^{|\beta|}\partial^{\beta}u(rx)$ for all $|\beta|\leq 2$, we find that $r^{|\beta|p-n}\Vert \partial^{\beta}u \Vert^p_{L^p(A_r)}=\Vert \partial^{\beta}(S_ru)\Vert^p_{L^p(A_1)}$. Then, we use interior elliptic estimates to write
\begin{align*}
\int_{A_r}r^{|\beta|p-n}|\partial_x^{\beta}u(x)|^pdx&=\int_{A_1}|\partial^{\beta}_x(S_ru)(x)|^pdx\leq \Vert S_ru\Vert^p_{W^{2,p}(A_1)},\\
&\leq C \int_{\frac{1}{4}\leq |x|\leq 4}\left( |\Delta(S_ru)(x)|^p + |(S_ru)(x)|^p \right)dx,\\
&=C \int_{\frac{r}{4}\leq |x|\leq 4r}\left( r^{2p-n}|\Delta u(x)|^p + r^{-n}|u(x)|^p \right)dx,
\end{align*}
where the important point is that $C>0$ depends on $\Delta,p,n$ but neither on $u$ nor $r>0$. Let us now multiply the above inequality by $r^{-\rho p}$, to get
\begin{multline*}
\int_{r\leq |x|\leq 2r}\big|r^{-(\rho - |\beta| + \frac{n}{p})}|\partial_x^{\beta}u(x)|\big|^pdx\leq C \int_{\frac{r}{4}\leq |x|\leq 4r}\left( \big|r^{-(\rho - 2 + \frac{n}{p})}|\Delta u(x)|\big|^p + \big|r^{-\rho - \frac{n}{p}}|u(x)|\big|^p \right)dx,
\end{multline*}
Notice now that above implies that there is another constant $C'$, which may now also depend on $\rho$ and $|\beta|\leq m$ (but is still independent of $r$), for which we have
\begin{multline*}
\int_{r\leq |x|\leq 2r}\big||x|^{-(\rho - |\beta| + \frac{n}{p})}|\partial_x^{\beta}u(x)|\big|^pdx\leq C' \int_{\frac{r}{4}\leq |x|\leq 4r}\left( \big||x|^{-(\rho- 2 + \frac{n}{p})}|\Delta u(x)|\big|^p + \big||x|^{-\rho - \frac{n}{p}}|u(x)|\big|^p \right)dx.
\end{multline*}
For $r\geq 1$, by modifying $C'$ we can replace in the above inequality $|x|$ by $\sigma(x)$. Then, we can pick $r=2^{j}$, $j\in \mathbb{N}_0$, and sum over $j$ to get
\begin{align*}
\int_{\mathbb{R}^n\backslash B_{1}}\big||\sigma(x)|^{-(\rho - |\beta| + \frac{n}{p})}|\partial_x^{\beta}u(x)|\big|^pdx\leq C' \left(\Vert \Delta u \Vert^p_{L^p_{\rho-2}(\mathbb{R}^n)} + \Vert u\Vert^p_{L^p_{\rho}(\mathbb{R}^n)}\right),
\end{align*}
where the right-hand side is finite by hypotheses. Summing over $|\beta|\leq 2$, we find
\begin{align}
\Vert u\Vert^p_{W^{2,p}_{\rho}(\mathbb{R}^n\backslash B_1)}\leq C'' \left(\Vert \Delta u \Vert^p_{L^p_{\rho-2}(\mathbb{R}^n)} + \Vert u\Vert^p_{L^p_{\rho}(\mathbb{R}^n)}\right),
\end{align}
for some other constant $C''>0$, which implies the desired statements.
\end{proof}

The above lemma will be useful to produce estimates near infinity in more general AE manifolds treated in the following theorem:
\begin{theo}\label{WeightedEstimatesPreliminarThm}
Consider the Laplacian operator $\Delta_g$ of a $W_{\tau}^{2,q}$-AE metric, with $q>\frac{n}{2}$ and $\tau<0$, defined on $\mathbb{R}^n$. If $1<p\leq q$ and $\rho\in \mathbb{R}$, then there is a constant $C>0$ such that for all $u\in W^{2,p}_{\rho}(\mathbb{R}^n)$ we have the following elliptic estimate:
\begin{align}
\Vert u\Vert_{W^{2,p}_{\rho}(\mathbb{R}^n)} \leq C\left( \Vert \Delta_g u\Vert_{L^p_{\rho-2}(\mathbb{R}^n)} + \Vert u\Vert_{L^p_{\rho}(\mathbb{R}^n)} \right).
\end{align}
\end{theo}
\begin{proof}
First, use a cut-off function $\chi_R$ equal to one on $\overline{B_R(0)}$ and supported in $B_{2R}(0)$, constructed such that $|\partial^k\chi_R(x)|\leq C_kR^{-k}$ for any $k\in \mathbb{N}$ and a constant $C_k>0$ independent of $R$, and then decompose
\begin{align*}
u=\chi_Ru + (1-\chi_R) u=u_1+u_2,
\end{align*}
where $u_1=\chi_Ru$ and $u_2=(1-\chi_R) u$. Since on $B_{2R}(0)$ we have $u_1\in W^{2,p}_0(B_{2R}(0))$, one may deduce from classical interior estimates (such as \cite[Theorem 9.11]{GilbargTrudinger}) applied to the principal part of $\Delta_g$ together with a perturbation argument to introduce the first order terms in (\ref{LapaclianLocApp1}) and an appeal to Theorem \ref{BesselMultLocal} and interpolation, that there is a constant $C=C(g,p,n,R)>0$, such that\footnote{We notice that this estimate may also be directly deduced from \cite[Theorem 2.21]{MaxwellHolstRegularity}.}
\begin{align*}
\Vert u_1\Vert_{W^{2,p}(B_{2R}(0))}&\leq C(\Vert \Delta_g u_1\Vert_{L^p(B_{2R}(0))} + \Vert u_1\Vert_{L^p(B_{2R}(0))}),\\
&\leq C(\Vert \chi_R\Delta_g u\Vert_{L^p(B_{2R}(0))} + \Vert [\Delta_g,\chi_R] u\Vert_{L^p(B_{2R}(0))} + \Vert u_1\Vert_{L^p(B_{2R}(0))}).
\end{align*}
Then, from
\begin{align}\label{LowerOrderCritical.1}
[\Delta_g,\chi_R] u&=2\<\nabla\chi_R,\nabla u\>_g + u\Delta_g\chi_R,
\end{align}
since the derivatives of $\chi_R$ are bounded to any order, one finds
\begin{align*}
\Vert [L,\chi_R] u\Vert_{L^p(B_{2R}(0))}&\leq 2\Vert \<\nabla\chi_R,\nabla u\>_g\Vert_{L^p(B_{2R}(0))} + \Vert u\Delta_g\chi_R\Vert_{L^p(B_{2R}(0))}\leq C(g,R) \Vert u\Vert_{W^{1,p}(B_{2R}(0))}.
\end{align*}
Therefore, we then get a constant $C'=C(g,p,n,R)>0$ such that
\begin{align*}
\Vert u_1\Vert_{W^{2,p}(B_{2R}(0))}&\leq C'(\Vert \Delta_g u\Vert_{L^p(B_{2R}(0))} + \Vert u\Vert_{W^{1,p}(B_{2R}(0))}).
\end{align*}
Using the equivalence of the $L^{p}$ and $L^p$-weighted norms on bounded domains, we deduce that there is a fixed constant $C_1>0$, independent of $u$, and which is fixed once $R$ is fixed to any finite number, such that
\begin{align}\label{WeightedEstimatesInner}
\Vert u_1\Vert_{W^{2,p}_{\rho}(\mathbb{R}^n)}\leq C_1(g,n,p,R)\left( \Vert \Delta_g u\Vert_{L^p_{\rho}(\mathbb{R}^n)} + \Vert u\Vert_{W^{1,p}_{\rho}(\mathbb{R}^n)}\right).
\end{align}

To estimate $u_2$, we start decomposing
\begin{align*}
\Delta_g u_2&=\Delta u_2 + (g^{ij}-\delta_{ij})\partial_{ij}u_2 + g^{ij}\Gamma^l_{ij}\partial_lu_2.
\end{align*}
Since $u_2\in W^{2,p}_{\rho}(\mathbb{R}^n)$ by hypothesis, we can apply (\ref{EstimateAtInfinityCteCoef}) to $\Delta$ and use the above decomposition to obtain:
\begin{align*}
\Vert u_2\Vert_{W^{2,p}_{\rho}(\mathbb{R}^n)}&\leq C( \Vert \Delta u_2\Vert_{L^p_{\rho-2}(\mathbb{R}^n)} + \Vert u_2\Vert_{L^p_{\rho}(\mathbb{R}^n)}),\\
&\leq C( \Vert \Delta_g u_2\Vert_{L^p_{\rho-2}(\mathbb{R}^n)} + \Vert (g^{ij}-\delta_{ij})\partial_{ij}u_2 \Vert_{L^{p}_{\rho-2}(\mathbb{R}^n)} + \Vert g^{ij}\Gamma^l_{ij}\partial_lu_2 \Vert_{L^{p}_{\rho-2}(\mathbb{R}^n)} + \Vert u_2\Vert_{L^p_{\rho}(\mathbb{R}^n)}).
\end{align*}
Noticing that $\mathrm{supp} (u_2)\subset \mathbb{R}^n\backslash\overline{B_R(0)}$ and $W^{2,q}_{\tau}\hookrightarrow C^{0}_{\tau}$, gives us
\begin{align*}
\Vert (g^{ij}-\delta_{ij})\partial_{ij}u_2 \Vert_{L^{p}_{\rho-2}(\mathbb{R}^n)}&=\Vert (g^{ij}-\delta_{ij})\partial_{ij}u_2 \Vert_{L^{p}_{\rho-2}(\mathbb{R}^n\backslash\overline{B_R(0)})},\\
&\leq \left(\sup_{x\in \mathbb{R}^n\backslash\overline{B_R(0)}} |g^{ij}-\delta_{ij}\vert \right)\Vert \partial^{2}u_2 \Vert_{L^{p}_{\rho-2}(\mathbb{R}^n\backslash\overline{B_R(0)})},\\
&\leq \left(\Vert g^{-1}-\delta\Vert_{C^{0}_{\tau}}\right)\sigma(R)^{\tau}\Vert u_2 \Vert_{W^{2,p}_{\rho}(\mathbb{R}^n\backslash\overline{B_R(0)})}\\
&=C(g)R^{\tau}\Vert u_2 \Vert_{W^{2,p}_{\rho}(\mathbb{R}^n\backslash\overline{B_R(0)})}.
\end{align*}

Concerning the lower order terms, first notice that $g^{ij}\Gamma^l_{ij}\in W^{1,q}_{\tau-1}(\mathbb{R}^n)\hookrightarrow W^{1,q}_{\gamma-1}(\mathbb{R}^n)$ for any $0>\gamma>\tau$. Since $\mathrm{supp} (u_2)\subset \mathbb{R}^n\backslash\overline{B_R(0)}$, let us pick a fixed cut-off function $\eta\in C^{\infty}_0(B_1(0))$, $0\leq \eta\leq 1$, $\eta\equiv 1$  on $\overline{B_{\frac{1}{2}}(0)}$, and then scale it to $\eta_R(x)\doteq \eta(R^{-1}x)\in C^{\infty}_0(B_R(0))$, and $\eta_R|_{\overline{B_{\frac{R}{2}}(0)}}\equiv 1$. Then $(1-\eta_R)$ is supported in $\mathbb{R}^{n}\backslash \overline{B_{\frac{R}{2}}(0)}$ and is identically equal to one on $\mathbb{R}^{n}\backslash \overline{B_{R}(0)}$, implying that $(1-\eta_R)|_{\mathrm{supp}(u_2)}\equiv 1$. Therefore
\begin{align*}
g^{ij}\Gamma^l_{ij}\partial_lu_2=(1-\eta_R)g^{ij}\Gamma^l_{ij}\partial_lu_2.
\end{align*}
Then, appealing to the multiplication property in Theorem \ref{SobolevPorpsAE}, we find that
\begin{align*}
W^{1,q}_{\gamma-1}(\mathbb{R}^n)\otimes W^{1,p}_{\rho - 1}(\mathbb{R}^n)\hookrightarrow L^p_{\rho-2}(\mathbb{R}^n),
\end{align*}
holds as long as $\tau<\gamma<0$, since $q>\frac{n}{2}$. Thus, there is a constant $C>0$, independent of $u_2$ and $R$, such that
\begin{align}\label{GammaEsti}
\begin{split}
\Vert g^{ij}\Gamma^l_{ij}\partial_lu_2\Vert_{L^{p}_{\rho-2}(\mathbb{R}^n)}&\leq C \Vert (1-\eta_R)g^{ij}\Gamma^l_{ij}\Vert_{W^{1,q}_{\gamma-1}(\mathbb{R}^n)}\Vert \partial_lu_2\Vert_{W^{1,p}_{\rho - 1}(\mathbb{R}^n)}\\
&\leq C'(g) \left(\sum_{|i,j,l|\leq n}\Vert  (1-\eta_R)\Gamma^l_{ij}\Vert_{W^{1,q}_{\gamma-1}(\mathbb{R}^n)}\right)\Vert u_2\Vert_{W^{2,p}_{\rho }(\mathbb{R}^n)}.
\end{split}
\end{align}
Notice now that
\begin{align}\label{GammaEsti.0}
\Vert  (1-\eta_R)\Gamma^l_{ij}\Vert_{L^{q}_{\gamma-1}(\mathbb{R}^n)}\leq 2 \Vert  \Gamma^l_{ij}\Vert_{L^{q}_{\gamma-1}(\mathbb{R}^n\backslash \overline{B_{\frac{R}{2}}(0)})}.
\end{align}
Similarly, since
\begin{align*}
\partial((1-\eta_R)\Gamma^l_{ij})&=(1-\eta_R)\partial\Gamma^l_{ij} - \partial\eta_R \Gamma^l_{ij},\\
&=(1-\eta_R)\partial\Gamma^l_{ij} - R^{-1}(\partial\eta)(R^{-1}x) \Gamma^l_{ij},
\end{align*}
we find
\begin{align}\label{GammaEsti.1}
\Vert  \partial ((1-\eta_R)\Gamma^l_{ij})\Vert_{L^{q}_{\gamma-2}(\mathbb{R}^n)}\leq \Vert  (1-\eta_R)\partial\Gamma^l_{ij}\Vert_{L^{q}_{\gamma-2}(\mathbb{R}^n)} + \Vert  R^{-1}(\partial\eta)(R^{-1}\cdot) \Gamma^l_{ij}\Vert_{L^{q}_{\gamma-2}(\mathbb{R}^n)}.
\end{align}
But since
\begin{align*}
\int_{\mathbb{R}^n}|R^{-1}(\partial\eta)(R^{-1}x) \Gamma^l_{ij}|^q\sigma^{-(\gamma-2+\frac{n}{q})q}dx&=\int_{B_{R}\backslash\overline{B_{\frac{R}{2}}}}|R^{-1}(\partial\eta)(R^{-1}x) \Gamma^l_{ij}|^q\sigma^{-(\gamma-2+\frac{n}{q})q}dx,\\
&\leq \left(\sup_{x\in B_{R}\backslash\overline{B_{\frac{R}{2}}}}|\partial\eta(R^{-1}x)|\right)^{q}\int_{B_{R}\backslash\overline{B_{\frac{R}{2}}}}R^{-q}| \Gamma^l_{ij}|^q\sigma^{-(\gamma-2+\frac{n}{q})q}dx,\\
&\leq \Vert \partial\eta\Vert^q_{L^{\infty}(B_1(0))}\int_{B_{R}\backslash\overline{B_{\frac{R}{2}}}}| \Gamma^l_{ij}|^q|x|^{-q}\sigma^{-(\gamma-2+\frac{n}{q})q}dx,\\
&\leq C\Vert \partial\eta\Vert^q_{L^{\infty}(B_1(0))}\int_{B_{R}\backslash\overline{B_{\frac{R}{2}}}}| \Gamma^l_{ij}|^q\sigma^{-q}\sigma^{-(\gamma-2+\frac{n}{q})q}dx,\\
&\leq C\Vert \partial\eta\Vert^q_{L^{\infty}(B_1(0))}\int_{B_{R}\backslash\overline{B_{\frac{R}{2}}}}| \Gamma^l_{ij}|^q\sigma^{-(\gamma-1+\frac{n}{q})q}dx,
\end{align*}
where the constant $C=C(q)>0$ appearing in the fourth-line above arises from the estimate $|x|^{-q}\leq C(q)(1+|x|^2)^{-\frac{q}{2}}$ for any $|x|\geq 1$. The above estimate implies that 
\begin{align*}
\Vert  R^{-1}(\partial\eta)(R^{-1}\cdot) \Gamma^l_{ij}\Vert_{L^{q}_{\gamma-2}(\mathbb{R}^n)}\leq C(n,q)\Vert \Gamma^l_{ij}\Vert_{L^{q}_{\gamma-1}(\mathbb{R}^n\backslash\overline{B_{\frac{R}{2}}(0)})},
\end{align*}
which put together with (\ref{GammaEsti.1}) implies
\begin{align*}
\Vert  \partial ((1-\eta_R)\Gamma^l_{ij})\Vert_{L^{q}_{\gamma-2}(\mathbb{R}^n)}\leq C'(n,q)\Vert \Gamma^l_{ij}\Vert_{W^{1,q}_{\gamma-1}(\mathbb{R}^n\backslash\overline{B_{\frac{R}{2}}(0)})}.
\end{align*}
This can be put together with (\ref{GammaEsti.0}) to obtain
\begin{align}\label{GammaEsti.2}
\Vert  ((1-\eta_R)\Gamma^l_{ij})\Vert_{W^{1,q}_{\gamma-1}(\mathbb{R}^n)}\leq C''(n,q)\Vert \Gamma^l_{ij}\Vert_{W^{1,q}_{\gamma-1}(\mathbb{R}^n\backslash\overline{B_{\frac{R}{2}}(0)})}.
\end{align}
Now, since $\tau<\gamma<0$, choosing $R$ large enough, we furthermore have
\begin{align}\label{GammaEsti.3}
\begin{split}
\Vert \Gamma^l_{ij}\Vert_{W^{1,q}_{\gamma-1}(\mathbb{R}^n\backslash\overline{B_{\frac{R}{2}}(0)})}&\leq \sigma^{-(\gamma-\tau)}\left(\frac{R}{2}\right)\Vert \Gamma^l_{ij}\Vert_{W^{1,q}_{\tau-1}(\mathbb{R}^n\backslash\overline{B_{\frac{R}{2}}(0)})},\\
&\leq CR^{-(\gamma-\tau)}\Vert \Gamma^l_{ij}\Vert_{W^{1,q}_{\tau-1}(\mathbb{R}^n\backslash\overline{B_{\frac{R}{2}}(0)})},
\end{split}
\end{align}
for a constant independent of $R>1$. Therefore, putting (\ref{GammaEsti}) together with (\ref{GammaEsti.2}) and (\ref{GammaEsti.3}), we find:
\begin{align*}
\Vert  g^{ij}\Gamma^l_{ij}\partial_{l}u_2\Vert_{L^{p}_{\rho-2}(\mathbb{R}^n\backslash\overline{B_R(0)})}&\leq C(n,p,q,g)R^{-(\gamma-\tau)}\Vert u_2\Vert_{W^{2,p}_{\rho }(\mathbb{R}^n)}.
\end{align*}
Thus, there are constants $C, C'>0$, independent booth of $u$ and $R$, such that,
\begin{align*}
\Vert u_2\Vert_{W^{2,p}_{\rho}(\mathbb{R}^n)}
&\leq C\big(  \Vert \Delta_gu_2\Vert_{L^{p}_{\rho-2}(\mathbb{R}^n)}+ C'(n,p,q,g)(R^{\tau}+R^{-(\gamma-\tau)})\Vert u_2 \Vert_{W^{2,p}_{\rho}(\mathbb{R}^n)} + \Vert u_2\Vert_{L^p_{\rho}(\mathbb{R}^n)}\big)
\end{align*}
Picking $R$ sufficiently large, we can absorb the highest order term in the right-hand side into the left-hand side and obtain
\begin{align*}
\Vert u_2\Vert_{W^{2,p}_{\rho}(\mathbb{R}^n)}
&\leq C_2(n,p,q,g,R)\big(  \Vert \Delta_gu_2\Vert_{L^{p}_{\rho-2}(\mathbb{R}^n)} + \Vert u_2\Vert_{L^p_{\rho}(\mathbb{R}^n)}\big).
\end{align*}
Once $R$ has been fixed so that the above estimate holds, we deduce that
\begin{align*}
\Vert u_2\Vert_{W^{2,p}_{\rho}(\mathbb{R}^n)}&\leq C_2( \Vert (1-\chi_R)\Delta_gu\Vert_{L^{p}_{\rho-2}(\mathbb{R}^n)} + \Vert [\Delta_g,1-\chi_R]u\Vert_{L^{p}_{\rho-2}(\mathbb{R}^n)} + \Vert u\Vert_{L^{p}_{\rho}(\mathbb{R}^n)}).
\end{align*}
Using that $[\Delta_g,1-\chi_R]u=-2\<\nabla\chi_R,\nabla u\>_g - u\Delta_g\chi_R$ is of first order on $u$ and is supported in $B_{2R}(0)$ (since each term involved at least one derivative on $\chi_R$), with the same kind of computations done after (\ref{LowerOrderCritical.1}), we find
\begin{align*}
\Vert u_2\Vert_{W^{2,p}_{\rho}(\mathbb{R}^n)}&\leq C_2(\Vert \Delta_gu\Vert_{L^{p}_{\rho-2}(\mathbb{R}^n)} + \Vert u\Vert_{W^{1,p}(B_{2R}(0))} + \Vert u\Vert_{L^{p}_{\rho}(\mathbb{R}^n)}).
\end{align*} 
Noticing now that the bounded set $B_{2R}(0)$ is fixed, since $R$ has been fixed already, we once more appeal to the equivalence of weighted and unweighted norms on compacts to estimate $\Vert u\Vert_{W^{1,p}(B_{2R}(0))}\leq C(R)\Vert u\Vert_{W^{1,p}_{\rho}(B_{2R}(0))}$, and thus we find
\begin{align*}
\Vert u_2\Vert_{W^{2,p}_{\rho}(\mathbb{R}^n)}&\leq C_3(n,p,q,g,R)\left(\Vert \Delta_g u\Vert_{L^{p}_{\rho-2}(\mathbb{R}^n )} + \Vert u\Vert_{W^{1,p}_{\rho}(\mathbb{R}^n)}\right).
\end{align*} 
Recalling once more that the constant in (\ref{WeightedEstimatesInner}) is also fixed, putting all of this together
\begin{align*}
\begin{split}
\Vert u\Vert_{W^{2,p}_{\rho}(\mathbb{R}^n)}&\leq \Vert u_1\Vert_{W^{2,p}_{\rho}(\mathbb{R}^n)} + \Vert u_2\Vert_{W^{2,p}_{\rho}(\mathbb{R}^n)}\leq C\left(\Vert \Delta_gu\Vert_{L^{p}_{\rho-2}(\mathbb{R}^n )} + \Vert u\Vert_{W^{1,p}_{\rho}(\mathbb{R}^n)} \right)
\end{split}
\end{align*}
Finally, the interpolation inequality of Theorem \ref{AEWeightedEmbeedings} finishes the proof.
\end{proof}

Let us now recall the following interpolation result, extracted from \cite[Lemma 3.1]{NirenbergWalker}:

\begin{lem}
Let $\chi\in C^{\infty}_0(B_{2}(0))$ be a a cut-off function which is equal to one on $\overline{B_{1}(0)}$, and define $\chi_R(x)\doteq \chi(\frac{x}{R})$ for any $R>1$. Given any $\epsilon>0$, there is a constant $C_{\epsilon}>0$, such that
\begin{align}\label{NirenbergWalkerInterpolation}
\sum_{|\alpha|=1}\Vert \chi_R\sigma^{1-\delta-\frac{n}{p}}\partial^{\alpha} u \Vert_{L^p(\mathbb{R}^n)}\leq \epsilon\sum_{|\alpha|\leq 2}\Vert \chi^2_R\sigma^{|\alpha| -\delta -\frac{n}{p}} \partial^{\alpha} u \Vert_{L^{p}(\mathbb{R}^n)} + C_{\epsilon}\Vert u\Vert_{L^p_{\delta}(\mathbb{R}^n)}
\end{align}
holds for all $u\in W^{2,p}_{loc}(\mathbb{R}^n)\cap L^p_{\delta}(\mathbb{R}^n)$.  
\end{lem}
\begin{remark}
Let us highlight that \cite[Lemma 3.1]{NirenbergWalker} is stated under the stronger assumption that $u\in W^{2,p}(\mathbb{R}^n)$, although it follows from the proof that the same result holds if $u\in W^{2,p}_{loc}(\mathbb{R}^n)\cap L^p_{\delta}(\mathbb{R}^n)$.
\end{remark}


We exploit the above lemma in the following theorem, which represents our version of \cite[Proposition 1.6]{BartnikMass}. 

\begin{theo}\label{BartniksProp1.6}
Let $(M^n,g)$ be a $W^{2,q}_{\tau}$-AE manifold with $q>\frac{n}{2}$ and $\tau<0$. Letting $p$ be a real number such that $1<p\leq q$, if $u\in L^{p}_{\delta}\cap L^{q'}_{loc}$, $\delta\in\mathbb{R}$, and $\Delta_gu\in L^p_{\delta-2}$, then $u\in W^{2,p}_{\delta}$ and 
\begin{align}\label{BartnikDecayEstimate}
\Vert u\Vert_{W^{2,p}_{\delta}}\leq C\left(\Vert \Delta_gu\Vert_{L^p_{\delta-2}} + \Vert u\Vert_{L^p_{\delta}} \right).
\end{align}
\end{theo}
\begin{proof}
First, under our hypotheses, locally we have $\Delta_g=g^{ij}\partial_{ij}-g^{ij}\Gamma^l_{ij}\partial_l\in \mathcal{L}^2(W^{2,q})$ and has vanishing zero-th order coefficient.\footnote{In the notations of \cite[Theorem 2.4]{MaxwellHolstRegularity}, this means that $\Delta_g\in \mathcal{L}^2_1(H^{2,q},U)$, where $U$ stands for the associated local chart.} Thus, by \cite[Theorem 2.21]{MaxwellHolstRegularity}, $u\in L^{q'}_{loc}$ and $\Delta_gu\in L^p_{loc}$ implies $u\in W^{2,p}_{loc}$ as long as $\frac{1}{p}\geq \frac{1}{q}-\frac{1}{n}$. Since $p\leq q$, this last condition holds and we find $u\in W^{2,p}_{loc}$. Then, assuming for simplicity $M$ has only one end, we consider a partition of unity $\{\eta_{\alpha}\}^{N}_{\alpha=1}$, with $\{\eta_i\}_{i=1}^{N-1}$ covering the compact core of $M$ and $\mathrm{supp}(\eta_i)\subset U_i$,\footnote{We take the $U_i$ sets as bounded subsets with smooth boundary, which is always possible.} while $\eta_N$ is supported in the end $M\backslash\mathcal{K}$ of $M$, so that $u=\sum_{\alpha}\eta_{\alpha}u$, and fixing $\eta\doteq \eta_N$ we concentrate on the case of $v\doteq \eta u$, which we treat as an element of $W^{2,p}_{loc}(\mathbb{R}^n)\cap L^{p}_{\delta}(\mathbb{R}^n)$ which is supported on $\mathbb{R}^n\backslash\overline{B_1(0)}$. Consider then a cut-off function $\chi\in C^{\infty}_0(B_{2}(0))$ which is equal to one on $B_{1}(0)$, and then define $\chi_R(x)\doteq \chi(\frac{x}{R})$ for some $R>1$. Then, $\chi_Rv$ is supported on $B_{2R}(0)\backslash\overline{B_1(0)}$ and
\begin{align*}
\partial(\chi^2_Rv)(x)&=\chi^2_R(x)\partial v(x) + \partial\chi^2_R(x) v(x)= \chi^2_R(x)\partial v(x) + 2R^{-1}\chi_R\partial\chi\left(\frac{x}{R}\right) v(x),\\
\partial^2(\chi^2_Rv)(x)&=\chi^2_R(x)\partial^2 v(x) + 4R^{-1}\chi_R\partial\chi\left(\frac{x}{R}\right) \partial v(x) + 2R^{-2}\partial\chi\left(\frac{x}{R}\right)\partial\chi\left(\frac{x}{R}\right) v(x) \\
&+ 2R^{-2}\chi_R\partial^2\chi\left(\frac{x}{R}\right) v(x)
\end{align*}
and we aim to estimate
\begin{align}\label{WeightedEstimatesStrong.1}
\begin{split}
\!\!\!\!\!\!\!\!\!\!\!\!\!\!\!\!\!\!\!\!\!\!\sum_{|\alpha|\leq 2}\Vert \chi^2_R\sigma^{|\alpha| -\delta -\frac{n}{p}} \partial^{\alpha} v \Vert_{L^{p}(\mathbb{R}^n)}&=\Vert \sigma^{2 -\delta -\frac{n}{p}} \chi^2_R\partial^{2} v \Vert_{L^{p}(\mathbb{R}^n)} + \Vert \sigma^{1 -\delta -\frac{n}{p}} \chi^2_R\partial v \Vert_{L^{p}(\mathbb{R}^n)}+ \Vert \chi^2_R\sigma^{ -\delta -\frac{n}{p}} v \Vert_{L^{p}(\mathbb{R}^n)},\\
&\leq C\Big( \Vert \sigma^{2 -\delta -\frac{n}{p}} \partial^2(\chi^2_Rv) \Vert_{L^{p}(\mathbb{R}^n)} + \Vert R^{-1}\chi_R\sigma^{2 -\delta -\frac{n}{p}} \partial v \Vert_{L^{p}(A_{R})} \\
&+ \Vert \sigma^{2 -\delta -\frac{n}{p}} R^{-2}v \Vert_{L^{p}(A_R)} + \Vert R^{-2}\chi_R\sigma^{2 -\delta -\frac{n}{p}}  v \Vert_{L^{p}(A_{R})}\\
& + \Vert \sigma^{1 -\delta -\frac{n}{p}} \partial (\chi^2_Rv) \Vert_{L^{p}(B_{2R}(0))} + \Vert R^{-1}\chi_R\sigma^{1 -\delta -\frac{n}{p}} \partial v \Vert_{L^{p}(A_{R})} \\
&+ \Vert \sigma^{ -\delta -\frac{n}{p}} \chi^2_R v \Vert_{L^{p}(B_{2R}(0))}\Big),\\
&\leq C'\Big(\Vert \chi^2_Rv \Vert_{W^{2,p}_{\delta}(\mathbb{R}^n)} + \Vert \chi_R\sigma^{1 -\delta -\frac{n}{p}} \partial v \Vert_{L^{p}(B_{2R}(0))} \\
&+ \Vert \chi_R\sigma^{ -\delta -\frac{n}{p}}  v \Vert_{L^{p}(B_{2R}(0))} + \Vert \sigma^{ -\delta -\frac{n}{p}}v  \Vert_{L^{p}(\mathbb{R}^n)}\Big).
\end{split}
\end{align}
where above $C,C'>0$ are independent of $R$, and in the second estimate we have denoted by $A_R\doteq B_{2R}(0)\backslash\overline{B_R(0)}$ highlighting that $\mathrm{supp}(\partial\chi_R)\subset A_R$. Also, in the third estimate we have made use of the fact that on $A_R$ one has $|x|\leq 2R$ and thus, there is a constant $C''>0$, independent of $R$, such that:
\begin{align*}
\Vert R^{-1}\chi_R\sigma^{2 -\delta -\frac{n}{p}} \partial v \Vert_{L^{p}(A_{R}(0))}&\leq 2 \Vert (2R)^{-1}\chi_R\sigma^{2 -\delta -\frac{n}{p}} \partial v \Vert_{L^{p}(A_{R}(0))}\leq 2\Vert |x|^{-1}\chi_R\sigma^{2 -\delta -\frac{n}{p}} \partial v \Vert_{L^{p}(A_{R}(0))},\\
&\leq C'' \Vert \chi_R\sigma^{1 -\delta -\frac{n}{p}} \partial v \Vert_{L^{p}(A_{R}(0))},\\
\Vert \sigma^{2 -\delta -\frac{n}{p}} R^{-2}v \Vert_{L^{p}(A_R)}&\leq C''  \Vert \sigma^{ -\delta -\frac{n}{p}} v \Vert_{L^{p}(A_R)}\leq C''\Vert \sigma^{ -\delta -\frac{n}{p}} v \Vert_{L^{p}(\mathbb{R}^n)},\\
\Vert R^{-2}\chi_R\sigma^{2 -\delta -\frac{n}{p}}  v \Vert_{L^{p}(A_{R}(0))}&\leq C'' \Vert \sigma^{ -\delta -\frac{n}{p}}  \chi_Rv \Vert_{L^{p}(A_{R}(0))},\\
\Vert R^{-1}\chi_R\sigma^{1 -\delta -\frac{n}{p}} \partial v \Vert_{L^{p}(A_{R}(0))}&\leq C'' \Vert \sigma^{ -\delta -\frac{n}{p}} \chi_R\partial v \Vert_{L^{p}(A_{R}(0))}\leq C'' \Vert \sigma^{ 1 -\delta -\frac{n}{p}} \chi_R\partial v \Vert_{L^{p}(A_{R}(0))} 
\end{align*}

We can apply the estimates Theorem \ref{WeightedEstimatesPreliminarThm} to the first term in the right-hand side of (\ref{WeightedEstimatesStrong.1}), so that 
\begin{align*}
\Vert \chi^2_Rv\Vert_{W^{2,p}_{\delta}(\mathbb{R}^n)}&\leq C_0( \Vert \Delta_g(\chi^2_Rv)\Vert_{L^{p}_{\delta-2}(\mathbb{R}^n)} + \Vert \chi^2_Rv\Vert_{L^{p}_{\delta}(\mathbb{R}^n)}),\\
&\leq C_0(\Vert \chi^2_R\Delta_gv\Vert_{L^{p}_{\delta-2}(\mathbb{R}^n)} + \Vert [\Delta_g,\chi^2_R]v\Vert_{L^{p}_{\delta-2}(\mathbb{R}^n)} + \Vert \chi^2_Rv\Vert_{L^{p}_{\delta}(\mathbb{R}^n)}),
\end{align*}
where 
\begin{align*}
[\Delta_g,\chi^2_R]v=2\< \nabla \chi^2_R,\nabla v\>_g + v\Delta_g\chi^2_R,
\end{align*}
which is compactly supported in the annulus $A_{R}= B_{2R}(0)\backslash \overline{B_{R}(0)}$. Therefore, 
\begin{align}\label{WalkerNirenbergEstimates.1}
\big\Vert [\Delta_g,\chi^2_R]v \big\Vert_{L^{p}_{\delta-2}(A_{R})}&\leq 2\big\Vert |\nabla \chi^2_R| \:|\nabla v| \big\Vert_{L^{p}_{\delta-2}(A_{R})}+ \Vert v\Delta_g\chi^2_R\Vert_{L^{p}_{\delta-2}(A_{R})},
\end{align}
and using that
\begin{align*}
\partial_k\chi^2_R&=2R^{-1}\chi_R\partial_k\chi\left(\frac{x}{R}\right),\\
\partial_{lk}\chi^2_R&=2R^{-2}\chi_R\partial_{lk}\chi\left(\frac{x}{R}\right) + 2R^{-2}\partial_l\chi\left(\frac{x}{R}\right)\partial_k\chi\left(\frac{x}{R}\right),
\end{align*}
we find a constant $C=C(g)>0$, independent of $R$, such that:
\begin{align*}
\Vert v\Delta_g\chi^2_R\Vert_{L^{p}_{\delta-2}(A_{R})}&\leq C(g)( \Vert v\partial^2\chi^2_R\Vert_{L^{p}_{\delta-2}(A_{R})} + \Vert v\Gamma\partial\chi^2_R\Vert_{L^{p}_{\delta-2}(A_{R})}),\\
&\leq C'(g)( R^{-2}\Vert v\Vert_{L^{p}_{\delta-2}(A_{R})}  +  R^{-1}\Vert \partial\chi (R^{-1}\cdot)\Gamma \chi_Rv\Vert_{L^{p}_{\delta-2}(\mathbb{R}^n)}),\\
&\leq C''(n,p,q,g)( R^{-2}\Vert v\Vert_{L^{p}_{\delta-2}(A_{R})}  +  R^{-1}\Vert \partial\chi (R^{-1}\cdot)\Gamma\Vert_{W^{1,q}_{\tau-1}(\mathbb{R}^n\backslash\overline{B_1})} \Vert\chi_Rv\Vert_{W^{1,p}_{\delta-1}(\mathbb{R}^n\backslash\overline{B_1}))}),\\
&\leq C'''(n,p,q,g)\Big( R^{-2}\Vert v\Vert_{L^{p}_{\delta-2}(A_{R})}  +  R^{-1}\left( \Vert\chi_Rv\Vert_{L^{p}_{\delta-1}(B_{2R}\backslash\overline{B_1})} + \Vert \partial(\chi_Rv)\Vert_{L^{p}_{\delta-2}(B_{2R}\backslash\overline{B_1})}\right)\Big),\\
&\leq C'''(n,p,q,g)( R^{-2}\Vert v\Vert_{L^{p}_{\delta-2}(A_{R})}  +  R^{-1} \Vert\chi_Rv\Vert_{L^{p}_{\delta-1}(B_{2R}\backslash\overline{B_1})} + R^{-1}\Vert \chi_R\partial v\Vert_{L^{p}_{\delta-2}(B_{2R}\backslash\overline{B_1})}),
\end{align*}
where in the third inequality we used the multiplication property on $\mathbb{R}^n$, then the compact support of $\chi$ to localise, and then the factor $\Vert \partial\chi (R^{-1}\cdot)\Gamma\Vert_{W^{1,q}_{\tau-1}(\mathbb{R}^n\backslash\overline{B_1})}$ is estimated along the same lines as (\ref{GammaEsti.2}). Using these expressions back in (\ref{WalkerNirenbergEstimates.1}), we find a new constant $C=C(n,p,q,g)>0$ such that:
\begin{align*}
\big\Vert [\Delta_g,\chi^2_R]v \big\Vert_{L^{p}_{\delta-2}(A_{R})}&\leq C( R^{-1}\Vert \chi_R \nabla v\Vert_{L^{p}_{\delta-2}(B_{2R}\backslash\overline{B_1})}+ R^{-2}\Vert v\Vert_{L^{p}_{\delta-2}(B_{2R}\backslash\overline{B_1}))} +  R^{-1} \Vert\chi_Rv\Vert_{L^{p}_{\delta-1}(B_{2R}\backslash\overline{B_1})}),\\
&=C\big(  R^{-1}\Vert \chi_R\nabla v \sigma^{-\delta+2-\frac{n}{p}}\Vert_{L^{p}(B_{2R}\backslash\overline{B_1}))}+ R^{-2}\Vert v \sigma^{-\delta+2-\frac{n}{p}}\Vert_{L^{p}(B_{2R}\backslash\overline{B_1}))} \\
&+  R^{-1} \Vert\chi_Rv\sigma^{-\delta+1-\frac{n}{p}}\Vert_{L^{p}(B_{2R}\backslash\overline{B_1})}\big),\\
&\leq C\big(\Vert \chi_R\nabla v |x|^{-1}\sigma^{-\delta+2-\frac{n}{p}}\Vert_{L^{p}(B_{2R}\backslash\overline{B_1}))}+ \Vert v |x|^{-2}\sigma^{-\delta+2-\frac{n}{p}}\Vert_{L^{p}(B_{2R}\backslash\overline{B_1}))} ,\\
&+   \Vert v|x|^{-1}\sigma^{-\delta+1-\frac{n}{p}}\Vert_{L^{p}(B_{2R}\backslash\overline{B_1})}\big),\\
&\leq C'(n,p,q,g)\big( \Vert \chi_R\nabla v \sigma^{-\delta+1-\frac{n}{p}}\Vert_{L^{p}(B_{2R}\backslash\overline{B_1})}+\Vert v \sigma^{-\delta-\frac{n}{p}}\Vert_{L^{p}(B_{2R}\backslash\overline{B_1}))}\big) ,\\
&\leq C'(n,p,q,g)\big( \Vert \chi_R\nabla v \sigma^{1-\delta-\frac{n}{p}}\Vert_{L^{p}(B_{2R}\backslash\overline{B_1})}+\Vert v \Vert_{L^{p}_{\delta}(\mathbb{R}^n)}\big),
\end{align*}
We therefore find another constant $C=C(n,p,q,g)>0$ such that
\begin{align*}
\Vert \chi^2_Rv\Vert_{W^{2,p}_{\delta}(\mathbb{R}^n)}&\leq C( \Vert \chi^2_R\Delta_gv\Vert_{L^{p}_{\delta-2}(\mathbb{R}^n)} + \Vert \chi_R\partial v \sigma^{1-\delta-\frac{n}{p}}\Vert_{L^{p}(\mathbb{R}^n)}+ \Vert v\Vert_{L^{p}_{\delta}(\mathbb{R}^n)}),
\end{align*}
which gives us
\begin{align*}
\sum_{|\alpha|\leq 2}\Vert \chi^2_R\sigma^{|\alpha| -\delta -\frac{n}{p}} \partial^{\alpha} v \Vert_{L^{p}(\mathbb{R}^n)}&\leq C\Big( \Vert \chi^2_R\Delta_gv\Vert_{L^{p}_{\delta-2}(\mathbb{R}^n)} + \Vert \chi_R\sigma^{1-\delta-\frac{n}{p}}\partial v \Vert_{L^{p}(\mathbb{R}^n)}+ \Vert v\Vert_{L^{p}_{\delta}(\mathbb{R}^n)} \Big)
\end{align*}
We can estimate the first order term above using (\ref{NirenbergWalkerInterpolation}), so as to obtain
\begin{align*}
\sum_{|\alpha|\leq 2}\Vert \chi^2_R\sigma^{|\alpha| -\delta -\frac{n}{p}} \partial^{\alpha} v \Vert_{L^{p}(\mathbb{R}^n)}&\leq C\Big( \Vert \chi^2_R\Delta_gv\Vert_{L^{p}_{\delta-2}(\mathbb{R}^n)} + \epsilon\sum_{|\alpha|\leq 2}\Vert \chi^2_R\sigma^{|\alpha| -\delta -\frac{n}{p}} \partial^{\alpha} u \Vert_{L^{p}(\mathbb{R}^n)} +(1+ C_{\epsilon})\Vert u\Vert_{L^p_{\delta}(\mathbb{R}^n)}  \Big)
\end{align*}
Therefore, picking $\epsilon<\frac{1}{C(n,p,q,g)}$ in (\ref{NirenbergWalkerInterpolation}), we find that there is some other constant $C'>0$ independent of $R>1$ and $v$ such that
\begin{align*}
\sum_{|\alpha|\leq 2}\Vert \chi^2_R\sigma^{|\alpha| -\delta -\frac{n}{p}} \partial^{\alpha} v \Vert_{L^{p}(\mathbb{R}^n)}&\leq C'\Big( \Vert \chi^2_R\Delta_gv\Vert_{L^{p}_{\delta-2}(\mathbb{R}^n)}  + \Vert u\Vert_{L^p_{\delta}(\mathbb{R}^n)}  \Big)
\end{align*}
As long as $\Delta_gv\in L^{p}_{\delta-2}(\mathbb{R}^n)$, we can take the limit $R\rightarrow\infty$ in the above inequality to obtain $v\in W^{2,p}_{\delta}(\mathbb{R}^n)$, with the estimate:
\begin{align*}
\Vert v\Vert_{W^{2,p}_{\delta}(\mathbb{R}^n)}= \sum_{|\alpha|\leq 2}\Vert \sigma^{|\alpha| -\delta -\frac{n}{p}} \partial^{\alpha} v \Vert_{L^{p}(\mathbb{R}^n)}&\leq C'\Big( \Vert \Delta_gv\Vert_{L^{p}_{\delta-2}(\mathbb{R}^n)}  + \Vert u\Vert_{L^p_{\delta}(\mathbb{R}^n)}  \Big).
\end{align*}

Now recalling that $v=\eta u$, this already proves that $u\in W^{2,p}_{\delta}(M)$, and putting together the above estimate with the corresponding interior estimates for each $\eta_i u$ supported in the compact core, we find (\ref{BartnikDecayEstimate}).
\end{proof}

\medskip
\begin{cor}\label{CoroGeneralDecayBootstrap}
Let $(M^n,g)$ be a $C^{\infty}_{\tau}$-AE manifold with $\tau<0$. If $u\in L^{p}_{\delta}(M)$, $\delta\in\mathbb{R}$, $p>1$, and $\Delta_gu\in W^{k,p}_{\delta-2}(M)$, then $u\in W^{k+2,p}_{\delta}(M)$.
\end{cor}
\begin{proof}
The case $k=0$ is given by Theorem \ref{BartniksProp1.6}, and to establish the general case we work by induction. Assuming the claim holds for some $k\geq 0$ and that $\Delta_gu\in W^{k+1,p}_{\delta-2}$, we have a priori from the inductive hypothesis that $u\in W^{k+2,p}_{\delta}$. Considering a cut-off function $\eta$ equal to zero in the compact core of $M$ and to one in a neighbourhood of infinity, we have
\begin{align}\label{GeneralDecayBootstrap}
\Delta_g(\eta \partial_i u)&=\eta\Delta_g(\partial_i u) + 2\< \nabla\eta,\nabla(\partial_i u) \>_g + \partial_iu\Delta_g\eta.
\end{align}
Noticing that
\begin{align*}
\eta\Delta_g(\partial_i u)&=\eta\left(g^{ab}\partial_{ab}(\partial_iu) - g^{ab}\Gamma^l_{ab}\partial_{li}u \right),\\
&=\eta\left(\partial_i\left(g^{ab}\partial_{ab}u\right) - \partial_ig^{ab}\partial_{ab}u - \partial_i\left(g^{ab}\Gamma^l_{ab}\partial_{l}u \right) + g^{ab}\partial_i\Gamma^l_{ab}\partial_{l}u + \partial_ig^{ab}\Gamma^l_{ab}\partial_{l}u \right),\\
&=\eta\left(\partial_i\left(\Delta_g u\right) - \partial_ig^{ab}\partial_{ab}u + g^{ab}\partial_i\Gamma^l_{ab}\partial_{l}u + \partial_ig^{ab}\Gamma^l_{ab}\partial_{l}u \right).
\end{align*}
Since the last two terms in (\ref{GeneralDecayBootstrap}) are compactly supported and 
\begin{align*}
&\eta \partial_ig^{ab}\partial_{ab}u\in C^{\infty}_{\tau-1}\otimes W^{k,p}_{\delta-2}\hookrightarrow W^{k,p}_{\delta-3},\\
&\eta g^{ab}\partial_i\Gamma^l_{ab}\partial_lu\in L^{\infty}\otimes C^{\infty}_{\tau-2}\otimes W^{k+1,p}_{\delta-1}\hookrightarrow W^{k,p}_{\delta-3},\\
&\eta \partial_ig^{ab}\Gamma^l_{ab}\partial_{l}u\in C^{\infty}_{\tau-1}\otimes C^{\infty}_{\tau-1}\otimes W^{k+1,p}_{\delta-1}\hookrightarrow W^{k,p}_{\delta-3}
\end{align*}
Thus, $\Delta_g(\eta \partial_i u)-\eta \partial_i\left(\Delta_gu\right) \in W^{k,p}_{\delta-3}$, and $\Delta_gu\in W^{k+1,p}_{\delta-2}$. Then, the inductive hypothesis implies
\begin{align*}
\Delta_g(\eta \partial_i u)\in W^{k,p}_{\delta-3} \Longrightarrow \eta\partial_iu\in W^{k+2,p}_{\delta-1}.
\end{align*}
Therefore, we find
\begin{align*}
\Delta_gu\in W^{k+1,p}_{\delta-2}\Longrightarrow \eta u\in W^{k+2,p}_{\delta} \text{ and } \eta\partial_iu\in W^{k+2,p}_{\delta-1}\Longrightarrow \eta u\in W^{k+3,p}_{\delta},
\end{align*}
which proves the inductive step and general claim follows.
\end{proof}

\medskip
We now present semi-Fedholm properties for $\Delta_g$, but it is exactly at this point that the \emph{exceptional values} of the weight-parameter come into play crucially. Let recall that on an $n$-dimensional manifold they are given by the the real numbers  $\mathbb{Z}\backslash \{3-n,\cdots,-1\}$ and thus one says that $\delta$ is a non-exceptional weight if it does not lie in such a set of exceptional values. Furthermore, given $\delta\in \mathbb{R}$, one defines $k^{-}(\delta)$ to be the maximum exceptional value such that $k^{-}(\delta)\leq \delta$.

\begin{theo}\label{BartniksSemiFredThm}
Let $(M^n,g)$ be a $W^{2,q}_{\tau}$-AE manifold with $q>\frac{n}{2}$ and $\tau<0$. If $1<p\leq q$ and $\delta$ is non-exceptional, then the map $\Delta_g:W^{2,p}_{\delta}\to L^p_{\delta-2}$ is semi-Fredholm.
\end{theo}
\begin{proof}
The proof in this case follows from the arguments of \cite[Theorem 1.10]{BartnikMass}. First, let $S_R\hookrightarrow M$ be a large sphere contained in the end of $M$, define $\Omega_R$ to be the interior domain such that $\partial \Omega_R=S_R$ and $A_R\doteq \Omega_{2R}\backslash\overline{\Omega_R}$. Then use a cut-off function $\chi_R$ such that $0\leq \chi_R\leq 1$ satisfies $\chi_R(x)=1$ on $\overline{\Omega_R}$ while $\mathrm{supp}(\chi_R)\subset \Omega_{2R}$. Let $\Vert \cdot \Vert_{Op,R}$ denote the operator norm for maps $W^{2,p}_{\delta}(M)\to L^p_{\delta-2}(M)$ restricted to functions with support in $M\backslash\overline{\Omega_R}$, so that multiplication properties guarantee that $\Vert \Delta_g-\Delta \Vert_{Op,R}=o(1)$ as $R\rightarrow\infty$.\footnote{This is a difference between our proof and the one in \cite[Theorem 1.10]{BartnikMass}: there, the functional hypotheses on the coefficients are weaker. In particular, there the author admits operators with zero order coefficients in $L^{r}_{\tau-2}$ for $r>\frac{n}{2}$. Nevertheless, our multiplication properties do not guarantee $L^r_{\tau-2}(\mathbb{R}^n)\otimes W^{2,p}_{\delta}(\mathbb{R}^n)\hookrightarrow L^p_{\delta-2}(\mathbb{R}^n)$, even for $p>\frac{n}{2}$, unless one assumes a priori $r\geq p>\frac{n}{2}$ as well.} Then, one can follow the proof of \cite[Theorem 1.10]{BartnikMass} between equations (1.27) and (1.28) to obtain
\begin{align*}
\Vert u_{\infty}\Vert_{W^{2,p}_{\delta}(M)}\leq C\left( \Vert \Delta_gu \Vert_{L^p_{\delta-2}(M)} + \Vert u \Vert_{W^{1,p}_{\delta}(A_R)} \right),
\end{align*}
where $u_{\infty}=(1-\chi_R)u$, although one subtle clarification is in order. In \cite[Theorem 1.10]{BartnikMass}, the first estimate after equation (1.27) appeals to a sharp estimate obtained in \cite[Theorem 1.7, equation (1.22)]{BartnikMass}. This last sharp estimate is obtained (among other things) appealing to \cite[Proposition 1.6, equation (1.19)]{BartnikMass},\footnote{See \cite[Page 669]{BartnikMass}.} which (as stated within the proposition) we have commented seems to be out of the scope of standard regularity results, motivating all this revision in the first place. Notice nonetheless, that Theorem \ref{BartniksProp1.6} is our self-contained version of \cite[Proposition 1.6]{BartnikMass} with the estimate (\ref{BartnikDecayEstimate}) replacing equation (1.19) in \cite{BartnikMass}. In such a case, when Theorem \ref{BartniksProp1.6} is applied to the Euclidean Laplacian, then the exponent $p$ can be taken in the interval $1<p<\infty$, which would agree with \cite[Proposition 1.6]{BartnikMass} in this special case, which is the one used within the proof of \cite[Theorem 1.7, equation (1.22)]{BartnikMass}, and thus the argument in between equations (1.27) and (1.28) within the proof of \cite[Theorem 1.10]{BartnikMass} is sound.

Then, we can apply estimate (\ref{BartnikDecayEstimate}) to $u_0\doteq \chi_Ru\in W^{2,p}_{\delta}$ so as to obtain
\begin{align*}
\Vert u_0\Vert_{W^{2,p}_{\delta}(M)}&\leq C\left( \Vert \Delta_gu_0 \Vert_{L^p_{\delta-2}(M)} + \Vert u_0 \Vert_{L^{p}_{\delta-2}(\Omega_{2R})} \right),\\
&\leq C'\left( \Vert \chi_R\Delta_gu \Vert_{L^p_{\delta-2}(M)} + \Vert u \Vert_{W^{1,p}_{\delta}(A_R)} + \Vert \chi_Ru \Vert_{L^{p}_{\delta-2}(\Omega_{2R})} \right).
\end{align*}
Thus, for some other constant $C>0$ independent of $u$, it follows that:
\begin{align*}
\Vert u\Vert_{W^{2,p}_{\delta}(M)}\leq C\left( \Vert \Delta_gu \Vert_{L^p_{\delta-2}(M)} + \Vert u \Vert_{W^{1,p}_{\delta}(A_R)} + \Vert u \Vert_{L^{p}_{\delta-2}(\Omega_{2R})} \right),
\end{align*}
Having already fixed $R$ in the above procedure, the above estimate implies the following one:
\begin{align}\label{BartnikBrokenScale.1}
\Vert u\Vert_{W^{2,p}_{\delta}(M)}\leq C'\left( \Vert \Delta_gu \Vert_{L^p_{\delta-2}(M)} + \Vert u \Vert_{W^{1,p}(\Omega_{2R})}  \right),
\end{align}
for a constant $C'$ depending on such a choice of $R$. We may then apply the interpolation inequality
\begin{align*}
 \Vert u \Vert_{W^{1,p}(\Omega_{2R})}&\leq \epsilon\Vert u \Vert_{W^{2,p}(\Omega_{2R})} + C_{\epsilon}\Vert u \Vert_{L^{p}(\Omega_{2R})}\leq \epsilon C''(R)\Vert u \Vert_{W^{2,p}_{\delta}(\Omega_{2R})}+ C_{\epsilon}\Vert u \Vert_{L^{p}(\Omega_{2R})},\\
&\leq \epsilon C''(R)\Vert u \Vert_{W^{2,p}_{\delta}(M)}+ C_{\epsilon}\Vert u \Vert_{L^{p}(\Omega_{2R})},
\end{align*}
and absorb the second order term into the left hand side of (\ref{BartnikBrokenScale.1}) by picking $\epsilon$ small enough, finally obtaining:
\begin{align*}
\Vert u\Vert_{W^{2,p}_{\delta}}\leq C\left( \Vert \Delta_gu \Vert_{L^p_{\delta-2}(M)} + \Vert u \Vert_{L^{p}(\Omega_{2R})} \right),
\end{align*}
for sufficiently large $R$. This \emph{broken scale estimate} allows one to continue as in \cite[Theorem 1.10]{BartnikMass} to establish the semi-Fredholm property via standard functional analytic methods.
\end{proof}

\medskip
We now want to present Fredholm properties of $\Delta_g$, and for that one needs a more explicit characterisation of its adjoint map. With this in mind, let us recall:
\begin{align*}
W^{-k,p'}_{-\delta-n}(M)\doteq \left( W^{k,p}_{\delta}(M)\right)'.
\end{align*}

\begin{prop}\label{LapBelDualWeighted}
Let $(M^n,g)$ be a $W^{2,q}_{\tau}$-AE manifold, with $q>\frac{n}{2}$ and $\tau<0$. Letting $1<p\leq q$ and $\delta\in \mathbb{R}$, the Laplace operator extends by duality to a bounded map $\Delta_g:L^{p'}_{2-\delta-n}(M)\to W^{-2,p'}_{-\delta-n}(M)$ and, moreover, $\Delta^{*}_g\vert_{L^{p'}_{2-\delta-n}(M)}=\sqrt{\mathrm{det}(g)}\Delta_g\vert_{L^{p'}_{2-\delta-n}(M)}$.
\end{prop}
\begin{proof}
Let $U$ denote either a bounded coordinate patch in $M$ or the domain of an asymptotic coordinate chart on $\mathbb{R}^n\backslash\overline{B_1(0)}$, and in either case associate coordinates $\{x^i\}_{i=1}^n$, where we write $\Delta_gu=\frac{1}{\sqrt{\mathrm{det}(g)}}\partial_i\left( \sqrt{\mathrm{det}(g)}g^{ij}\partial_ju\right)$. By duality, we have that
\begin{align*}
\partial_i=(-\partial_i)^{*}:L^{p'}_{2-\delta-n}(U)&\to W^{-1,p'}_{1-\delta-n}(U),\\
\partial_i:W^{-1,p'}_{1-\delta-n}(U)&\to W^{-2,p'}_{-\delta-n}(U)
\end{align*}
are bounded maps. Similarly, given a function $f\in W^{2,q}_{\tau}(U)$, the multiplication map
\begin{align*}
M_f:W^{-1,p'}_{1-\delta-n}(U)&\to W^{-1,p'}_{1-\delta-n}(U),
\end{align*}
acting via $(M_fu)(v)=u(fv)$ for all $v\in \overset{\circ}{W}{}^{1,p}_{\delta-1}(U)$, will be bounded provided that $M_f$ is bounded on $\overset{\circ}{W}{}^{1,p}_{\delta-1}(U)$.\footnote{$\overset{\circ}{W}{}^{k,p}_{\delta}(U)$ denotes the closure of $C^{\infty}_0(U)$ in  the $W^{k,p}_{\delta}$-norm.} Using the multiplication property of Theorem \ref{SobolevPorpsAE}, this holds as long as $1<p\leq q$, $q>\frac{n}{2}$ and $\tau<0$, which hold by hypothesis. Therefore, we see that the map
\begin{align*}
L^{p'}_{2-\delta-n}(U)&\to W^{-1,p'}_{1-\delta-n}(U),\\
u&\mapsto \sqrt{\mathrm{det}(g)}g^{ij}\partial_ju,
\end{align*}
is well-defined and bounded, and then 
\begin{align*}
L^{p'}_{2-\delta-n}(U)&\to W^{-2,p'}_{-\delta-n}(U),\\
u&\mapsto \partial_i\left(\sqrt{\mathrm{det}(g)}g^{ij}\partial_ju\right),
\end{align*}
is again well-defined and bounded. Finally, we analyse the case of the multiplication map $M_f$ acting on $W^{-2,p'}_{-\delta-n}(U)$, which by duality will be a bounded map with range in $W^{-2,p'}_{-\delta-n}(U)$ iff $M_f:\overset{\circ}{W}{}^{2,p}_{\delta}(U)\to \overset{\circ}{W}{}^{2,p}_{\delta}(U)$ is a bounded map. Once more from Theorem \ref{SobolevPorpsAE}, this holds as long as $1<p\leq q$, $q>\frac{n}{2}$ and $\tau<0$. These statements and a partition of unity argument then show that $\Delta_g:L^{p'}_{2-\delta-n}(M)\to W^{-2,p'}_{\delta-n}(M)$ is a bounded map. 

Let us then appeal to the Riemannian measure induced by $g$ to identify $L^{p'}_{2-\delta-n}(M)\cong \left(L^{p}_{\delta-2}(M)\right)'$ via the $L^2(M,dV_g)$ inner-product:
\begin{align*}
u(v)=\int_Muvd V_g, \: u\in L^{p'}_{2-\delta-n}(M) \text{ and } v\in L^{p}_{\delta-2}(M).
\end{align*}
Then, take a coordinate cover of the form $\{U_{\alpha}\}_{\alpha=1}^{N+1}$, with $\{U_{\alpha}\}_{\alpha=1}^{N}$ covering he compact core of $M$ and $U_{N+1}\cong \mathbb{R}^n\backslash\overline{B_1(0)}$, and then consider a partition of unity $\{\eta_{\alpha}\}_{\alpha=1}^{N+1}$ subordinate to it, and decompose any $\phi\in W^{2,p}_{\delta}(M)$ via $\phi=\sum_{\alpha}\eta_{\alpha}\phi$. Given $u\in L^{p'}_{2-\delta-n}(M)$, it then follows that
\begin{align*}
[u](\Delta_g\phi)&=\sum_{\alpha}[u](\Delta_g\phi_{\alpha})=\sum_{\alpha}\langle u,\Delta_g\phi_{\alpha}\rangle_{L^2(M,dV_g)}=\sum_{\alpha}\langle \sqrt{\mathrm{det}(g)}u,\Delta_g\phi_{\alpha}\rangle_{(U_{\alpha},\delta)},\\
&=\sum_{\alpha}\langle \partial_i\left( \sqrt{\mathrm{det}(g)}g^{ij}\partial_ju\right),\phi_{\alpha})\rangle_{(U_{\alpha},\delta)}=\sum_{\alpha}\langle \sqrt{\mathrm{det}(g)}\Delta_gu,\phi_{\alpha})\rangle_{(U_{\alpha},\delta)},\\
&=(\sqrt{\mathrm{det}(g)}\Delta_gu)(\phi),
\end{align*}
which proves that $\sqrt{\mathrm{det}(g)}\Delta_g\vert_{L^{p'}_{2-\delta-n}(M)}=\Delta^{*}_g\vert_{L^{p'}_{2-\delta-n}(M)}$, which is a bounded map.
\end{proof}

Using the above results, we find that:
\begin{cor}\label{LapFredCoro1}
Let $(M^n,g)$ be a $C^{\infty}_{\tau}$-AE manifold with $\tau<0$. If $1< p< \infty$ and $\delta$ is non-exceptional, then the map $\Delta_g:W^{2,p}_{\delta}(M)\to L^p_{\delta-2}(M)$ is Fredholm.
\end{cor}
\begin{proof}
We know this map is semi-Fredholm, so we need only establish that $\mathrm{Ker}(\Delta_g^{*}:L^{p'}_{2-\delta-n}(M)\to W^{-2,p'}_{-\delta-n}(M))$ is finite dimensional. Proposition \ref{LapBelDualWeighted} guarantees that 
\begin{align*}
u\in \Ker\left(\Delta_g^{*}|_{L^{p'}_{2-\delta-n}(M)}\right) \Longleftrightarrow u\in \Ker\left(\Delta_g|_{L^{p'}_{2-\delta-n}(M)}\right),
\end{align*}
thus, Theorem \ref{BartniksProp1.6} guarantees that $u\in \mathrm{Ker}(\Delta_g:L^{p'}_{2-\delta-n}\to W^{-2,p'}_{-\delta-n})\subset W^{2,p'}_{2-n-\delta}$. Then, Theorem \ref{BartniksSemiFredThm} shows that $\mathrm{Ker}(\Delta_g\vert_{ W^{2,p'}_{2-n-\delta}})=\mathrm{Ker}(\Delta_g\vert_{L^{p'}_{2-\delta-n}})$ is finite dimensional. That is,
\begin{align*}
\mathrm{dim}(\Ker(\Delta_g^{*}|_{L^{p'}_{2-\delta-n}(M)}))=\mathrm{dim}(\Ker(\Delta_g|_{W^{2,p'}_{2-\delta-n}(M)}))<\infty,
\end{align*}
which shows that $\Delta_g:W^{2,p}_{\delta}(M)\to L^p_{\delta-2}(M)$ is Fredholm.
\end{proof}




Finally, by approximation arguments we find:

\begin{cor}\label{WeightedFredholmCoro}
Let $(M^n,g)$ be a $W^{2,q}_{\tau}$-AE manifold with $q>\frac{n}{2}$ and $\tau<0$. If $1< p\leq q$ and $\delta$ is non-exceptional, then the map $\Delta_g:W^{2,p}_{\delta}\to L^p_{\delta-2}$ is Fredholm.
\end{cor}
\begin{proof}
First of all, from Theorem \ref{BartniksSemiFredThm} we know $\Delta_g:W^{2,p}_{\delta}\to L^p_{\delta-2}$ is semi-Fredholm. Furthermore, we can consider a sequence $\{g_j\}_{j=1}^{\infty}$ of $C^{\infty}_{\tau}$-AE metrics such that $g-g_j\xrightarrow[]{W^{2,q}_{\tau}}0$, from which it follows
\begin{align*}
\Vert \Delta_g-\Delta_{g_j}\Vert_{Op(W^{2,p}_{\delta};L^p_{\delta-2})}\xrightarrow[j\rightarrow\infty]{}0.
\end{align*}
Thus, since by Corollary \ref{LapFredCoro1} we know that each $\Delta_{g_j}$ is Fredholm, appealing to the stability of the index of a semi-Fredholm operator \cite[Theorem 19.1.5]{Hormander3}, we know that the index of $\Delta_g$ is finite and thus its cokernel must also be finite dimensional, establishing that that $\Delta_g$ is Fredholm under our hypotheses.

\end{proof}


\section{Appendix II: Compactification and decompactification of rough metrics}\label{AppendixMaxwell}

The main goal of this appendix is to present Lemma 5.2 of \cite{MaxwellDiltsYamabeAE}, but we state this lemma highlighting some details which follow from its proof as it appears in this reference, although this additional information does no appear in the statement of \cite[Lemma 5.2]{MaxwellDiltsYamabeAE}. Since all the extra information we include in the statement is contained in the proof of \cite[Lemma 5.2]{MaxwellDiltsYamabeAE} without any necessary changes, we merely restate this last result as follows and direct the reader to the original reference for the detailed proof.

\begin{lem}\label{LemmaMaxwellDilts}
Let $p>\frac{n}{2}$ and $\tau=\frac{n}{p}-2$ be real numbers, so that $\tau\in (-2,0)$, and suppose that $(M^n,g)$ is a $W^{2,p}_{\tau}$-AE manifold with respect to a structure of infinity with coordinates $\{z^i\}_{i=1}^n$. Then, there is a conformal factor $\phi$, which is equal to $|z|^{2-n}$ in a neighbourhood of infinity, such that the metric $\bar{g}\doteq \phi^{\frac{4}{n-2}}g$ extends to a $W^{2,p}(\bar{M})$ metric on the one-point compactification $\bar{M}$ of $M$, and $\phi\in C^{\infty}(\hat{M}\backslash\{p_{\infty}\})$, where $p_{\infty}\in \hat{M}^n$ denotes the added point at infinity. Moreover, the inversion $x=\frac{z}{|z|^2}$ provides a coordinate system around $p_{\infty}$ with the properties that $x(p_{\infty})=0$ and $\bar{g}(\partial_{x^i},\partial_{x^j})|_{p_{\infty}}=\delta_{ij}$. 

Conversely, suppose that $(\bar{M}^n,\bar{g})$ is a closed manifold with  $\bar{g}\in W^{2,p}$, $p>\frac{n}{2}$ and $p\neq n$. Given a point $p\in \bar{M}$ and a coordinate system $\{x^i\}^n_{i=1}$ around $p$ such that $g(\partial_{x^i},\partial_{x^j})|_{p}=\delta_{ij}$, which in the case $p>n$ we furthermore demand to be a normal-type coordinate system, then there is a conformal factor $\phi\in C^{\infty}(\bar{M}\backslash\{p\})$ such that $\phi=|x|^{2-n}$ in a neighbourhood of $p$ and $(M\doteq \bar{M}\backslash\{p\},g\doteq \phi^{\frac{4}{n-2}}\bar{g})$ is a $W^{2,p}_{\tau}$-AE manifold, with $\tau=\frac{n}{p}-2$, with respect to the structure of infinity given by the inverted coordinates $z^i=\frac{x^i}{|x|^2}$.
\end{lem}

\begin{remark}
In the second part of the lemma, we have stated that, if $p>n$, then one demands the coordinate system $\{x^i\}_{i=1}^n$ to be a ``normal-type" coordinate system. By this we mean that besides $g(\partial_{x^i},\partial_{x^j})|_{p}=\delta_{ij}$, we also ask that $\partial_{x^k}g(\partial_{x^i},\partial_{x^j})|_{p}=0$ for all $i,j,k=1,\cdots,n$. In order to construct such coordinates canonically via the exponential map, one would need local $C^{1,1}$ regularity for the metric to grant uniqueness of solutions to the geodesic equation. While $W^{2,p}(\bar{M})\hookrightarrow C^{1,\alpha}(\bar{M})$ is below this threshold for $p>n$, one may construct such normal-type coordinates directly by hand, for instance following \cite[Proposition 1.25]{AubinBook}.
\end{remark}

The last step in the proof of \cite[Lemma 5.2]{MaxwellDiltsYamabeAE}, which concerns the converse implication in the statement of Lemma \ref{LemmaMaxwellDilts}, appeals to an approximation result for functions $k\in W^{2,p}_0(B)$ which vanish at the origin, where $B\subset \mathbb{R}^n$ denotes a ball around the origin.\footnote{See \cite[Lemma 5.2]{MaxwellDiltsYamabeAE}, just before equation (5.16).} This approximation result, which in \cite[Lemma 5.2]{MaxwellDiltsYamabeAE} is left to the reader, is of independent interest for us since it also allows us to prove Proposition \ref{PropOnePointSupport} below, used in the core of this paper. Thus, we provide a detailed proof of it in the proposition below.
\begin{prop}\label{PropApproximation}
Let $B\subset \mathbb{R}^n$ stand for a ball around the origin and $n\geq 3$. Then:
\begin{enumerate}
\item Given $1<p<n$ and a function $\bar{k}\in W_0^{1,p}(B)$, there is a sequence $\{\bar{k}_j\}_{j=1}^{\infty}\subset C^{\infty}_0(B)$ such that all $\bar{k}_j$ vanish in a neighbourhood of the origin and $\bar{k}_{j}\xrightarrow[j\rightarrow\infty]{W^{1,p}}\bar{k}$;
\item If $p>n$, $\bar{k}\in W_0^{1,p}(B)$ and $\bar{k}|_{0}=0$, there is a sequence $\{\bar{k}_j\}_{j=1}^{\infty}\subset C^{\infty}_0(B)$ such that all $\bar{k}_j$ vanish in a neighbourhood of the origin and $\bar{k}_{j}\xrightarrow[j\rightarrow\infty]{W^{1,p}}\bar{k}$;
\item Given $\frac{n}{2}<p<n$ and a function $\bar{k}\in W_0^{2,p}(B)$ such that $\bar{k}|_{0}=0$, there is a sequence $\{\bar{k}_j\}_{j=1}^{\infty}\subset C^{\infty}_0(B)$ such that all $\bar{k}_j$ vanish in a neighbourhood of the origin and $\bar{k}_{j}\xrightarrow[j\rightarrow\infty]{W^{2,p}}\bar{k}$;
\item Finally, if $p>n$, $\bar{k}\in W_0^{2,p}(B)$, $\bar{k}|_{0}=0$, and $\partial\bar{k}|_{0}=0$, then there is a sequence $\{\bar{k}_j\}_{j=1}^{\infty}\subset C^{\infty}_0(B)$ such that all $\bar{k}_j$ vanish in a neighbourhood of the origin and $\bar{k}_{j}\xrightarrow[j\rightarrow\infty]{W^{2,p}}\bar{k}$.
\end{enumerate}
\end{prop}

\begin{proof}
As a common framework for all the above items, since $\bar{k}\in W^{l,p}_0(B)$ with $l=1,2$, we know there is some sequence $\{\tilde{k}_j\}_{j=1}^{\infty}\subset C^{\infty}_0(B)$ such that 
\begin{align*}
\tilde{k}_j\xrightarrow[j\rightarrow\infty]{W^{l,p}(B)}\bar{k}.
\end{align*}
Let us then consider a cut-off function $\eta\in C^{\infty}_0(B_2(0))$ which is equal to one on $\overline{B_1(0)}$, and then, given $\epsilon>0$, the scaled function $\eta_{\epsilon}(x)\doteq \eta(\epsilon^{-1}x)\subset C^{\infty}_0(B_{2\epsilon}(0))$ which is equal to one on $\overline{B_{\epsilon}(0)}$. We can then consider the cut-off
\begin{align*}
\chi_{\epsilon}(x)\doteq 1- \eta_{\epsilon}(x)=\begin{cases}
0, \text{ if } |x|\leq \epsilon,\\
1, \text{ if } |x|\geq 2\epsilon,
\end{cases}
\end{align*}
and the sequence $\{\bar{k}_j\doteq \chi_{\epsilon_j}\tilde{k}_j\}_{j=1}^{\infty}\subset C^{\infty}_0(B)$ with $\epsilon_j\doteq 2^{-j}$. This sequence has the property that each element vanishes in a neighbourhood of the origin, and we aim to show that it still converges to $\bar{k}$ in $W^{l,p}(B)$ under each restriction on $p$ mentioned above.

With the above in mind, first notice that
\begin{align}\label{CutoffSequence.1}
\Vert\bar{k}_j-\bar{k}\Vert_{L^p(B)}\leq\Vert(\chi_{\epsilon_j}-1)\tilde{k}_j\Vert_{L^p(B)} + \Vert\tilde{k}_j-\bar{k}\Vert_{L^p(B)}=\Vert\eta_{\epsilon_j}\tilde{k}_j\Vert_{L^p(B)} + \Vert\tilde{k}_j-\bar{k}\Vert_{L^p(B)} ,
\end{align}
the second term in the right-hand side of the above expression goes to zero by hypothesis, while for the first term we have that
\begin{align*}
\Vert\eta_{\epsilon_j}\tilde{k}_j\Vert_{L^p(B)}&=\Vert\eta_{\epsilon_j}\tilde{k}_j\Vert_{L^p(B_{2\epsilon_j})}\leq \Vert\tilde{k}_j\Vert_{L^p(B_{2\epsilon_j})} \leq \Vert\tilde{k}_j-\bar{k}\Vert_{L^p(B_{2\epsilon_j})} + \Vert\bar{k}\Vert_{L^p(B_{2\epsilon_j})},\\
&\lesssim \Vert\tilde{k}_j-\bar{k}\Vert_{L^p(B)} + \Vert\bar{k}\Vert_{L^p(B_{2\epsilon_j})} \xrightarrow[j\rightarrow\infty]{} 0.
\end{align*}
Therefore, we obtain that $\tilde{k}_j\xrightarrow[j\rightarrow\infty]{L^p(B)} \bar{k}$. Then, notice that
\begin{align*}
\partial_x\bar{k}_j(x) = \chi_{\epsilon_j}(x)\partial_x\tilde{k}_j(x) + \partial_x\chi_{\epsilon_j}(x)\tilde{k}_j(x)=\chi_{\epsilon_j}(x)\partial_x\tilde{k}_j(x) - \epsilon^{-1}_j\partial_x\eta (\epsilon^{-1}_jx)\tilde{k}_j(x).
\end{align*}
Now, the fact that $\chi_{\epsilon_j}\partial_x\tilde{k}_j\xrightarrow[j\rightarrow\infty]{L^p(B)}\partial_x\bar{k}$ follows along the same lines as in the case of (\ref{CutoffSequence.1}). Thus, we need to show that the second term in the right-hand side of the above expression converges to zero under each of the conditions in the proposition.  
 
We first notice that $\mathrm{supp}(\partial\eta(\epsilon^{-1}_j\cdot))\subset A_{\epsilon_j}\doteq B_{2\epsilon_j}\backslash\overline{B_{\epsilon_j}}$ and is uniformly bounded on it. Let us simplify notations by denoting $u_j\doteq \partial\eta(\epsilon^{-1}_j\cdot)\tilde{k}_j\in W^{1,p}_0(B_{2\epsilon_j})$, and notice that
\begin{align*}
\frac{1}{\epsilon^p_j}\int_B|u_j|^pdx=\frac{1}{\epsilon^p_j}\int_{A_{\epsilon_j}}|u_j|^pdx.
\end{align*}
If $1<p<n$, then we estimate the above appealing to $W^{1,p}_0(B_{2\epsilon_j})\hookrightarrow L^{q}(B_{2\epsilon_j})\hookrightarrow L^{p}(B_{2\epsilon_j})$, for $q=\frac{np}{n-p}$, with explicit estimates
\begin{align*}
\int_{A_{\epsilon_j}}|u_j|^pdx\leq (\mu(B_{2\epsilon_j}))^{1-\frac{p}{q}}\Vert u_j\Vert^p_{L^q(B_{2\epsilon_j})}\lesssim \epsilon_j^{n-\frac{np}{q}}\Vert u_j\Vert^p_{L^q(B_{2\epsilon_j})},
\end{align*}
where $\mu(B_{2\epsilon_j})$ stands for the volume of the ball of radius $2\epsilon_j$. Then, we obtain
\begin{align*}
\frac{1}{\epsilon^p_j}\int_B|u_j|^pdx\leq C(n) \epsilon_j^{n-\frac{np}{q}-p}\Vert u_j\Vert^p_{L^q(B_{2\epsilon_j})}\leq C(n) \Vert u_j\Vert^p_{L^q(B_{2\epsilon_j})},
\end{align*}
where we have used that $\frac{np}{q}=n-p$ and thus $n-\frac{np}{q}-p=0$. Notice then that since $\partial\eta(\epsilon^{-1}_j\cdot)$ are uniformly bounded, then 
\begin{align}\label{CutoffSequence.2}
\begin{split}
\Vert u_j\Vert_{L^q(B_{2\epsilon_j})}&\leq \Vert \partial\eta (\epsilon^{-1}_j\cdot)\Vert_{L^{\infty}(B_{2\epsilon_j})} \Vert \tilde{k}_j\Vert_{L^q(B_{2\epsilon_j})},\\
&\leq \Vert \partial\eta \Vert_{L^{\infty}(B_{2})}\left( \Vert \bar{k}\Vert_{L^q(B_{2\epsilon_j})} + \Vert \bar{k}-\tilde{k}_j\Vert_{L^q(B_{2\epsilon_j})}\right)\xrightarrow[j\rightarrow\infty]{}0
\end{split}
\end{align}
All this shows that $\partial_x\bar{k}_j\xrightarrow[j\rightarrow\infty]{L^p(B)} \partial_x\bar{k}$ when $1<p<n$, and thus establishes item $1$. 

\medskip
If $p>n$, then we have $W^{1,p}_0(B)\hookrightarrow C^{0,\alpha}(B)$ for $\alpha=1-\frac{n}{p}$ \cite[Chapter 4, Theorem 4.12, Part II]{Adams}, and since $\bar{k}|_{0}=0$, this implies that $|\bar{k}|=O(|x|^{\alpha})$, and from $\tilde{k}_j\xrightarrow[]{C^{0,\alpha}}\bar{k}$ we can actually chose the sequence $\{\tilde{k}_j\}$ satisfying $\tilde{k}_j(0)=0$ for all $j$. Thus, 
\begin{align*}
\begin{split}
\frac{1}{\epsilon^p_j}\int_B|u_j|^pdx&=\frac{1}{\epsilon^p_j}\int_{B_{2\epsilon_j}}|u_j|^pdx\leq \frac{1}{\epsilon^p_j}\sup_{x\in B_{2\epsilon_j}}|\partial\eta(\epsilon_j^{-1}x)|^p\int_{B_{2\epsilon_j}}|\tilde{k}_j|^pdx=\frac{1}{\epsilon^p_j}\sup_{y\in B_{2}}|\partial\eta(y)|^p\int_{B_{2\epsilon_j}}|\tilde{k}_j|^pdx.
\end{split}
\end{align*}
To proceed further, 
 we can estimate
\begin{align*}
\!\!\!\!\int_{B_{2\epsilon_j}}|\tilde{k}_j|^p(x)dx\leq  \Big(\sup_{\underset{|x|\neq 0}{x\in B_{2\epsilon_j}}}\frac{|\tilde{k}_j(x)|}{|x|^{1-\frac{n}{p}}}\Big)^{p}\int_{B_{2\epsilon_j}}|x|^{p-n}dx\leq \Big(\sup_{\underset{|x-y|\neq 0}{x,y\in B_{2\epsilon_j}}}\frac{|\tilde{k}_j(x) - \tilde{k}_j(y)|}{|x-y|^{1-\frac{n}{p}}}\Big)^{p} C(n,p)\epsilon^{p}_j,
\end{align*}
where we have used that the chosen sequence $\tilde{k}_j$ satisfies $\tilde{k}_j(0)=0$ for all $j$. Putting the above estimates together, we find a constant $C'(n,p)$, depending only on $n$ and $p$, such that 
\begin{align}\label{CutoffSequence.3}
\begin{split}
\Vert \epsilon_j^{-1}u_j \Vert_{L^p(B)}&\leq C'(n,p) \sup_{\underset{|x-y|\neq 0}{x,y\in B_{2\epsilon_j}}}\frac{|\tilde{k}_j(x) - \tilde{k}_j(y)|}{|x-y|^{1-\frac{n}{p}}}.
\end{split}
\end{align}
In order to further estimate the right-hand side of the above expression we appeal to a rather sharp estimate, which follows from \cite[Chapter 4, Lemma 4.28]{Adams}. Denoting by $Q_{4\epsilon_j}$ the $n$-cube centred at the origin and with sides of size $4\epsilon_j$, examination of the first part of the proof of this lemma guarantees that, given points $(x,y)\in Q_{4\epsilon_j}$ such that $|x-y|<4\epsilon_j$, there is a constant $K=K(n,p)>0$ such that the following estimate holds for all functions $u\in C^{\infty}(Q_{4\epsilon_j})$:
\begin{align}\label{CutoffSequence.3.1.0}
\sup_{\underset{0<|x-y|<4\epsilon_j}{x,y\in Q_{4\epsilon_j}}}\frac{|u(x) - u(y)|}{|x-y|^{1-\frac{n}{p}}}\leq K(n,p)\Vert \nabla u \Vert_{L^{p}(Q_{4\epsilon_j})}.
\end{align}
Noticing then that for $(x,y)\in B_{2\epsilon_j}\subset Q_{4\epsilon_j}$ we always have $|x-y|<4\epsilon_j$, the estimate (\ref{CutoffSequence.3.1.0}) implies:
\begin{align}\label{CutoffSequence.3.1}
\sup_{\underset{|x-y|\neq 0}{x,y\in B_{2\epsilon_j}}}\frac{|u(x) - u(y)|}{|x-y|^{1-\frac{n}{p}}}\leq K(n,p)\Vert \nabla u \Vert_{L^{p}(Q_{4\epsilon_j})}.
\end{align}

Putting together (\ref{CutoffSequence.3}) with (\ref{CutoffSequence.3.1}) applied to $\tilde{k}_j$ we find
\begin{align}\label{CutoffSequence.3.2}
\Vert \epsilon_j^{-1}u_j \Vert_{L^p(B)}&\leq K'(n,p)\Vert \nabla \tilde{k}_j \Vert_{L^{p}(Q_{4\epsilon_j})}\leq K'(n,p)\left( \Vert \nabla \tilde{k}_j - \nabla \tilde{k}\Vert_{L^{p}(B)} + \Vert \nabla \tilde{k} \Vert_{L^{p}(Q_{4\epsilon_j})}\right)\xrightarrow[j\rightarrow\infty]{} 0,
\end{align}
finally establishing
\begin{align}\label{CutoffSequence.4}
\partial_x\bar{k}_j\xrightarrow[j\rightarrow\infty]{L^p(B)} \partial_x\bar{k},
\end{align} 
which proves item $2$.

\medskip
Let us now consider items $3$ and $4$, and thus analyse the convergence of
\begin{align}\label{CutoffSequence.5}
\begin{split}
\partial^2_{x^ix^a}\bar{k}_j(x) &= \chi_{\epsilon_j}(x)\partial^2_{x^ix^a}\tilde{k}_j(x) - \epsilon^{-1}_j(\partial_{x^i}\eta (\epsilon^{-1}_jx)\partial_{x^a}\tilde{k}_j(x) + \partial_{x^a}\eta (\epsilon^{-1}_jx)\partial_{x^i}\tilde{k}_j(x) ) \\
&- \epsilon^{-2}_j\partial^2_{x^ix^a}\eta (\epsilon^{-1}_jx)\tilde{k}_j(x).
\end{split}
\end{align}
First of all, the same arguments as is (\ref{CutoffSequence.1}) show that $\chi_{\epsilon_j}(x)\partial^2_x\tilde{k}_j(x)\xrightarrow[j\rightarrow\infty]{L^p(B)} \partial^2\bar{k}$, and also since $\partial_x\eta (\epsilon^{-1}_jx)\partial_x\tilde{k}_j(x)\in W^{1,p}_0(B_{2\epsilon_j})$, the same arguments as in (\ref{CutoffSequence.2}) (when $\frac{n}{2}<p<n$) and (\ref{CutoffSequence.3.2}) (when $p>n$) can be applied to show that 
\begin{align*}
\epsilon^{-1}_j\partial_x\eta (\epsilon^{-1}_j\cdot)\partial_x\tilde{k}_j\xrightarrow[j\rightarrow\infty]{L^p(B)}  0
\end{align*}
To analyse the last term in (\ref{CutoffSequence.5}), first assume that $\frac{n}{2}<p<n$, so that $W^{2,p}_0(B_{2\epsilon_j})\hookrightarrow W^{1,q}_0(B_{2\epsilon_j})$, with $q\doteq \frac{np}{n-p}>n$. We can then estimate
\begin{align*}
\epsilon^{-2p}_j\int_{B_{2\epsilon_j}}|\partial^2_x\eta (\epsilon^{-1}_jx)\tilde{k}_j(x)|^pdx&\leq \epsilon^{-2p}_j\Vert\partial^2_x\eta (\epsilon^{-1}_j\cdot)\Vert^p_{C^{0}(B_{2\epsilon_j})}\int_{B_{2\epsilon_j}}|\tilde{k}_j|^{p}dx=\epsilon^{-2p}_j\Vert\partial^2\eta\Vert^p_{C^{0}(B_{2})}\int_{B_{2\epsilon_j}}|\tilde{k}_j|^{p}dx,
\end{align*}
where as in the proof of the estimates before (\ref{CutoffSequence.3}), we have
\begin{align*}
\int_{B_{2\epsilon_j}}|\tilde{k}_j|^{p}dx&\leq \Big( \sup_{\underset{|x|\neq 0}{x\in B_{2\epsilon_j}}}\frac{|\tilde{k}_j(x)|}{|x|^{1-\frac{n}{q}}}\Big)^p\int_{B_{2\epsilon_j}}|x|^{p-\frac{np}{q}}dx\leq \Big( \sup_{\underset{|x-y|\neq 0}{x,y\in B_{2\epsilon_j}}}\frac{|\tilde{k}_j(x)-\tilde{k}_j(y)|}{|x-y|^{1-\frac{n}{q}}}\Big)^pC(n,p)\epsilon_j^{p-\frac{np}{q}+n}.
\end{align*}
Putting together the above estimates, we find a constant $C'=C'(n,p)>0$, such that
\begin{align*}
\epsilon_j^{-2}\Vert \tilde{k}_j \Vert_{L^{p}(B_{2\epsilon_j})}\leq C'(n,p) \sup_{\underset{|x-y|\neq 0}{x,y\in B_{2\epsilon_j}}}\frac{|\tilde{k}_j(x)-\tilde{k}_j(y)|}{|x-y|^{1-\frac{n}{q}}}\epsilon_j^{-1-\frac{n}{q}+\frac{n}{p}}=C'(n,p) \sup_{\underset{|x-y|\neq 0}{x,y\in B_{2\epsilon_j}}}\frac{|\tilde{k}_j(x)-\tilde{k}_j(y)|}{|x-y|^{1-\frac{n}{q}}}.
\end{align*}
Using once more the estimate (\ref{CutoffSequence.3.1}), we find
\begin{align*}
\epsilon_j^{-2}\Vert \tilde{k}_j \Vert_{L^{p}(B_{2\epsilon_j})}&\leq K'(n,p)\Vert \nabla\tilde{k}_j\Vert_{L^q(Q_{4\epsilon_j})}\leq  K'(n,p)\left(\Vert \nabla\tilde{k}_j-\nabla\tilde{k}\Vert_{L^q(B)} + \Vert \nabla\tilde{k}\Vert_{L^q(Q_{4\epsilon_j})} \right)\xrightarrow[j\rightarrow\infty]{} 0
\end{align*}
which establishes case $3$.

\medskip
Finally, if $p>n$, we use $\partial_x\bar{k}|_{0}=0$, so that we can take the sequence $\tilde{k}_j$ satisfying $\tilde{k}_j|_{0},\partial_x\tilde{k}_j|_{0}=0$. We furthermore have the optimal embedding $W^{2,p}_0(B_{2\epsilon_j})\hookrightarrow C^{1,\alpha}(B_{2\epsilon_j})$, this time with $\alpha\doteq 1-\frac{n}{p}$, and thus from the mean value inequality
\begin{align*}
|\tilde{k}_j(x)|&\leq |\partial_x\tilde{k}_j(cx)||x|, \text{ for any } x\in B_{2\epsilon_j} \text{ and some } 0\leq c\leq 1,\\
&\leq \Big(\sup_{\underset{|x|\neq 0}{x\in B_{2\epsilon_j}}}\frac{|\partial_x\tilde{k}_j(cx)|}{|x|^{1-\frac{n}{p}}}\Big)|x|^{1+1-\frac{n}{p}}\leq c^{1-\frac{n}{p}} |x|^{2-\frac{n}{p}}\Big(\sup_{\underset{|x|\neq 0}{x\in B_{2\epsilon_j}}}\frac{|\partial_x\tilde{k}_j(cx)|}{|cx|^{1-\frac{n}{p}}}\Big),\\
&\leq C(n,p) |x|^{2-\frac{n}{p}}\sup_{\underset{|x-y|\neq 0}{x,y\in B_{2\epsilon_j}}}\frac{|\partial_x\tilde{k}_j(x) - \partial_x\tilde{k}_j(y)|}{|x-y|^{1-\frac{n}{p}}}
\end{align*}
and therefore
\begin{align*}
\epsilon^{-2p}_j\int_{A_{\epsilon_j}}|\partial^2_x\eta (\epsilon^{-1}_jx)\tilde{k}_j(x)|^pdx&\leq \epsilon^{-2p}_j\sup_{x\in B_{2\epsilon_j}}|\partial^2_x\eta (\epsilon^{-1}_j\cdot)|^p\int_{B_{2\epsilon_j}}|\tilde{k}_j|^{p}dx= \epsilon^{-2p}_j\sup_{y\in B_{2}}|\partial^2\eta (y)|^p\int_{B_{2\epsilon_j}}|\tilde{k}_j|^{p}dx,\\
&\leq C''(n,p) \epsilon^{ -2p}_j\Big(\sup_{\underset{|x-y|\neq 0}{x,y\in B_{2\epsilon_j}}}\frac{|\partial_x\tilde{k}_j(x) - \partial_x\tilde{k}_j(y)|}{|x-y|^{1-\frac{n}{p}}}\Big)^{p}\int_{B_{2\epsilon_j}}|x|^{(2-\frac{n}{p})p}dx,\\
&\leq C'''(n,p) \Big(\sup_{\underset{|x-y|\neq 0}{x,y\in B_{2\epsilon_j}}}\frac{|\partial_x\tilde{k}_j(x) - \partial_x\tilde{k}_j(y)|}{|x-y|^{1-\frac{n}{p}}}\Big)^{p}.
\end{align*}
Using once more (\ref{CutoffSequence.3.1}) applied to $\partial_x\tilde{k}_j\in C^{\infty}(Q_{4\epsilon_j})$, we find
\begin{align*}
\Vert \epsilon^{-2}_j\partial^2_x\eta (\epsilon^{-1}_j\cdot)\tilde{k}_j \Vert_{L^p(B_{2\epsilon_j})}&\leq K'(n,p)\Vert \nabla^2\tilde{k}_j \Vert_{L^p(Q_{4\epsilon_j})}\leq K'(n,p)\left( \Vert \nabla^2\tilde{k}_j - \nabla^2\tilde{k} \Vert_{L^p(B)} + \Vert \nabla^2\tilde{k} \Vert_{L^p(Q_{4\epsilon_j})} \right),
\end{align*}
where the right-hand side of the above expression goes to zero as $j\rightarrow\infty$, since $\tilde{k}_j\xrightarrow[j\rightarrow\infty]{W^{2,p}(B)} \tilde{k}$ by hypothesis, which finally establishes case $4$ in the proposition.
\end{proof}

The above approximation argument can be used to deduce the following result:
\begin{prop}\label{PropOnePointSupport}
Let $B\subset \mathbb{R}^n$ stand for a ball around the origin and $n\geq 3$. Letting $p>\frac{n}{2}$, if $V\in W^{-1,p}(B)$ is supported at the origin, then $V\equiv 0$.
\end{prop}
\begin{proof}
Recalling that $W^{-1,p}(B)=(W^{1,p'}_0(B))'$, under our present conditions we can compute that
\begin{align*}
\frac{1}{p'}=1-\frac{1}{p}>1-\frac{2}{n}=\frac{n-2}{n}\Longleftrightarrow p'<\frac{n}{n-2}.
\end{align*}
Since $\frac{n}{n-2}\leq n\Longleftrightarrow n\geq 3$, we always obtain $p'<n$. Therefore, given any $\phi\in W^{1,p'}_0(B)$, we can use item $1$ of Proposition \ref{PropApproximation} to obtain a sequence $\{\phi_j\}_{j=1}^{\infty}\subset C^{\infty}_0(B)$ such that $\phi_j\xrightarrow[j\rightarrow\infty]{W^{1,p'}} \phi$ and all the elements in the sequence are supported away from the origin. Therefore,
\begin{align*}
V(\phi)=\lim_{j\rightarrow\infty}V(\phi_j)=0,
\end{align*} 
where the last equality follows since $V$ is supported at the origin by hypothesis. Since in the above argument $\phi\in W^{1,p'}_0(B)$ was arbitrary, it then follows that $V\equiv 0$.
\end{proof}

\section*{Statements and Declarations}

Partial financial support was received from the Alexander von Humboldt Foundation. The author has no relevant further financial or non-financial interests to disclose.

\addcontentsline{toc}{section}{References}
\printbibliography

\end{document}